\documentclass[12pt,a4paper]{article}

\usepackage{graphicx}
\usepackage{subfig}
\usepackage{amsmath,amsthm,amssymb}
\usepackage{esint}
\usepackage{mathrsfs}

\usepackage[T1]{fontenc}

\usepackage{hyperref}

\usepackage[left=2.2cm,right=2.2cm,top=3cm,bottom=3.5cm]{geometry}

\usepackage{color}
\usepackage{array}

\usepackage{multirow}

\theoremstyle{plain}
\newtheorem{thm}{Theorem}[section]
\newtheorem{prop}[thm]{Proposition}
\newtheorem{lemma}[thm]{Lemma}
\newtheorem{cor}[thm]{Corollary}
\theoremstyle{definition}
\newtheorem{defn}[thm]{Definition}

\theoremstyle{remark}
\newtheorem{rem}[thm]{Remark}

\numberwithin{equation}{section}

\usepackage{csquotes}
\usepackage[backend=bibtex8, style=numeric-comp, maxnames=5, sorting=nyt]{biblatex}
\usepackage{enumitem}
\renewcommand\labelenumi{(\alph{enumi})}
\renewcommand\theenumi\labelenumi

\newcommand{\R}{\mathbb{R}} 
\newcommand{\N}{\mathbb{N}}
\newcommand{\Z}{\mathbb{Z}} 

\newcommand{\Grad}{\nabla}  
\newcommand{\Div}{{\rm div}\,} 
\newcommand{\Lap}{\Delta} 
\newcommand{\parthree}[3]{\partial_{#1#2#3}^3}

\newcommand{\dx}{\,{\rm d}\vx}
\newcommand{\dt}{\,{\rm d}t}  
\newcommand{\ds}{\,{\rm d}s}  
\newcommand{\dnu}{\,{\rm d}\nu}
\newcommand{\dR}{\,{\rm d}\mathfrak{R}}
\newcommand{\dE}{\,{\rm d}\mathfrak{E}}

\renewcommand{\vec}[1]{{\bf #1}}
\newcommand{\vf}{\vec{f}}
\newcommand{\vk}{\vec{k}} 
\newcommand{\vv}{\vec{v}}
\newcommand{\vw}{\vec{w}}
\newcommand{\vx}{\vec{x}}
\newcommand{\vy}{\vec{y}}
\newcommand{\veta}{\boldsymbol{\eta}}
\newcommand{\vsigma}{\boldsymbol{\sigma}}
\newcommand{\vF}{\vec{F}}
\newcommand{\vu}{\vec{u}} 
\newcommand{\vm}{\vec{m}} 
\newcommand{\vphi}{\boldsymbol{\varphi}} 

\newcommand{\mU}{\mathbb{U}} 
\newcommand{\id}{\mathbb{I}}
\newcommand{\mA}{\mathbb{A}}
\newcommand{\mM}{\mathbb{M}}
\newcommand{\mS}{\mathbb{S}}

\newcommand{\sym}[1]{{\rm Sym}(#1)}
\newcommand{\symz}[1]{{\rm Sym}_0(#1)}

\newcommand{\tr}{{\rm tr}\,} 
\newcommand{\trans}{\top} 

\newcommand{\interior}[1]{{#1}^\circ} 

\newcommand{\T}{\mathbb{T}}
\renewcommand{\rho}{\varrho}
\newcommand{\half}{\tfrac{1}{2}}  
\newcommand{\Cc}{C^\infty_{\rm c}} 
\newcommand{\Cweak}{C_{\rm weak}}
\newcommand{\ov}[1]{\overline{#1}}  
\newcommand{\un}[1]{\underline{#1}}
\newcommand{\ep}{\varepsilon} 
\newcommand{\co}{{\rm co}} 
\newcommand{\diag}{{\rm diag}}
\newcommand{\adj}{{\rm adj}}
 
\renewcommand{\subset}{\subseteq}

\newcommand{\sU}{\mathcal{U}}
\newcommand{\sV}{\mathcal{V}}
\newcommand{\sW}{\mathcal{W}}
\newcommand{\sK}{\mathcal{K}}
\newcommand{\sB}{\mathcal{B}}
\newcommand{\sC}{\mathcal{C}}
\newcommand{\sF}{\mathcal{F}}
\newcommand{\sS}{\mathcal{S}}
\newcommand{\phase}{\mathcal{PH}}
\renewcommand{\setminus}{\smallsetminus}
\newcommand{\opL}{\mathscr{L}}

\newcommand{\ts}{{\rm s}}
\newcommand{\td}{{\rm d}}

\newcommand{\name}[1]{\textsc{#1}} 
\newcommand{\Leray}{\mathbb{P}}

\begin{document} 

\renewcommand{\arraystretch}{1.2}

\title{Non-uniqueness of global-in-time admissible weak solutions to the isentropic compressible Euler equations for a dense set of initial data}

\author{Daniel W.~Boutros\footnote{e-mail: \texttt{dwb42@cam.ac.uk}} \and Simon Markfelder\footnote{corresponding author; e-mail: \texttt{simon.markfelder@uni-konstanz.de}}} 

\date{June 12, 2026}

\maketitle

\bigskip

\centerline{$^{\ast}$ University of Cambridge, Department of Applied Mathematics and Theoretical Physics,}

\centerline{Cambridge CB3 0WA, UK}

\bigskip

\centerline{$^{\dagger}$ Universit\"at Konstanz, Department of Mathematics and Statistics,} 

\centerline{Post office box: 199, 78457 Konstanz, Germany} 

\bigskip

\begin{abstract} 
    In recent years, the technique of convex integration has demonstrated that for some initial data (also referred to as `wild initial data'), many PDE models of mathematical fluid mechanics allow for a multitude of admissible (weak) solutions. This paper is concerned with the question regarding how large the set of such wild initial data is for the isentropic Euler equations. We prove that wild initial data form a dense set in $L^r$ for any $r \in [1,\infty)$. In contrast to existing results in the literature, in this paper `wild initial data' are data which give rise to infinitely many \emph{global}-in-time weak solutions which are \emph{admissible} in the sense that the local energy inequality holds. In other words, the set of initial data with infinitely many admissible weak solutions (independent of the choice of time interval) is dense. A novel part of the construction is that we use a measure-valued (dissipative) solution as the ansatz for the subsolution. This requires several new ideas, in order to ensure the required regularity of the subsolution and to obtain a lower bound for the density. Another crucial ingredient of the proof is that the (local) energy density and the energy flux are constructed as part of the convex integration scheme, in order to obtain solutions which adhere to the local energy inequality.
\end{abstract}

\bigskip


\bigskip



\bigskip

\tableofcontents

\section{Introduction} \label{sec:intro} 

\subsection{The isentropic compressible Euler system} \label{subsec:intro-euler}

We study the \emph{isentropic compressible Euler system}
\begin{align} 
	\partial_t \rho + \Div \vm &= 0, \label{eq:euler-mass} \\
	\partial_t \vm + \Div \left(\frac{\vm \otimes \vm}{\rho}\right) + \Grad p(\rho) &= 0, \label{eq:euler-mom}
\end{align}
with unknown density $\rho$ and momentum $\vm$, both of which are functions of time $t$ and position in space $\vx$. We are interested in solutions on some time interval $t\in [0,T)$ with final time $T>0$. The spatial domain will be the torus $\T^n$, where $n\in\{2,3\}$ is the space dimension. In this paper $\T^n$ is equal to the periodic extension of $[0,1)^n$, in particular, $|\T^n|=1$. Note furthermore, that $\rho$ and $\vm$ take values in $\R^+$ and $\R^n$, respectively. 

\begin{rem} \label{rem:vacuum} 
	The weak solutions (see Defn.~\ref{defn:weak-sol} below) which are constructed in this paper do not contain any vacuum states. So, at this level, we will assume $\rho>0$. We would like to emphasise that this is not a restriction of our main result (see Thm.~\ref{thm:density} below), which concerns the fact that so-called wild initial data are dense. We also refer to Sects.~\ref{subsec:proof-reduction} and \ref{subsec:proof-outline} as well as Rem.~\ref{rem:dws-additional-facts} for the issue of vacuum. 
\end{rem} 

The pressure $p$, which appears in \eqref{eq:euler-mom}, is given by the $\gamma$-law
\begin{equation} \label{eq:isentropic-pressure}
	p(\rho) = \kappa \,\rho^\gamma \qquad \text{ with }\kappa>0, \ 1<\gamma\leq 1+\tfrac{2}{n}. 
\end{equation} 

\begin{rem} \label{rem:gamma-1} 
	The assumption that $\gamma\leq 1+\tfrac{2}{n}$ is essential\footnote{In particular, we will use this assumption when stating the energy inequality \eqref{eq:dws-energy} for dissipative weak solutions, see Rem.~\ref{rem:dws-energy-defect}, as well as in the proof of Lemma~\ref{lemma:estimate-lambda-un}, see Rem.~\ref{rem:gamma-2}.} in this paper. Note, however, that from a physical point of view this is not a severe restriction. Indeed, the theory of gasdynamics determines that $\gamma = 1 + \frac{2}{f}$ where $f\in \N$ is the number of degrees of freedom. Obviously each space dimension yields a translational degree of freedom and hence $f\geq n$. Thus 
	$$
		\gamma = 1 + \tfrac{2}{f} \leq 1 + \tfrac{2}{n}.
	$$
	In particular, the case of a monoatomic gas, where $\gamma=1+\frac{2}{n}$ (i.e.~all degrees of freedom are of translational type), is included in our consideration. 
\end{rem}

We want to solve the initial value problem, whose solutions $(\rho,\vm)$ shall satisfy the initial condition 
\begin{equation} \label{eq:initial}
	(\rho,\vm)(0,\cdot) = (\rho_0,\vm_0),
\end{equation}
with given functions $\rho_0\in L^\infty(\T^n;\R^+)$ and $\vm_0\in L^\infty(\T^n;\R^n)$, in addition to the partial differential equations (PDEs) \eqref{eq:euler-mass}, \eqref{eq:euler-mom}.

\subsection{Admissible weak solutions} \label{subsec:intro-adm}

In this paper we will look for weak solutions of \eqref{eq:euler-mass}-\eqref{eq:initial} which additionally satisfy the energy inequality
\begin{equation} \label{eq:euler-energy}
	\partial_t \left( \frac{|\vm|^2}{2\rho} + P(\rho) \right) + \Div\left[\left(\frac{|\vm|^2}{2\rho} + P(\rho) + p(\rho) \right) \frac{\vm}{\rho} \right] \leq 0 ,
\end{equation}
in the sense of distributions. In \eqref{eq:euler-energy}, $P$ is the \emph{pressure potential} which is -- in the isentropic case \eqref{eq:isentropic-pressure} -- given by 
\begin{equation} \label{eq:pressure-potential}
	P(\rho) = \frac{\kappa}{\gamma-1} \rho^\gamma.
\end{equation}

More precisely, we are interested in \emph{admissible weak solutions} defined as follows.

\begin{defn} \label{defn:weak-sol} 
\begin{enumerate}
	\item A pair\footnote{On $L^\infty\big((0,T)\times \T^n;\R^+\big)$ we do not only require that the essential supremum is less that $\infty$ but also that the essential infimum is strictly positive. In other words for any $\rho\in L^\infty\big((0,T)\times \T^n;\R^+\big)$ we require existence of $0<r<R<\infty$ such that $r\leq \rho(t,\vx) \leq R$ for a.e.~$(t,\vx)\in (0,T)\times \T^n$. This ensures that the terms $\frac{\vm\otimes \vm}{\rho}$ and $\frac{|\vm|^2}{\rho}$ in \eqref{eq:euler-weak-mom} and \eqref{eq:euler-weak-energy} are integrable.} $(\rho,\vm) \in L^\infty\big((0,T) \times \T^n; \R^+ \times \R^n\big)$ is a \emph{weak solution} of the initial value problem \eqref{eq:euler-mass}-\eqref{eq:initial} if the equations are satisfied in the sense of distributions, i.e.
	\begin{align} 
		\int_0^T \int_{\T^n} \Big[\rho \partial_t \phi + \vm\cdot\Grad \phi\Big]\dx\dt + \int_{\T^n} \rho_0\phi(0,\cdot) \dx &= 0 , \label{eq:euler-weak-mass} \\
		\int_0^T \int_{\T^n} \left[\vm \cdot\partial_t \vphi + \frac{\vm\otimes\vm}{\rho}:\Grad \vphi + p(\rho)\Div \vphi\right]\dx\dt + \int_{\T^n} \vm_0\cdot\vphi(0,\cdot) \dx &= 0 \label{eq:euler-weak-mom}
	\end{align}
	for all test functions $(\phi,\vphi) \in \Cc\big([0,T) \times \T^n; \R\times \R^n\big)$.
	 
	\item A weak solution $(\rho,\vm)$ of the initial value problem \eqref{eq:euler-mass}-\eqref{eq:initial} is called \emph{admissible} if the energy inequality \eqref{eq:euler-energy} (with pressure potential \eqref{eq:pressure-potential}) holds in the sense of distribution, i.e.
	\begin{align} 
		\int_0^T \int_{\T^n} \left[\bigg(\frac{|\vm|^2}{2\rho} + P(\rho)\bigg) \partial_t \psi + \bigg(\frac{|\vm|^2}{2\rho} + P(\rho) + p(\rho)\bigg)\frac{\vm}{\rho}\cdot\Grad \psi \right]\dx\dt \qquad & \notag \\ 
		+ \int_{\T^n} \bigg(\frac{|\vm_0|^2}{2\rho_0} + P(\rho_0)\bigg)\psi(0,\cdot) \dx &\geq 0 \label{eq:euler-weak-energy}
	\end{align}
	for all $\psi \in \Cc\big([0,\infty) \times \R^2;\R^+_0\big)$.
\end{enumerate}
\end{defn}

\begin{rem} \label{rem:better-regularity-1}
	In this paper we will solely deal with admissible weak solutions which even satisfy 
	$$
		 \rho \in C^1\big([0,T] \times \T^n; \R^+\big), \qquad  \vm\in \Cweak\big([0,T];L^2(\T^n;\R^n)\big) \cap L^\infty\big((0,T)\times \T^n; \R^n\big),
	$$
	cf.~Rem.~\ref{rem:better-regularity-2} below. In particular, the density is even continuously differentiable and the momentum is weakly continuous in time. 
\end{rem}

\subsection{Non-uniqueness and density of wild initial data} \label{subsec:intro-density} 

Building upon their earlier work on the \emph{incompressible} Euler equation \cite{DelSze09}, \name{De~Lellis}-\name{Sz{\'e}kelyhidi}~\cite{DelSze10} proved that the initial value problem \eqref{eq:euler-mass}-\eqref{eq:initial} surprisingly admits infinitely many admissible weak solution for certain initial data. The technique to construct those solutions is called \emph{convex integration}. Initial data which allow for infinitely many admissible weak solutions are usually termed \emph{wild initial data}. In several subsequent works, one has tackled the natural question on how large the class of wild data is. 

In the incompressible case, \name{Sz{\'e}kelyhidi}-\name{Wiedemann}~\cite{SzeWie12} were able to prove that the class of wild data are $L^2$-dense in the set of all possible $L^2$ initial data, see \cite[Cor.~3]{SzeWie12}. Notably, their solutions only satisfy a global version of the energy inequality (see \eqref{eq:global-energy} below for the analogous global energy inequality in the compressible setting). 

For the compressible Euler system \eqref{eq:euler-mass}, \eqref{eq:euler-mom} a large class of wild data was found by \name{Chiodaroli}~\cite{Chiodaroli14} and also by \name{Feireisl}~\cite{Feireisl14}, who both showed that for a (sufficiently regular) initial density there exists a wild initial momentum. While the solutions constructed in \cite{Chiodaroli14} are only local-in-time, the admissible weak solutions produced in \cite{Feireisl14} even exist globally in time. 

A large focus has also been put on so-called \emph{Riemann initial data} \cite{ChiDelKre15,ChiKre14,ChiKre18,KliMar18_1,Markfelder,BreChiKre18} with the result that Riemann data are wild as soon as the so-called \emph{1-D solution} contains a shock wave. 

Naturally, one has addressed the question whether wild data are dense also in the compressible setting. In this context, \name{Chen}-\name{Vasseur}-\name{Yu}~\cite{CheVasYu21} showed that wild data are indeed dense, however the solutions in \cite{CheVasYu21} only satisfy the global energy inequality 
\begin{equation} \label{eq:global-energy}
	\int_{\T^n} \left( \frac{|\vm(t,\cdot)|^2}{2\rho(t,\cdot)} + P(\rho(t,\cdot)) \right) \dx \leq \int_{\T^n} \left( \frac{|\vm_0|^2}{2\rho_0} + P(\rho_0) \right) \dx 
\end{equation}
rather than the local one \eqref{eq:euler-energy}. Still the solutions constructed in \cite{CheVasYu21} are defined globally in time. 

More recently \name{Chiodaroli}-\name{Feireisl}~\cite{ChiFei24_1} proved another result regarding density of wild data. In contrast to \cite{CheVasYu21}, the solutions found in \cite{ChiFei24_1} do satisfy the local energy inequality \eqref{eq:euler-energy}, i.e.~they are indeed admissible weak solutions in the sense of Defn.~\ref{defn:weak-sol}. The shortcoming of \cite{ChiFei24_1} is that -- in contrast to \cite{CheVasYu21} -- the solutions considered in \cite{ChiFei24_1} only exist locally in time. 

One may summarise the preceding paragraphs as follows. The question on the density of wild data may depend on the exact definition of ``wild initial datum''. In \cite{CheVasYu21} an initial datum is called ``wild'' if there are infinitely many \emph{global}-in-time weak solutions satisfying the \emph{global} energy inequality \eqref{eq:global-energy}. On the other hand, in \cite{ChiFei24_1} an initial datum is ``wild'' if there exist infinitely many \emph{local}-in-time \emph{admissible} weak solutions. In the latter case, the existence time may depend on the initial data itself, see also Rem.~\ref{rem:local-vs-global-in-time} below. 

The aim of the present paper is to prove density of wild data where we call an initial datum ``wild'' if there are infinitely many \emph{global}-in-time \emph{admissible} weak solutions, see Defn.~\ref{defn:wild-data} below. 

Let us finally remark, that a result which is analogous to \cite{ChiFei24_1} has been obtained by \name{Chiodaroli}-\name{Feireisl}~\cite{ChiFei24_2} also in the context of the \emph{full} compressible Euler equations.

\subsection{Outline of this paper} \label{subsec:intro-outline}

The paper is organised as follows. In Sect.~\ref{sec:result} we state our main result Thm.~\ref{thm:density}. In Sect.~\ref{sec:ci} we recall some convex integration results from the literature that are necessary to prove Thm.~\ref{thm:density}. With those results at hand, the proof of Thm.~\ref{thm:density} boils down to finding a suitable subsolution. This subsolution will be constructed from a so-called dissipative weak solution which coincides with the strong solution for small times. Both notions of solution will be recalled in Sect.~\ref{sec:bb}. Sect.~\ref{sec:proof} represents the main part of this paper and is devoted to the proof of Thm.~\ref{thm:density}. In Sect.~\ref{sec:remarks} we compare our result Thm.~\ref{thm:density} with \cite[Thm.~1.3]{ChiFei24_1}. The appendix summarises some further tools. Moreover, in the appendix we extend some two-dimensional results which we cite from the literature to the three-dimensional case.

\section{Statement of the main result} \label{sec:result}

As mentioned in Sect.~\ref{subsec:intro-density} above, the question whether wild data are dense may depend on the definition of ``wild initial datum''. In this paper we use the following definition.

\begin{defn} \label{defn:wild-data}
	We call an initial datum $(\rho_{0},\vm_{0})\in L^\infty(\T^n; \R^+\times \R^n)$ \emph{wild} if for any $T>0$ there exist infinitely many admissible weak solutions 
	\begin{equation} \label{eq:naive-solution-space}
		(\rho,\vm) \in L^\infty\big((0,T) \times \T^n; \R^+ \times \R^n\big).
	\end{equation}
\end{defn}

Now we are ready to state our main result.

\begin{thm} \label{thm:density}
	Let $r\in [1,\infty)$, $(\rho_0,\vm_0)\in L^r(\T^n; \R^+_0\times \R^n)$ and $\ep>0$. Then there exists a wild initial datum $(\rho_{0,\ep},\vm_{0,\ep})\in L^\infty(\T^n; \R^+\times \R^n)$ with 
	\begin{equation} \label{eq:thm-density-close}
		\|\rho_{0,\ep} - \rho_0\|_{L^r(\T^n)} < \ep, \quad \text{ and } \quad \|\vm_{0,\ep} - \vm_0\|_{L^r(\T^n)} <\ep .
	\end{equation}
\end{thm} 

\begin{rem} \label{rem:better-regularity-2} 
	We will prove a stronger version of Thm.~\ref{thm:density}, namely where \eqref{eq:naive-solution-space} in Defn.~\ref{defn:wild-data} is replaced by 
	$$
		\rho \in C^1\big([0,T] \times \T^n; \R^+\big), \qquad  \vm\in \Cweak\big([0,T];L^2(\T^n;\R^n)\big) \cap L^\infty\big((0,T)\times \T^n; \R^n\big),
	$$
	see also Rem.~\ref{rem:better-regularity-1}. In particular, the admissible weak solution which we construct in the proof of Thm.~\ref{thm:density} are weakly continuous in time.
\end{rem}

\begin{rem} \label{rem:local-vs-global-in-time} 
	The difference between our density result (Thm.~\ref{thm:density}) and the one by \name{Chiodaroli}-\name{Feireisl}~\cite[Thm.~1.3]{ChiFei24_1} lies in the definition of ``wild initial datum'', cf.~Sect.~\ref{subsec:intro-density}. While we require existence of infinitely many admissible weak solutions for any $T>0$, \name{Chiodaroli}-\name{Feireisl} only ask for existence of such a time $T>0$, which may (and will) depend on the datum itself. In fact, the set of wild data in the sense of Defn.~\ref{defn:wild-data} is smaller than the corresponding set considered in \cite{ChiFei24_1}. Consequently our result (Thm.~\ref{thm:density}) is indeed stronger than \cite[Thm.~1.3]{ChiFei24_1}. We discuss this in more detail in Sect.~\ref{sec:remarks} below. 
\end{rem}

\begin{rem} 
	In Rem.~\ref{rem:local-vs-global-in-time} above we explained that our result Thm.~\ref{thm:density} is global-in-time. On the other hand, it is notable that wild data also give rise to infinitely many admissible weak solutions on any arbitrarily \emph{short} time interval. So simply speaking, wild initial data are ``wild immediately''. In contrast, there exist initial data (in particular, Lipschitz continuous data as shown in \cite[Cor.~1.2]{ChiDelKre15}, or even smooth data as proved in \cite{CKMS21}) which lead to a unique solution locally in time but admit infinitely many admissible weak solutions on the long run. Despite of that, such initial data are not wild in the sense of Defn.~\ref{defn:wild-data}. 
\end{rem}

\section{Convex integration} \label{sec:ci} 

In order to prove Thm.~\ref{thm:density} we need two convex integration results that have been shown by the authors in \cite{BouMarTit26pre}. The first one (Prop.~\ref{prop:ci-nonuniqueness} below and Thm.~2.15 in \cite{BouMarTit26pre}) yields infinitely many weak solutions if a so-called \emph{subsolution} exists. Those weak solutions are even admissible as soon as the corresponding subsolution attains the initial energy. For naively chosen subsolutions the latter is in general not true, cf.~Rem.~\ref{rem:ci-nonuniqueness} below. To this end, one needs another result from \cite{BouMarTit26pre} (Prop.~\ref{prop:ci-data} below and Thm.~2.29 in \cite{BouMarTit26pre}), which modifies the subsolution as well as the initial data in order to attain the initial energy. 

Both of these results are shown in \cite{BouMarTit26pre} for general differential inclusions under certain assumptions. In order to apply them in the context of the isentropic Euler equations \eqref{eq:euler-mass}, \eqref{eq:euler-mom}, we need to rewrite the latter as a differential inclusion. The literature provides in principle two different ways to do so. The way which is mostly used (see e.g.~\cite{DelSze10,Chiodaroli14,Feireisl14,ChiDelKre15,ChiKre14,ChiKre18,KliMar18_1,CheVasYu21,ChiFei24_1}, just to list a few papers which utilise this ansatz) keeps the density and the kinetic energy fixed. This ansatz turns out to be insufficient for our needs. Another approach is to just fix the density and include the energy inequality into the differential inclusion. This approach was implemented by the second author in \cite{Markfelder24} (see also \name{Gebhard}-\name{Kolumb{\'a}n}~\cite{GebKol22} for a similar ansatz in the context of the incompressible Euler equations). We will use this approach in the current paper, too.

\subsection{Preliminaries} \label{subsec:ci-prel}

As mentioned before, we will make use of the general convex integration framework which was established in \cite{BouMarTit26pre}. We will apply this framework to the isentropic Euler equations \eqref{eq:euler-mass}, \eqref{eq:euler-mom} with energy inequality \eqref{eq:euler-energy} similarly to \cite{Markfelder24}. So for the preliminaries, we may follow \cite[Sect.~3.1]{Markfelder24}. 

First, we replace all non-linearities in the PDEs \eqref{eq:euler-mass}, \eqref{eq:euler-mom} and inequality \eqref{eq:euler-energy} by new unknowns $\mU,q,E$ and $\vF$ to obtain
\begin{align}
	\partial_t \rho + \Div \vm &= 0, \label{eq:eulerlin-mass-prime} \\
	\partial_t \vm + \Div (\mU + q\id) &= 0, \label{eq:eulerlin-mom-prime} \\
	\partial_t E + \Div \vF &\leq 0 . \label{eq:eulerlin-energy-prime}
\end{align} 
Note that \eqref{eq:eulerlin-mass-prime}, \eqref{eq:eulerlin-mom-prime} and \eqref{eq:eulerlin-energy-prime} coincide with \eqref{eq:euler-mass}, \eqref{eq:euler-mom} and \eqref{eq:euler-energy}, respectively, as soon as 
\begin{align} 
	\mU + q\id &= \frac{\vm\otimes \vm}{\rho} + p(\rho)\id , \label{eq:id-Uq}\\
	E&= \frac{|\vm|^2}{2\rho} + P(\rho) , \label{eq:id-E-prime}\\
	\vF&= \left(\frac{|\vm|^2}{2\rho} + P(\rho) + p(\rho)\right) \frac{\vm}{\rho} . \label{eq:id-F-prime}
\end{align}
We require $\mU$ to take values in $\symz{n}$, where the latter denotes the space of symmetric, trace-less $n\times n$-matrices, while the quantity $q$ represents the trace part of the right-hand side of \eqref{eq:id-Uq}. The functions $q$ and $E$ are scalar-valued while $\vF$ is vector-valued. Taking the trace in \eqref{eq:id-Uq}, we find
\begin{equation} \label{eq:trace-id-Uq}
	n q = \frac{|\vm|^2}{\rho} + n p(\rho) .
\end{equation}
Hence, as long as \eqref{eq:id-Uq} holds, \eqref{eq:id-E-prime} and \eqref{eq:id-F-prime} are equivalent to
\begin{align} 
	E&= \tfrac{n}{2} \big(q - p(\rho)\big) + P(\rho) , \label{eq:id-E}\\
	\vF&= \frac{\frac{n}{2} \big(q - p(\rho)\big) + P(\rho) + p(\rho)}{\rho} \vm , \label{eq:id-F} 
\end{align}
respectively.

Like in \cite{Markfelder24} we will treat $\rho$ as a parameter when implementing convex integration, which means that $\rho$ remains fixed while the other unknowns are constructed by repeatedly adding oscillations. Since the right-hand side of \eqref{eq:id-E} is linear in $q$, we may eliminate $E$ in \eqref{eq:eulerlin-energy-prime} via \eqref{eq:id-E}. All in all, we rewrite the system  \eqref{eq:eulerlin-mass-prime}, \eqref{eq:eulerlin-mom-prime}, \eqref{eq:eulerlin-energy-prime} as the following differential inclusion:
\begin{align} 
	\partial_t \rho + \Div \vm &= 0 ,  \label{eq:eulerlin-mass} \\
	\partial_t \vm + \Div (\mU + q\id) &= 0 , \label{eq:eulerlin-mom} \\
	\partial_t \Big(\tfrac{n}{2} \big(q - p(\rho)\big) +  P(\rho)\Big) + \Div \vF &\leq 0, \label{eq:eulerlin-energy} 
\end{align} 
together with the family of constitutive sets
\begin{equation} \label{eq:K}
	\sK_{\rho,Q}:= \left\{ (\vm,\mU,q,\vF)\in \phase \,\Big| \, \text{\eqref{eq:id-Uq} and \eqref{eq:id-F} hold, and }q\leq Q \right\}. 
\end{equation}
Here 
$$
	\phase := \R^n \times \symz{n} \times \R \times \R^n
$$
is an abbreviation for the ``extended'' phase space, i.e.~the space in which a tuple $(\vm,\mU,q,\vF)$ lives. 

The bound $Q$ in the definition of the constitutive sets $\sK_{\rho,Q}$ (see \eqref{eq:K}) is required in order to make them compact. We refer to \cite[Sect.~3]{Markfelder24} for more details. In contrast to \cite{Markfelder24}, where $Q$ is constant, we consider $Q$ to be $(t,\vx)$-dependent, in particular $Q\in C^1\big([0,T]\times \T^n;\R^+\big)$.

\subsection{Existence of infinitely many admissible weak solutions} \label{subsec:ci-nonuniqueness}

When applying the general framework \cite[Thm.~2.15]{BouMarTit26pre} to the differential inclusion \eqref{eq:eulerlin-mass}-\eqref{eq:K}, we obtain the following statement.

\begin{prop} \label{prop:ci-nonuniqueness} 
	Let $T>0$, $n\in\{2,3\}$. Assume there exist $\ov{\rho}\in C^1\big([0,T]\times\T^n;\R^+\big)$, and 
	$$
		(\ov{\vm},\ov{\mU},\ov{q},\ov{\vF})\in \Cweak\big([0,T];L^2(\T^n;\phase)\big) \cap C^1\big((0,T)\times \T^n;\phase\big) 
	$$ 
	with the following properties:
	\begin{itemize} 
		\item The tuple $(\ov{\rho},\ov{\vm},\ov{\mU},\ov{q},\ov{\vF})$ satisfies the partial differential equations and inequalities \eqref{eq:eulerlin-mass}-\eqref{eq:eulerlin-energy} pointwise for all $(t,\vx)\in (0,T)\times\T^n$;
		
		\item There exists $Q\in C^1\big([0,T]\times \T^n;\R^+\big)$ such that 
		\begin{equation}  \label{eq:ci-subs} 
			(\ov{\vm},\ov{\mU},\ov{q},\ov{\vF})(t,\vx) \ \in\ \sU_{\ov{\rho}(t,\vx),Q(t,\vx)} \qquad \text{ for all }(t,\vx)\in (0,T)\times \T^n , 
		\end{equation} 
		where $\sU_{\rho,Q}:= \interior{\big((\sK_{\rho,Q})^\co\big)}$, i.e.~$\sU_{\rho,Q}$ is the interior of convex hull of $\sK_{\rho,Q}$. 
	\end{itemize}
	
	Then there exist infinitely many 
	$$
		\vm\in \Cweak\big([0,T];L^2(\T^n;\R^n)\big) \cap L^\infty\big((0,T)\times \T^n; \R^n\big)
	$$ 
	such that $(\rho\equiv\ov{\rho},\vm)$ solve the isentropic Euler equations in the following sense: 
	\begin{align} 
		\int_0^T\int_{\T^n} \big[\rho \partial_t\phi + \vm\cdot \Grad\phi \big] \dx\dt - \left[\int_{\T^n} \ov{\rho} \phi \dx \right]_{t=0}^{t=T} &= 0 , \label{eq:ci-sol-mass}\\ 
		\int_0^T\int_{\T^n} \bigg[\vm\cdot\partial_t\vphi + \frac{\vm\otimes\vm}{\rho} : \Grad\vphi + p(\rho) \Div\vphi\bigg] \dx\dt - \left[\int_{\T^n} \ov{\vm}\cdot \vphi\dx \right]_{t=0}^{t=T} &= 0 , \label{eq:ci-sol-mom} \\  
		\int_0^T\int_{\T^n} \bigg[\left( \frac{|\vm|^2}{2\rho} + P(\rho)\right) \partial_t\psi + \left( \frac{|\vm|^2}{2\rho} + P(\rho) + p(\rho) \right) \frac{\vm}{\rho} \cdot \Grad\psi \bigg] \dx\dt \ \;\quad \qquad & \notag \\
		- \left[\int_{\T^n} \Big(\tfrac{n}{2}\big(\ov{q} - p(\ov{\rho})\big) + P(\ov{\rho}) \Big) \psi \dx \right]_{t=0}^{t=T} &\geq 0 , \label{eq:ci-sol-energy} 
	\end{align} 
	for all test functions $(\phi,\vphi,\psi)\in \Cc\big([0,T]\times \T^n;\R\times\R^n\times \R_0^+\big)$.
\end{prop}

\begin{rem} \label{rem:ci-nonuniqueness} 
	A tuple $(\ov{\rho},\ov{\vm},\ov{\mU},\ov{q},\ov{\vF})$ which satisfies the assumptions of\footnote{In particular, \eqref{eq:eulerlin-mass}-\eqref{eq:eulerlin-energy} as well as \eqref{eq:ci-subs} hold for all $(t,\vx)\in (0,T)\times \T^n$.} Prop.~\ref{prop:ci-nonuniqueness}, is usually called a \emph{(strict) subsolution}. So, simply speaking, Prop.~\ref{prop:ci-nonuniqueness} states that if there exists a subsolution, then there are infinitely many solutions whose initial (and terminal) condition is given by the subsolution. According to \eqref{eq:ci-sol-mass} and \eqref{eq:ci-sol-mom}, these solutions are indeed weak solutions of the isentropic Euler equations \eqref{eq:euler-mass}, \eqref{eq:euler-mom} with initial data given by $(\ov{\rho},\ov{\vm})(0,\cdot)$. If we can achieve that 
	\begin{equation} \label{eq:0001}
		(\ov{\vm},\ov{\mU},\ov{q},\ov{\vF})(0,\vx) \ \in\ \sK_{\ov{\rho}(0,\vx),Q(0,\vx)} \qquad \text{ for a.e. }\vx\in \T^n ,
	\end{equation}
	then by \eqref{eq:ci-sol-energy} those weak solutions are even admissible, since \eqref{eq:0001} implies 
	$$
		\frac{n}{2}\big(\ov{q}(0,\vx) - p(\ov{\rho}(0,\vx))\big) + P(\ov{\rho}(0,\vx)) = \frac{|\ov{\vm}(0,\vx)|^2}{2\ov{\rho}(0,\vx)} + P(\ov{\rho}(0,\vx)),
	$$
	cf.~\eqref{eq:id-E-prime}, \eqref{eq:id-E}. A general subsolution will however \emph{not} satisfy \eqref{eq:0001}. In this case the weak solutions provided by Prop.~\ref{prop:ci-nonuniqueness} are not admissible because the initial term in \eqref{eq:ci-sol-energy} does not match the initial energy. To overcome this problem, Prop.~\ref{prop:ci-data} below comes into play. Indeed, Prop.~\ref{prop:ci-data} allows to modify a given subsolution such that \eqref{eq:0001} holds for some time $T_0\in (0,T)$ which can then be used as the ``new initial time'', cf.~Rem.~\ref{rem:ci-data} below. 
\end{rem} 

\begin{proof}[Proof of Prop.~\ref{prop:ci-nonuniqueness}]
Prop.~\ref{prop:ci-nonuniqueness} follows immediately from \cite[Thm.~2.15]{BouMarTit26pre} as soon as we have shown that the differential inclusion \eqref{eq:eulerlin-mass}-\eqref{eq:K} satisfies the four ``structural assumptions'' stated in \cite{BouMarTit26pre}. Three of those assumptions were shown by the second author in \cite{Markfelder24} for $n=2$, see Lemmas~3.2, 3.3 and 3.5 therein. Looking at the proof of \cite[Lemmas~3.2 and 3.5]{Markfelder24}, it turns out that those lemmas still hold for $n=3$. A 3-D version of \cite[Lemma~3.3]{Markfelder24} is also true but less obvious. For this reason, we present a detailed proof of a 3-D version of \cite[Lemma~3.3]{Markfelder24} in the appendix, see Lemma~\ref{lemma:suitable-operator-3D}. 

For the fourth assumption we must find a map $D\in C(\R^+\times \R^+ \times \phase;\R)$ with the following properties
\begin{enumerate}
	\item \label{item:suitable-D-concave} For all $(\rho,Q)\in \R^+\times \R^+$, the map $(\vm,\mU,q,\vF)\mapsto D\big((\rho,Q),(\vm,\mU,q,\vF)\big)$ is concave.
	
	\item \label{item:suitable-D-K=>0} For all $(\rho,Q)\in \R^+\times \R^+$ and all $(\vm,\mU,q,\vF)\in \sK_{\rho,Q}$ there holds $D\big((\rho,Q),(\vm,\mU,q,\vF)\big)=0$.
	
	\item \label{item:suitable-D-0=>K} If $D\big((\rho,Q),(\vm,\mU,q,\vF)\big)=0$ and $(\vm,\mU,q,\vF)\in (\sK_{\rho,Q})^\co$, then $(\vm,\mU,q,\vF)\in \sK_{\rho,Q}$.
\end{enumerate}
We claim that 
$$
	D\big((\rho,Q),(\vm,\mU,q,\vF)\big):= n(q-p(\rho)) - \frac{|\vm|^2}{\rho}
$$
has all the required properties. Indeed, property \ref{item:suitable-D-concave} is obvious, while property \ref{item:suitable-D-K=>0} follows from \eqref{eq:trace-id-Uq}. So it remains to show property \ref{item:suitable-D-0=>K}.

If $(\vm,\mU,q,\vF)\in (\sK_{\rho,Q})^\co$, then $q\leq Q$ and 
$$
	\frac{\vm\otimes\vm}{\rho} + p(\rho)\id - \mU - q\id \leq 0 ,
$$
facts which have been proven in detail in \cite[Proof of Lemma~3.2]{Markfelder24}. Hence the matrix $\frac{\vm\otimes\vm}{\rho} + p(\rho)\id - \mU - q\id$ is negative semi-definite and, since $D\big((\rho,Q),(\vm,\mU,q,\vF)\big)=0$, its trace vanishes. This implies that the matrix is already equal to zero, i.e.~\eqref{eq:id-Uq} holds. Next, we deduce from Caratheodory's theorem (see e.g.~\cite[Prop.~B.2]{Markfelder24}) that there exist $N\in \N$ and $\big(\tau_i,(\vm_i,\mU_i,q_i, \vF_i)\big)\in \R^+\times \sK_{\rho,Q}$ (for $i\in \{1,...,N\}$) such that 
$$
	\sum_{i=1}^N \tau_i (\vm_i,\mU_i,q_i, \vF_i) = (\vm,\mU,q,\vF), \qquad \text{ and }\qquad \sum_{i=1}^N \tau_i  = 1.
$$
Hence, we have 
\begin{align}
	\left| \sum_{i=1}^N \tau_i \vm_i\right|^2 &= |\vm|^2 = \rho n (q - p(\rho)) = \rho n \left( \sum_{i=1}^N \tau_i q_i - p(\rho)\right) \notag\\
	&= \sum_{i=1}^N \tau_i \rho n (q_i - p(\rho)) =  \sum_{i=1}^N \tau_i |\vm_i|^2, \label{eq:0002}
\end{align}
where we have used that $(\vm,\mU,q,\vF)$ and all $(\vm_i,\mU_i,q_i, \vF_i)$ (for $i\in \{1,...,N\}$) satisfy \eqref{eq:trace-id-Uq}. From \eqref{eq:0002} and the strict convexity of the map $\vm\mapsto |\vm|^2$, we infer that $\vm_i=\vm$ for all $i\in \{1,...,N\}$. Consequently, there holds
$$
	\vF = \sum_{i=1}^N \tau_i \vF_i = \sum_{i=1}^N \tau_i \frac{\frac{n}{2} \big(q_i - p(\rho)\big) + P(\rho) + p(\rho)}{\rho} \vm_i = \frac{\frac{n}{2} \big(q - p(\rho)\big) + P(\rho) + p(\rho)}{\rho} \vm,
$$ 
i.e.~\eqref{eq:id-F} for all $(\vm_i,\mU_i,q_i, \vF_i)$ (for $i\in \{1,...,N\}$) implies \eqref{eq:id-F} for $(\vm,\mU,q,\vF)$. All in all, we have $(\vm,\mU,q,\vF)\in \sK_{\rho,Q}$ and thus property \ref{item:suitable-D-0=>K} holds as desired.

Now we may apply \cite[Thm.~2.15]{BouMarTit26pre} to the differential inclusion \eqref{eq:eulerlin-mass}-\eqref{eq:K} which finishes the proof. 
\end{proof}

\subsection{Modification of the subsolution and the initial data} \label{subsec:ci-data}

The following proposition results from applying \cite[Thm.~2.29]{BouMarTit26pre} to the differential inclusion \eqref{eq:eulerlin-mass}-\eqref{eq:K}. 

\begin{prop} \label{prop:ci-data}
	Let $T>0$, $n\in \{2,3\}$. Assume there exist $\ov{\rho}\in C^1\big([0,T]\times\T^n;\R^+\big)$, $Q\in C^1\big([0,T]\times \T^n;\R^+\big)$, and 
	$$
		(\ov{\vm},\ov{\mU},\ov{q},\ov{\vF})\in \Cweak\big([0,T];L^2(\T^n;\phase)\big) \cap C^1\big((0,T)\times \T^n;\phase\big) 
	$$ 
	satisfying the assumptions of Prop.~\ref{prop:ci-nonuniqueness}, i.e.~\eqref{eq:eulerlin-mass}-\eqref{eq:eulerlin-energy} as well as \eqref{eq:ci-subs} hold for all $(t,\vx)\in (0,T)\times \T^n$. Let furthermore $T_0\in (0,T)$ and $\beta,\ep>0$. 
	
	Then there exists
	$$
		(\widetilde{\vm},\widetilde{\mU},\widetilde{q},\widetilde{\vF}) \in \Cweak\big([0,T];L^2(\T^n;\phase)\big) \cap C^1\big(((0,T)\setminus\{T_0\})\times \T^n;\phase\big)
	$$ 
	with the following properties: 
	\begin{enumerate} 
		\item \label{item:idiK-pde} The tuple $(\ov{\rho}, \widetilde{\vm},\widetilde{\mU},\widetilde{q},\widetilde{\vF})$ solves the partial differential equations and inequalities \eqref{eq:eulerlin-mass}-\eqref{eq:eulerlin-energy} pointwise for all $t\in (0,T)\setminus\{T_0\}$ and all $\vx \in \T^n$.
		
		\item \label{item:idiK-valuesinUK} The tuple $(\widetilde{\vm},\widetilde{\mU},\widetilde{q},\widetilde{\vF})$ takes values in $\sU_{\ov{\rho}(t,\vx),Q(t,\vx)}$ if $t\neq T_0$ and in $\sK_{\ov{\rho}(t,\vx),Q(t,\vx)}$ for $t=T_0$, i.e. 
		\begin{align*}
			(\widetilde{\vm},\widetilde{\mU},\widetilde{q},\widetilde{\vF})(t,\vx) &\in \sU_{\ov{\rho}(t,\vx),Q(t,\vx)}\qquad \text{ for all }t\in (0,T)\setminus\{T_0\} \text{ and all }\vx\in \T^n , \\
			(\widetilde{\vm},\widetilde{\mU},\widetilde{q},\widetilde{\vF})(T_0,\vx) &\in \sK_{\ov{\rho}(T_0,\vx),Q(T_0,\vx)}\qquad \text{ for a.e. }\vx\in \T^n . 
		\end{align*} 
		
		\item \label{item:idiK-outside} Outside the temporal interval $(T_0-\beta,T_0+\beta)$, the tuple $(\widetilde{\vm},\widetilde{\mU},\widetilde{q},\widetilde{\vF})$ coincides with $(\ov{\vm},\ov{\mU},\ov{q},\ov{\vF})$, i.e. 
		\begin{equation*} 
		\begin{split}
			(\widetilde{\vm},\widetilde{\mU},\widetilde{q},\widetilde{\vF})(t,\vx) &= (\ov{\vm},\ov{\mU},\ov{q},\ov{\vF})(t,\vx) \\
			&\qquad \text{ for all }t\in (0,T)\setminus(T_0-\beta,T_0+\beta) \text{ and all }\vx\in \T^n.
		\end{split} 
		\end{equation*} 
		
		\item \label{item:idiK-m-close} The $L^2$ distance between $\widetilde{\vm}(T_0,\cdot)$ and $\ov{\vm}(T_0,\cdot)$ can be estimated as follows
		\begin{equation*} 
			\Big\|\widetilde{\vm}(T_0,\cdot) - \ov{\vm}(T_0,\cdot)\Big\|_{L^2(\T^n)}^2 \leq \big\|\widetilde{\vm}(T_0,\cdot)\big\|_{L^2(\T^n)}^2 - \big\|\ov{\vm}(T_0,\cdot)\big\|_{L^2(\T^n)}^2 + \ep.
		\end{equation*} 
	\end{enumerate}
\end{prop} 

\begin{rem} \label{rem:ci-data} 
	As mentioned in Rem.~\ref{rem:ci-nonuniqueness} above, one needs to find a strict subsolution $(\ov{\vm},\ov{\mU},\ov{q},\ov{\vF})$ satisfying \eqref{eq:0001} in order to obtain infinitely many \emph{admissible} weak solutions of the Euler equations \eqref{eq:euler-mass}, \eqref{eq:euler-mom} from Prop.~\ref{prop:ci-nonuniqueness}. Having found a subsolution $(\ov{\vm},\ov{\mU},\ov{q},\ov{\vF})$ which does not necessarily fulfill \eqref{eq:0001}, Prop.~\ref{prop:ci-data} yields another subsolution $(\widetilde{\vm},\widetilde{\mU},\widetilde{q},\widetilde{\vF})$ which does satisfy \eqref{eq:0001} at time $t=T_0$, see item~\ref{item:idiK-valuesinUK} in Prop.~\ref{prop:ci-data}. So we may use $T_0$ as the ``new initial time'' and $(\ov{\rho},\widetilde{\vm})(T_0,\cdot)$ as the ``new initial data'' for which Prop.~\ref{prop:ci-nonuniqueness} yields infinitely many admissible weak solutions. Notice furthermore that in the proof of Thm.~\ref{thm:density}, item~\ref{item:idiK-m-close} of Prop.~\ref{prop:ci-data} will allow to estimate the distance between the ``old'' initial momentum $\ov{\vm}(0,\cdot)$ and the ``new'' initial momentum $\widetilde{\vm}(T_0,\cdot)$, cf.~Sect.~\ref{subsubsec:moddata-s6} below. 
\end{rem}

\begin{proof}[Proof of Prop.~\ref{prop:ci-data}] 
The ``structural assumptions'' shown in the proof of Prop.~\ref{prop:ci-nonuniqueness} allow to apply \cite[Thm.~2.29]{BouMarTit26pre} to the differential inclusion \eqref{eq:eulerlin-mass}-\eqref{eq:K}. This immediately yields a tuple 
$$
	(\widetilde{\vm},\widetilde{\mU},\widetilde{q},\widetilde{\vF}) \in \Cweak\big([0,T];L^2(\T^n;\phase)\big) \cap C^1\big(((0,T)\setminus\{T_0\})\times \T^n;\phase\big)
$$ 
which satisfies \ref{item:idiK-pde}-\ref{item:idiK-outside}. Looking into the proof of \cite[Thm.~2.29]{BouMarTit26pre}, we observe that it is possible to achieve that 
$$
	d\big( (\ov{\vm},\ov{\mU},\ov{q},\ov{\vF}) , (\widetilde{\vm},\widetilde{\mU},\widetilde{q},\widetilde{\vF})\big) 
$$
is arbitrarily small, where $d(\cdot,\cdot)$ denotes the metric which induces the topology on the space $\Cweak\big([0,T];L^2(\T^n;\phase)\big)$, see \cite[Prop.~2.18]{BouMarTit26pre}. This allows to generate a sequence of 
$$
	(\widetilde{\vm}_i,\widetilde{\mU}_i,\widetilde{q}_i,\widetilde{\vF}_i)_{i\in \N} \subset \Cweak\big([0,T];L^2(\T^n;\phase)\big) \cap C^1\big(((0,T)\setminus\{T_0\})\times \T^n;\phase\big)
$$ 
with properties \ref{item:idiK-pde}-\ref{item:idiK-outside} and such that $\widetilde{\vm}_i(T_0,\cdot) \rightharpoonup \ov{\vm}(T_0,\cdot)$ in $L^2(\T^n)$ as $i\to \infty$. Thus, we find $j\in\N$ such that 
$$
	\left| \int_{\T^n} \widetilde{\vm}_j(T_0,\cdot) \cdot \ov{\vm}(T_0,\cdot) \dx - \big\|\ov{\vm}(T_0,\cdot)\big\|_{L^2(\T^n)}^2 \right| = \left| \int_{\T^n} \big(\widetilde{\vm}_j(T_0,\cdot) - \ov{\vm}(T_0,\cdot)\big) \cdot \ov{\vm}(T_0,\cdot) \dx \right| \leq \half\ep.
$$ 
Consequently, $(\widetilde{\vm},\widetilde{\mU},\widetilde{q},\widetilde{\vF}):= (\widetilde{\vm}_j,\widetilde{\mU}_j,\widetilde{q}_j,\widetilde{\vF}_j)$ has the additional property \ref{item:idiK-m-close} because
\begin{align*}
	&\Big\|\widetilde{\vm}(T_0,\cdot) - \ov{\vm}(T_0,\cdot)\Big\|_{L^2(\T^n)}^2 \\
	&= \big\|\widetilde{\vm}(T_0,\cdot)\big\|_{L^2(\T^n)}^2 + \big\|\ov{\vm}(T_0,\cdot)\big\|_{L^2(\T^n)}^2 - 2 \int_{\T^n} \widetilde{\vm}_j(T_0,\cdot) \cdot \ov{\vm}(T_0,\cdot) \dx \\
	&\leq \big\|\widetilde{\vm}(T_0,\cdot)\big\|_{L^2(\T^n)}^2 - \big\|\ov{\vm}(T_0,\cdot)\big\|_{L^2(\T^n)}^2 + \ep.
\end{align*}
\end{proof}

\subsection{Properties of $\sU_{\rho,Q}$} \label{subsec:ci-U}

In order to apply Prop.~\ref{prop:ci-nonuniqueness} or Prop.~\ref{prop:ci-data}, we need to find a subsolution, i.e.~a tuple $(\ov{\rho},\ov{\vm},\ov{\mU},\ov{q},\ov{\vF})$ which in particular satisfies \eqref{eq:ci-subs} for all $(t,\vx)\in (0,T)\times \T^n$. To this end, we need a criterion which allows us to decide whether or not a given tuple $(\vm,\mU,q,\vF)\in \phase$ lies in $\sU_{\rho,Q}$. Such a criterion is provided in Prop.~\ref{prop:U} below, which follows for $n=2$ from the second author's more detailed characterisation of $\sU_{\rho,Q}$ given in \cite[Sect.~3.4]{Markfelder24}. For $n=3$ we will provide an analogous characterisation in Appendix~\ref{sec:app-U}. 

Recall that 
$$
	\sK_{\rho,Q}:= \left\{ (\vm,\mU,q,\vF)\in \phase \,\Big| \, \text{\eqref{eq:cond-K-1}-\eqref{eq:cond-K-3} hold}\right\}, 
$$
where \eqref{eq:cond-K-1}-\eqref{eq:cond-K-3} read
\begin{align}
	\frac{\vm\otimes \vm}{\rho} + p(\rho)\id - \mU - q\id &= 0, \label{eq:cond-K-1} \\
	\vF - \frac{\frac{n}{2} \big(q - p(\rho)\big) + P(\rho) + p(\rho)}{\rho} \vm &= 0 , \label{eq:cond-K-2} \\
	q &\leq Q. \label{eq:cond-K-3}
\end{align}
Moreover, $\sU_{\rho,Q}= \interior{\big((\sK_{\rho,Q})^\co\big)}$. 

An explicit description of $\sU_{\rho,Q}$ should contain relaxations of the conditions \eqref{eq:cond-K-1}-\eqref{eq:cond-K-3}. By convexity, linearity and continuity, it is plausible that the sought relaxations of \eqref{eq:cond-K-1} and \eqref{eq:cond-K-3} read
$$
	\frac{\vm\otimes\vm}{\rho} + p(\rho)\id - \mU - q\id < 0 \qquad \text{ and } \qquad q<Q ,
$$
respectively, cf.~\cite[Lemma~4.3.6]{Markfelder} and references therein. As the left-hand side of \eqref{eq:cond-K-2} is neither convex nor concave in $(\vm,q)$, it is not at all obvious how \eqref{eq:cond-K-2} is relaxed. However, it turns out that it is not necessary to know an explicit description of the relaxation of \eqref{eq:cond-K-2} as long as we are allowed to play with the bound $Q$. In fact, one can even get rid of any condition on $\vF$ as shown by the following proposition.

\begin{prop} \label{prop:U} 
	Let $n\in \{2,3\}$, and $\sC \subset \R^+ \times \R^n \times \symz{n} \times \R$ be compact with the property that 
	\begin{equation} \label{eq:matrix-neg-def}
		\frac{\vm\otimes\vm}{\rho} + p(\rho)\id - \mU - q\id < 0 \qquad \text{ for any } (\rho,\vm,\mU,q) \in \sC. 
	\end{equation}
	Then the following assertions hold.
	\begin{enumerate}
		\item \label{item:U-1} For any $\sF\subset \R^n$ compact, there exists a constant $Q>0$ (depending on $\sC$ and $\sF$) such that
		$$
			(\vm,\mU,q,\vF) \in \sU_{\rho,Q} \qquad \text{ for all }\big((\rho,\vm,\mU,q),\vF\big) \in \sC\times \sF.
		$$
		
		\item \label{item:U-2} For any $\ep>0$, there exists $\delta>0$ (depending on $\ep$ and $\sC$) such that 
		$$
			(\vm,\mU,q,\vF) \in \sU_{\rho,q+\ep} \qquad \text{ for all }(\rho,\vm,\mU,q,\vF) \in \sB_\delta,
		$$ 
		where 
		$$
			\sB_\delta := \left\{ \big((\rho,\vm,\mU,q),\vF\big) \in \sC\times \R^n \, \Big|\,  \left| \vF -  \frac{\frac{n}{2} \big(q - p(\rho)\big) + P(\rho) + p(\rho)}{\rho} \vm \right| < \delta \right\}.
		$$
	\end{enumerate}
\end{prop} 

\begin{rem} 
	Prop.~\ref{prop:U}~\ref{item:U-1} shows that the condition \eqref{eq:matrix-neg-def} is already sufficient for a point $(\vm,\mU,q,\vF)$ to lie in $\sU_{\rho,Q}$ as long as the bound $Q$ is chosen very large. In part~\ref{item:U-2} the statement is tightened with regard to the size of the bound $Q$: if condition \eqref{eq:cond-K-2} is only mildly violated, then a small bound $Q=q+\ep$ is sufficient.
\end{rem}

\begin{proof}[Proof of Prop.~\ref{prop:U}]
	Let us first consider the case $n=2$. We recall the following notation from \cite{Markfelder24}. We introduce the four vectors $\vsigma^j\in \R^2$:
	\begin{align*}
		\vsigma^1 &:= \left(\begin{array}{r} 1 \\ 1 \end{array} \right), & \vsigma^2 &:= \left(\begin{array}{r} -1 \\ -1 \end{array} \right), & \vsigma^3 &:= \left(\begin{array}{r} 1 \\ -1 \end{array} \right), & \vsigma^4 &:= \left(\begin{array}{r} -1 \\ 1 \end{array} \right). 
	\end{align*}
	Moreover, for $j=1,2,3,4$, $(\rho,\vm,\mU,q)\in \sC$ and $Q>q$, we define
	$$
		A_{\rho,Q}^j(\vm,\mU,q) := \frac{\det\left(\frac{\vm\otimes \vm}{\rho} - \mU + (p(\rho) - q)\id \right)}{- ([\vsigma^j]_2 , -[\vsigma^j]_1) \cdot \left(\frac{\vm\otimes \vm}{\rho} - \mU + (p(\rho) - q)\id \right)\cdot \left( \begin{array}{r} [\vsigma^j]_2 \\ -[\vsigma^j]_1 \end{array}\right)},
	$$
	as well as 
	$$
		r_{\rho,Q}^j(\vm,\mU,q) := \frac{1}{2 (Q-q)} \left( - \vm \cdot \vsigma^j + \sqrt{(\vm \cdot \vsigma^j)^2 + 4\rho A_{\rho,Q}^j(\vm,\mU,q) + 4\rho (Q-q)} \right) .
	$$
	As explained in detail in \cite[Sect.~3.4.1]{Markfelder24}, there holds 
	$$
		A_{\rho,Q}^j(\vm,\mU,q) > 0 \quad \text{ and } \quad r_{\rho,Q}^j(\vm,\mU,q) > 0 \qquad \text{ for all } (\rho,\vm,\mU,q)\in \sC, \text{ and all } Q>q. 
	$$
	Hence, we may set 
	$$
		\vf_{\rho,Q}^j(\vm,\mU,q) := \frac{A_{\rho,Q}^j(\vm,\mU,q)}{r_{\rho,Q}^j(\vm,\mU,q)} \vsigma^j, 
	$$
	again for $j=1,2,3,4$, $(\rho,\vm,\mU,q)\in \sC$ and $Q>q$. Finally, the sets $\sW_{\rho,Q}$ are defined as
	\begin{align} 
		\sW_{\rho,Q} := \Bigg\{ (\vm,\mU,q,\vF) \in \phase\,\Big|\, & \bullet \ \frac{\vm\otimes\vm}{\rho} + p(\rho)\id - \mU -q \id< 0, \label{eq:WrhoQ-2d} \\
		& \bullet \ \exists\,\kappa_1,\kappa_2,\kappa_3,\kappa_4\in \R^+\text{ with } \ \sum_{j=1}^4\kappa_j =1\ \text{ such that } \notag \\
		&\quad\qquad \vF-\frac{q+P(\rho)}{\rho}\vm = \sum_{j=1}^4\kappa_j \vf^j_{\rho,Q}(\vm,\mU,q), \notag \\
		& \bullet \ q< Q \quad \Bigg\} . \notag 
	\end{align}
	From \cite[Prop.~3.9]{Markfelder24} we know that $\sW_{\rho,Q}\subset \sU_{\rho,Q}$. 
	
	The proof of Prop.~\ref{prop:U} is now straightforward. Indeed by compactness of $\sC$, we infer that there exist constants $c_1(\sC),c_2(\sC)>0$ such that 
	\begin{equation} \label{eq:est-fj}
		\Big| \vf_{\rho,Q}^j(\vm,\mU,q) \Big| \geq \frac{1}{\frac{c_1(\sC)}{Q-q} + \frac{c_2(\sC)}{\sqrt{Q-q}}} \qquad \text{ for all } (\rho,\vm,\mU,q)\in \sC, \text{ and all } Q>q. 
	\end{equation}
	Hence, in order to show item \ref{item:U-1} of Prop.~\ref{prop:U}, we may take $Q$ large enough, so that 
	$$
		\vF-\frac{q+P(\rho)}{\rho}\vm
	$$
	is contained in the convex polytope spanned by the four points $\vf_{\rho,Q}^j(\vm,\mU,q)$ ($j=1,2,3,4$) for any $\big((\rho,\vm,\mU,q),\vF\big) \in \sC\times \sF$ as both $\sC$ and $\sF$ are compact. Thus, $(\vm,\mU,q,\vF)\in \sW_{\rho,Q}\subset \sU_{\rho,Q}$ for any $\big((\rho,\vm,\mU,q),\vF\big) \in \sC\times \sF$ as desired.
	
	Regarding item \ref{item:U-2} of Prop.~\ref{prop:U}, we observe that estimate \eqref{eq:est-fj} yields 
	$$
		\Big| \vf_{\rho,q+\ep}^j(\vm,\mU,q) \Big| \geq \frac{1}{\frac{c_1(\sC)}{\ep} + \frac{c_2(\sC)}{\sqrt{\ep}}} \qquad \text{ for all } (\rho,\vm,\mU,q)\in \sC, \text{ and all } \ep>0. 
	$$
	Choosing 
	$$
		\delta:=\frac{1}{\frac{c_1(\sC)\sqrt{2}}{\ep} + \frac{c_2(\sC)\sqrt{2}}{\sqrt{\ep}}},
	$$
	we find that 
	$$
		\vF-\frac{q+P(\rho)}{\rho}\vm
	$$
	is contained in the convex polytope spanned by the four points $\vf_{\rho,q+\ep}^j(\vm,\mU,q)$ ($j=1,2,3,4$) for any $\big((\rho,\vm,\mU,q),\vF\big) \in \sB_\delta$. Consequently, $(\vm,\mU,q,\vF)\in \sW_{\rho,q+\ep}\subset \sU_{\rho,q+\ep}$ for any $(\rho,\vm,\mU,q,\vF) \in \sB_\delta$ as desired.
	
	In the case $n=3$, we may proceed in a similar fashion once we have proved a three-dimensional version of \cite[Prop.~3.9]{Markfelder24}. We will carry out the latter in the appendix, see Prop.~\ref{prop:WsubsetU-3D}. 
\end{proof}

\section{Building blocks for the subsolution} \label{sec:bb}

In view of Props.~\ref{prop:ci-nonuniqueness} and \ref{prop:ci-data}, we will need to construct a suitable subsolution in order to prove Thm.~\ref{thm:density}. This subsolution will be build from a so-called \emph{dissipative weak solution}, a notion which has been introduced by \name{Feireisl~et al.}, see e.g.~\cite[Sect.~5.2.1]{FLMS}. We recall its definition in Sect.~\ref{subsec:bb-dws} below. On a short time interval, the dissipative weak solutions which appear in the proof of Thm.~\ref{thm:density} will coincide with a strong solution, whose local-in-time existence we recall in Sect.~\ref{subsec:bb-strong} below.

\subsection{Strong solutions} \label{subsec:bb-strong}

Recall that strong solutions, i.e.~$(\rho,\vm)\in C^1\big([0,T]\times\T^n; \R^+ \times \R^n\big)$ which satisfy the equations \eqref{eq:euler-mass}, \eqref{eq:euler-mom} pointwise in the classical sense, exist for small $T$ and sufficiently regular initial data $(\rho_0,\vm_0)$:

\begin{prop} \label{prop:ex-strong}
	For any initial data $(\rho_0,\vm_0)\in H^s(\T^n;\R^+ \times \R^n)$ with $s> \frac{n}{2}+1$, there exists $T>0$ (depending on $\|(\rho_0,\vm_0)\|_{H^s(\T^n)}$) and a unique strong solution 
	$$
		(\rho,\vm)\in C^0\big([-T,T];H^s(\T^n;\R^+\times \R^n)\big) \cap C^1\big([-T,T];H^{s-1}(\T^n;\R^+\times \R^n)\big)
	$$
	of the initial value problem \eqref{eq:euler-mass}-\eqref{eq:initial}.
\end{prop}

Note that by Sobolev embedding, the initial data as well as the solution are continuously differentiable, i.e.~$(\rho_0,\vm_0)\in C^1(\T^n;\R^+ \times \R^n)$ and $(\rho,\vm)\in C^1\big([-T,T]\times \T^n;\R^+\times \R^n\big)$. 

Prop.~\ref{prop:ex-strong} was originally proven by \name{Kato}~\cite{Kato75}. We also refer to standard textbooks, e.g.~\cite[Sect.~2.1 together with Sect.~1.3]{Majda} or \cite[Sect.~10.1 together with Sect.~13.2.3]{BenSer} for a detailed proof. 

Notice that usually in the literature the local-in-time existence result is only stated on the non-negative time interval $[0,T]$. Obviously, it is also possible to solve the equations backwards in time which yields solutions on $[-T,T]$ as stated in Prop.~\ref{prop:ex-strong}.

Let us finally note that strong solutions automatically satisfy the energy inequality \eqref{eq:euler-energy} with equality.

\subsection{Dissipative weak solutions} \label{subsec:bb-dws}

We first introduce some notation following \cite{FLMS}.

Let $\mathcal{Q}\subset \R^m$ compact. 
\begin{itemize}
	\item $\mathcal{M}(\mathcal{Q})$ denotes the set of all signed Radon measures. Note that $\mathcal{M}(\mathcal{Q})$ can be identified as the dual of $C(\mathcal{Q})$.
	
	\item For an $n$-dimensional vector space $E$, $\mathcal{M}(\mathcal{Q};E)$ denotes the set of vector valued measures, i.e. 
	$$
		\mathcal{M}(\mathcal{Q};E) := \left\{ (\nu_1,..., \nu_n) \, \Big|\, \nu_i \in \mathcal{M}(\mathcal{Q}) \text{ for any }i=1,...,n\right\}.
	$$
	
	\item $\mathcal{M}^+(\mathcal{Q})$ denotes the set of non-negative Radon measures. 
	
	\item $\mathcal{M}^+(\mathcal{Q};\sym{n})$ denotes the set of positive semidefinite matrix valued measures, i.e. 
	\begin{align*}
		\mathcal{M}^+(\mathcal{Q};\sym{n}) &= \bigg\{ \nu \in \mathcal{M}(\mathcal{Q};\sym{n})\,\Big|\, \\
		&\qquad\xi\otimes \xi:\int_{\mathcal{Q}} \phi \dnu \geq 0 \text{ for all }\xi\in \R^n \text{ and all }\phi\in C(\mathcal{Q})\text{ with } \phi\geq 0\bigg\}. 
	\end{align*}
\end{itemize}
Moreover $1_{\rho>0}$ denotes the indicator function, more precisely
$$
	1_{\rho>0}(t,\vx) = \left\{ \begin{array}{rl} 1 , & \text{ if }\rho(t,\vx)>0, \\ 0, &\text{ else}. \end{array} \right. 
$$

Now we are ready to recall the definition of a \emph{dissipative weak solution} from \cite{FLMS}. 

\begin{defn}[See {\cite[Defn.~5.6]{FLMS}}] \label{defn:dws}
	A pair 
	$$
		\rho\in \Cweak\big([0,T];L^\gamma(\T^n;\R^+_0)\big), \qquad \vm\in \Cweak\big([0,T];L^{\frac{2\gamma}{\gamma+1}}(\T^n;\R^n)\big)
	$$ 
	is called a \emph{dissipative weak solution} of the isentropic compressible Euler equations \eqref{eq:euler-mass}, \eqref{eq:euler-mom} (with pressure \eqref{eq:isentropic-pressure}) with initial data $\rho_0\in L^\gamma(\T^n;\R^+)$, $\vm_0\in L^{\frac{2\gamma}{\gamma+1}}(\T^n;\R^n)$ if there exists a Reynolds defect 
	\begin{equation} \label{eq:space-fracR}
		\mathfrak{R} \in L^\infty \big((0,T);\mathcal{M}^+ (\T^n; \sym{n})\big)
	\end{equation}
	such that the following assertions hold:
	\begin{itemize}
		\item The equations are satisfied in the following form
		\begin{align}
			\int_0^\tau\int_{\T^n} \Big[ \rho \partial_t \phi + \vm \cdot \Grad \phi\Big] \dx\dt &= \left[ \int_{\T^n} \rho \phi \dx \right]_{t=0}^{t=\tau}, \label{eq:dws-mass} \\
			\int_0^\tau\int_{\T^n} \Big[ \vm \cdot \partial_t \vphi + 1_{\rho>0} \frac{\vm\otimes \vm}{\rho} : \Grad \vphi + p(\rho) \Div \vphi \Big] \dx\dt \quad & \notag \\
			+ \int_0^\tau\int_{\T^n} \Grad\vphi : \dR(t) \dt &= \left[ \int_{\T^n} \vm\cdot \vphi \dx \right]_{t=0}^{t=\tau} \label{eq:dws-mom}
		\end{align}
		for all $\tau\in [0,T]$ and all test functions $\phi\in W^{1,\infty}\big((0,T)\times \T^n\big)$ and $\vphi\in C^M\big([0,T]\times \T^n; \R^n\big)$, $M\geq 1$. 
			
		\item The energy inequality holds in the following form 
		\begin{equation} \label{eq:dws-energy}
			\int_{\T^n} \left[ 1_{\rho>0} \frac{|\vm|^2}{2\rho} + P(\rho) \right](\tau, \cdot) \dx + \frac{1}{2} \int_{\T^n} \,{\rm d}(\tr \mathfrak{R})(\tau) \leq \int_{\T^n} \left[ \frac{|\vm_0|^2}{2\rho_0} + P(\rho_0) \right] \dx
		\end{equation}
		for almost all $\tau\in [0,T]$. 
	\end{itemize}
\end{defn}

\begin{rem} \label{rem:dws-energy-defect} 
	In \cite[Defn.~5.6]{FLMS} an energy defect $\mathfrak{E}$ is considered, and consequently the term $\frac{1}{2} \int_{\T^n} \,{\rm d}(\tr \mathfrak{R})(\tau)$ in the energy inequality \eqref{eq:dws-energy} is replaced by $\int_{\T^n} \dE(\tau)$. Furthermore, it is required in \cite[Defn.~5.6]{FLMS} that the defects are compatible in the sense that 
	$$
		\un{d} \mathfrak{E} \leq \tr \mathfrak{R} \leq \ov{d} \mathfrak{E} \qquad \text{ for some constants } 0 \leq \un{d} \leq \ov{d}. 
	$$ 
	As explained in \cite[Rem.~5.17]{FLMS}, one may get rid of the energy defect by setting $\mathfrak{E} = \frac{1}{\ov{d}} \tr \mathfrak{R}$. This is the reason why we stated Defn.~\ref{defn:dws} without mentioning an energy defect $\mathfrak{E}$. It remains to explain that one may choose $\ov{d}= 2$. To this end, we look at the more detailed treatment of the defect terms as carried out in \cite{BreFeiHof20_1}. Therein, the energy defect is split into a kinetic and an internal part, i.e.~$\mathfrak{E} = \mathfrak{E}_{\rm kin} + \mathfrak{E}_{\rm int}$. We observe from \cite[Defn.~2.1]{BreFeiHof20_1} that 
	$$
		\tr \mathfrak{R} = 2 \mathfrak{E}_{\rm kin} + n (\gamma-1) \mathfrak{E}_{\rm int}.
	$$
	Our assumption $\gamma-1 \leq \frac{2}{n}$ (see \eqref{eq:isentropic-pressure} therefore implies $\tr \mathfrak{R} \leq 2 \mathfrak{E}$, i.e.~$\ov{d}=2$ as desired. 
\end{rem} 

\begin{rem} \label{rem:dws-additional-facts} 
	\begin{itemize}
		\item Note that dissipative weak solutions may contain vacuum states, i.e.~the case $\rho=0$ is not excluded. Still we will be able to construct weak solutions which do \emph{not} contain vacuum as indicated already in Rem.~\ref{rem:vacuum} above. We will explain in Sect.~\ref{subsec:proof-outline} below how this is achieved. 
		
		\item In contrast to admissible weak solutions, where we require the local energy inequality to hold (see Defn.~\ref{defn:weak-sol}), a dissipative weak solution satisfies a global version of the energy inequality, i.e.~\eqref{eq:dws-energy}. We will point out in Sect.~\ref{subsec:proof-outline} below how the global energy inequality \eqref{eq:dws-energy} for dissipative weak solutions will finally lead to the local energy inequality \eqref{eq:euler-energy} of the weak solutions constructed in Sect.~\ref{subsec:proof-wildness} below.
		
		\item Any dissipative weak solution may be interpreted as a barycenter of a measure-valued solution, see \cite[Thm.~5.4]{FLMS} for more details.
	\end{itemize}
\end{rem} 

Dissipative weak solutions exist for arbitrary time intervals $[0,T]$:

\begin{prop} \label{prop:ex-dws}
	For any $T>0$ and initial data $\rho_0\in L^\gamma(\T^n;\R^+)$, $\vm_0\in L^{\frac{2\gamma}{\gamma+1}}(\T^n;\R^n)$ there exists a dissipative weak solution $(\rho,\vm)$ of the isentropic compressible Euler equations \eqref{eq:euler-mass}, \eqref{eq:euler-mom} (with pressure \eqref{eq:isentropic-pressure}).
\end{prop}

For the proof of Prop.~\ref{prop:ex-dws} we refer to \cite{FLMS}. Indeed, \cite[Thms.~5.3 and 5.4]{FLMS} reveal that every so-called \emph{consistent approximation} (see \cite[Defn.~5.4]{FLMS}) gives rise to a dissipative weak solution. Hence, it suffices to find a consistent approximation. It turns out that there exist several examples for consistent approximations, e.g.~numerical schemes as shown in \cite[Sect.~9 or Sect.~12]{FLMS}, or the approximation used in \cite[Sect.~3.2]{BreFeiHof20_1}.

We finish this subsection with recalling the following weak-strong-uniqueness principle. It says that dissipative weak solutions coincide with the strong solution as long as the latter exists.

\begin{prop} \label{prop:wsu}
	Let $(\rho,\vm)$ be a dissipative weak solution of the isentropic compressible Euler equations \eqref{eq:euler-mass}, \eqref{eq:euler-mom} (with pressure \eqref{eq:isentropic-pressure}) in $[0,T]\times \T^n$ with initial data $(\rho_0,\vm_0)$. Suppose that there exists a strong solution $(\rho_\ts,\vm_\ts)\in C^1\big([0,T]\times \T^n;\R^+\times \R^n\big)$ of \eqref{eq:euler-mass}, \eqref{eq:euler-mom} for the same initial data $(\rho_0,\vm_0)$. Then 
	$$
		\rho = \rho_\ts, \quad \vm=\vm_\ts, \quad \text{ and } \quad \mathfrak{R}=0 \qquad \text{ on } [0,T] \times \T^n.
	$$ 
\end{prop} 

The proof of Prop.~\ref{prop:wsu} can be found in \cite[Sect.~6.1.2]{FLMS}.

\section{Proof of the main result} \label{sec:proof} 

We are now ready to prove our main result, i.e.~Thm.~\ref{thm:density}. This section is organised as follows. In Sect.~\ref{subsec:proof-reduction} we reduce the statement of Thm.~\ref{thm:density} to a weaker version (see Thm.~\ref{thm:density-2}). We give a brief outline of the proof of Thm.~\ref{thm:density-2} in Sect.~\ref{subsec:proof-outline}. The remaining subsections, i.e.~Sect.~\ref{subsec:proof-modification} and \ref{subsec:proof-wildness}, are then devoted to the proof of Thm.~\ref{thm:density-2}. In Sect.~\ref{subsec:proof-modification} we construct the modified initial datum $(\rho_{0,\ep},\vm_{0,\ep})$ from the given datum $(\rho_0,\vm_0)$. Finally, in Sect.~\ref{subsec:proof-wildness}, we show that the initial datum $(\rho_{0,\ep},\vm_{0,\ep})$ is indeed wild in the sense of Defn.~\ref{defn:wild-data}.

\subsection{Reduction to a weaker version} \label{subsec:proof-reduction} 

Goal of this subsection is to reduce Thm.~\ref{thm:density} to the following version.

\begin{thm} \label{thm:density-2} 
	Let $(\rho_0,\vm_0)\in H^s(\T^n; \R^+\times \R)$ with $s>6$, and $\ep>0$. Then there exists a wild initial datum $(\rho_{0,\ep},\vm_{0,\ep})\in L^\infty(\T^n; \R^+\times \R)$ with 
	\begin{equation} \label{eq:thm-density-2-close}
		\|\rho_{0,\ep} - \rho_0\|_{L^\infty(\T^n)} < \ep, \qquad \|\vm_{0,\ep} - \vm_0\|_{L^2(\T^n)} <\ep , \quad \text{ and } \quad \|\vm_{0,\ep}\|_{L^\infty(\T^n)} \leq \|\vm_0\|_{L^\infty(\T^n)} + \ep.
	\end{equation}
\end{thm} 

Note that there are three differences between Thms.~\ref{thm:density} and \ref{thm:density-2}:
\begin{itemize}
	\item The initial datum $(\rho_0,\vm_0)$ in Thm.~\ref{thm:density-2} has a certain Sobolev regularity while in Thm.~\ref{thm:density} it is merely integrable.
	
	\item The initial density $\rho_0$ in Thm.~\ref{thm:density-2} is strictly positive while in Thm.~\ref{thm:density} it may contain vacuum states.

	\item The distance between the data $(\rho_0,\vm_0)$ and $(\rho_{0,\ep},\vm_{0,\ep})$ is measured in $L^\infty$ for the density and in $L^2$ for the momentum in Thm.~\ref{thm:density-2} while it is measured with respect to a general $L^r$-norm (where $r\in [1,\infty)$) in Thm.~\ref{thm:density}. 
\end{itemize}

Next we show that Thm.~\ref{thm:density-2} indeed implies Thm.~\ref{thm:density}. The proof of Thm.~\ref{thm:density-2} will be carried out in Sects.~\ref{subsec:proof-modification} and \ref{subsec:proof-wildness} below.

\begin{proof}[Proof of Thm.~\ref{thm:density}] 
By mollification and adding a small constant to the density, we observe that $H^s(\T^n; \R^+\times \R)$ is dense in $L^r(\T^n; \R^+_0\times \R)$ for any $s>6$ and any $r\in [1,\infty)$. Thus it suffices to consider $(\rho_0,\vm_0)\in H^s(\T^n; \R^+\times \R)$ rather than $(\rho_0,\vm_0)\in L^r(\T^n; \R^+_0\times \R)$.

In order to prove \eqref{eq:thm-density-close}, we distinguish between two cases. Firstly, we suppose that $r\in[1,2]$. Then $\|\cdot\|_{L^r(\T^n)}\leq \|\cdot\|_{L^2(\T^n)}$ and thus \eqref{eq:thm-density-2-close} immediately implies \eqref{eq:thm-density-close} as desired.

Secondly, we assume that $r\in (2,\infty)$. We run Thm.~\ref{thm:density-2} with 
$$
	\ep':= \min \Big\{ \ep, 1, (2\|\vm_0\|_{L^\infty(\T^n)} +1 )^{1-\frac{r}{2}} \ep^{\frac{r}{2}} \Big\}. 
$$
Similarly as above, we immediately obtain 
$$
	\|\rho_{0,\ep} - \rho_0\|_{L^r(\T^n)}\leq \|\rho_{0,\ep} - \rho_0\|_{L^\infty(\T^n)} < \ep' \leq \ep.
$$
By interpolation, we infer from \eqref{eq:thm-density-2-close}
\begin{align*}
	\|\vm_{0,\ep} - \vm_0\|_{L^r(\T^n)} &\leq \|\vm_{0,\ep} - \vm_0\|_{L^2(\T^n)}^{\frac{2}{r}} \|\vm_{0,\ep} - \vm_0\|_{L^\infty(\T^n)}^{1- \frac{2}{r}} \\
	&< (\ep')^{\frac{2}{r}} (2\|\vm_0\|_{L^\infty(\T^n)} +\ep')^{1- \frac{2}{r}} \\
	&\leq (2\|\vm_0\|_{L^\infty(\T^n)} +1 )^{(1-\frac{r}{2})\frac{2}{r}} \ep (2\|\vm_0\|_{L^\infty(\T^n)} +1)^{1- \frac{2}{r}} = \ep.
\end{align*}
We have proven \eqref{eq:thm-density-close} as desired.
\end{proof}

\subsection{Outline of the proof and comparison with the literature} \label{subsec:proof-outline} 

The general ideas for proving a density-of-wild-data statement like Thm.~\ref{thm:density} can by expressed as follows:
\begin{itemize}
	\item[1.] Construct a subsolution for the initial data $(\rho_0,\vm_0)$, which will not necessarily satisfy \eqref{eq:0001}.
	\item[2.] Run Prop.~\ref{prop:ci-data} to adjust the data so that \eqref{eq:0001} holds and the new data $(\rho_{0,\ep},\vm_{0,\ep})$ are close to $(\rho_0,\vm_0)$ as desired.
	\item[3.] Run Prop.~\ref{prop:ci-nonuniqueness} in order to prove non-uniqueness of admissible weak solutions for the initial data $(\rho_{0,\ep},\vm_{0,\ep})$. 
\end{itemize}
We also refer to Rems.~\ref{rem:ci-nonuniqueness} and \ref{rem:ci-data} where the role of \eqref{eq:0001} is explained in detail. 

The aforementioned strategy has already been used in the density-of-wild-data result by \name{Chen}-\name{Vasseur}-\name{Yu}~\cite{CheVasYu21} and the one by \name{Chiodaroli}-\name{Feireisl}~\cite{ChiFei24_1}. In \cite{CheVasYu21} a solution of the compressible Navier-Stokes equations with degenerate viscosity is used to construct a suitable subsolution in the first step, while in \cite{ChiFei24_1} the strong solution, which exists locally in time, is used instead. In this paper a dissipative weak solution as considered in Sect.~\ref{subsec:bb-dws} will take this role. In particular, we will see that the Reynolds defect has the correct sign (in the sense of definiteness) to be included into the matrix $\mU$, see Sect.~\ref{subsubsec:infsol-s4} below. 

This will, however, lead to four difficulties: 
\begin{itemize}[leftmargin=3cm]
	\item[\textbf{Difficulty 1:}] Subsolutions are required to be $C^1$ while dissipative weak solutions do not have this property. In order to solve this issue, we will mollify the dissipative weak solution, see Sect.~\ref{subsubsec:infsol-s2}.
	
	\item[\textbf{Difficulty 2:}] Dissipative weak solutions may contain vacuum states. To overcome this problem, we will add a small constant $r>0$ to the density, see Sect.~\ref{subsubsec:infsol-s3}.
	
	\item[\textbf{Difficulty 3:}] Dissipative weak solutions only satisfy the global energy inequality \eqref{eq:dws-energy} but the solutions which we want to construct have to fulfill the local energy inequality \eqref{eq:euler-energy}. We still achieve the latter by using a convex integration ansatz where the energy inequality is part of the differential inclusion. The global energy inequality \eqref{eq:dws-energy} ensures that we can construct a suitable $\vF$ by solving a Poisson equation, see Sect.~\ref{subsubsec:infsol-s5}.
	
	\item[\textbf{Difficulty 4:}] If we then run Prop.~\ref{prop:ci-data} to adjust the initial data, we would not have control over the distance between $\vm_0$ and $\vm_{0,\ep}$. In order to fix this problem, we have to invoke Prop.~\ref{prop:U}~\ref{item:U-2}. Then, however, we need to guarantee that \eqref{eq:cond-K-2} is only mildly violated. To achieve the latter, we must ensure that the Reynolds defect is zero for small times. But this is true since any dissipative weak solution coincides on a short time interval with the strong solution, cf.~Prop.~\ref{prop:wsu}.
\end{itemize}

As mentioned before, we split the Proof of Thm.~\ref{thm:density-2} into two parts: Firstly, in Sect.~\ref{subsec:proof-modification} we modify the initial data, which requires a local-in-time subsolution that results from the strong solution (cf.~Difficulty 4). Secondly, in Sect.~\ref{subsec:proof-wildness} we construct a subsolution globally in time. As the latter must coincide for small times with the local-in-time subsolution produced in Sect.~\ref{subsec:proof-modification}, the preceding steps of mollification (see Difficulty 1) and adding $r$ (see Difficulty 2) have to be carried out also in the first part, i.e.~in the local-in-time construction.

In total, both parts consist of six steps: 
\begin{itemize}[leftmargin=1.9cm]
	\item[Step 1:] The starting points, i.e.~the local-in-time strong solution (see Sect.~\ref{subsubsec:moddata-s1}) and the (global-in-time) dissipative weak solution (see Sect.~\ref{subsubsec:infsol-s1}) will be explained.
	
	\item[Step 2:] We will carry out a mollification in order to overcome Difficulty 1, see Sects.~\ref{subsubsec:moddata-s2} and \ref{subsubsec:infsol-s2}.
	
	\item[Step 3:] We will add a positive constant $r>0$ to the density to solve Difficulty 2, see Sects.~\ref{subsubsec:moddata-s3} and \ref{subsubsec:infsol-s3}.
	
	\item[Step 4:] The functions $\mU$ and $q$ will be defined, see Sects.~\ref{subsubsec:moddata-s4} and \ref{subsubsec:infsol-s4}.
	
	\item[Step 5:] The function $\vF$ will be defined by solving a Poisson equation, see Sects.~\ref{subsubsec:moddata-s5} and \ref{subsubsec:infsol-s5}.
	
	\item[Step 6:] Props.~\ref{prop:ci-data} and \ref{prop:ci-nonuniqueness}, respectively, will be run in order to modify the initial data (see Sect.~\ref{subsubsec:moddata-s6}) and to obtain infinitely many admissible weak solutions (see Sect.~\ref{subsubsec:infsol-s6}).
\end{itemize}

\subsection{Modification of the initial datum} \label{subsec:proof-modification} 

Let us turn our attention towards the proof of Thm.~\ref{thm:density-2}. In this subsection we will construct an initial datum $(\rho_{0,\ep},\vm_{0,\ep})\in L^\infty(\T^n; \R^+\times \R)$ from $(\rho_0,\vm_0)\in H^s(\T^n; \R^+\times \R)$. This will be done in six steps as mentioned in Sect.~\ref{subsec:proof-outline}. Recall that $s>6$ and $\ep>0$ are given. 

In the sequel, we will have to mollify functions in space-time. To this end, let $\psi$ denote a standard one-dimensional mollifier. By setting $\varphi(\vx):= C\psi (|\vx|)$ with $C$ such that $\int_{\R^n} \varphi(\vx) \dx=1$, we obtain an $n$-dimensional mollifier. The space-time mollifier $\phi$, which will be used in this paper, is then defined as the product $\phi(t,\vx) := \psi(t) \varphi(\vx)= C \psi(t) \psi(|\vx|)$. For a function $u=u(t,\vx)$ and a parameter $\sigma>0$, we write $u\ast \phi_{\sigma}$ for the space-time mollification, i.e.~the convolution of $u$ with $\phi_\sigma(t,\vx):= \frac{1}{\sigma^{n+1}} \phi\big(\frac{t}{\sigma}, \frac{\vx}{\sigma}\big)$. Similarly, $u \ast_t \psi_\sigma$ and $u\ast_\vx \varphi_\sigma$ denote the mollification solely in time and solely in space, respectively, where $\psi_\sigma(t):= \frac{1}{\sigma} \psi\big(\frac{t}{\sigma}\big)$ and $\varphi_\sigma(\vx):= \frac{1}{\sigma^{n}} \varphi\big(\frac{\vx}{\sigma}\big)$. The way we chose the mollifier $\phi$, there holds $u\ast \phi_{\sigma} = (u\ast_t \psi_\sigma) \ast_\vx \varphi_\sigma = (u\ast_\vx \varphi_\sigma) \ast_t \psi_\sigma$, i.e.~instead of mollifying in space-time, we may mollify subsequently in time and afterwards in space, or vice versa. This will be essential in the proofs of Lemmas~\ref{lemma:estimate-lambda-ov}, \ref{lemma:pdes-mollified-dws} and \ref{lemma:estimate-lambda-un}, as well as equation \eqref{eq:bound-rho2} below.

\subsubsection{Step 1: Strong solution locally in time} \label{subsubsec:moddata-s1} 

Prop.~\ref{prop:ex-strong} yields $T_\ts>0$ and a unique strong solution 
$$
	(\rho_\ts,\vm_\ts)\in C^0\big([-T_\ts,T_\ts];H^6(\T^n;\R^+\times \R^n)\big) \cap C^1\big([-T_\ts,T_\ts];H^5(\T^n;\R^+\times \R^n)\big)
$$ 
of \eqref{eq:euler-mass}-\eqref{eq:initial} with initial datum $(\rho_{0},\vm_{0})$. According to Sobolev embedding, we have $(\rho_\ts,\vm_\ts)\in C^1\big([-T_\ts,T_\ts];C^3(\T^n;\R^+\times \R^n)\big)$. 

For later use, we fix a constant $c>1$ such that 
\begin{equation} \label{eq:bounds-c-strongsol} 
	\rho_\ts(t,\vx) \geq \frac{1}{c} , \qquad \rho_\ts(t,\vx) + \ep \leq c, \qquad |\vm_\ts(t,\vx)| \leq c \quad \text{ and }\quad \Div \left( \frac{\vm_\ts(t,\vx)}{\rho_\ts(t,\vx)} \right) \leq c
\end{equation}
for all $(t,\vx)\in [-T_\ts,T_\ts]\times \T^n$, as well as
\begin{equation} \label{eq:bounds-c-ep} 
	\frac{1}{c} < \frac{\ep^2}{288 n}\qquad \text{ and } \qquad \frac{\ep^2}{144} \leq c^4, 
\end{equation}
and such that $c$ can be taken as the implicit constant in \eqref{eq:poisson-sobolev} in Cor.~\ref{cor:poisson-mf-sobolev}.

Moreover, we define compact and convex sets
\begin{align*}
	\sC_\rho &:= \left[\frac{1}{c}, c\right]\subset \R^+, \\ 
	\sC_\vm &:= \left\{ \vm\in \R^n \,\Big| \, |\vm|\leq c\right\} \subset \R^n, \\
	\sC_\mU &:= \left\{ \mU\in \symz{n} \,\Big|\, \text{all eigenvalues of }\mU\text{ are }\leq (n-1) \left(\frac{2c^3}{n} + p(c)\right) \right\} \subset \symz{n}, \\ 
	\sC_q &:= \left[ 0 , \frac{2c^3}{n} + p(c) \right] \subset \R^+_0. 
\end{align*} 

By construction and the choice of $c$ (see \eqref{eq:bounds-c-strongsol}), there holds 
$$
	(\rho_\ts, \vm_\ts, \mU_\ts, q_\ts) (t,\vx) \in \sC_\rho\times \sC_\vm \times \sC_\mU \times \sC_q \qquad \text{ for all } (t,\vx)\in [-T_\ts,T_\ts]\times \T^n
$$
where
\begin{align}
	\mU_\ts &:= \frac{\vm_\ts \otimes \vm_\ts}{\rho_\ts} - \frac{|\vm_\ts|^2}{n\rho_\ts} \id, & q_\ts &:= \frac{|\vm_\ts|^2}{n \rho_\ts} + p(\rho_\ts) . \label{eq:strong-U-q} 
\end{align} 

In the sequel, let $\delta>0$ from Prop.~\ref{prop:U}~\ref{item:U-2} where $\frac{\ep^2}{48 n c}$ plays the role of $\ep$, and 
\begin{equation} \label{eq:defn-C}
	\sC := \left\{ (\rho, \vm,\mU, q) \in \sC_\rho\times\sC_\vm\times \sC_\mU\times \sC_q\,\Big|\, \frac{\vm\otimes\vm}{\rho} + p(\rho)\id - \mU - q\id \leq - \frac{1}{c^2} \exp(-c T_\ts) \id \right\} . 
\end{equation} 
Note that $\sC$ is compact since it is a closed subset of the compact set $\sC_\rho\times\sC_\vm\times \sC_\mU\times \sC_q$. We may furthermore assume without loss of generality that 
\begin{equation} \label{eq:bound-delta}
	\delta\leq \frac{\ep^2}{72 c^4} 
\end{equation}
by shrinking $\delta$ if necessary.

\subsubsection{Step 2: Mollification of the local-in-time solution and definition of $\ov{\vm}$} \label{subsubsec:moddata-s2} 

Next, we mollify the strong solution $(\rho_\ts,\vm_\ts)$ with parameter $0<\sigma<\half T_\ts$, i.e.~we set
\begin{align*} 
	\rho_1 &:= \rho_\ts \ast \phi_{\sigma} , &
	\ov{\vm} &:= \vm_\ts \ast \phi_{\sigma} , 
\end{align*}
where $\sigma$ is chosen sufficiently small such that the following estimates hold: 
\begin{align} 
	\begin{split}
		& \bullet \quad \|\rho_1(0,\cdot) - \rho_0\|_{L^\infty(\T^n)} < \tfrac{1}{4}\ep \qquad \text{ and } \qquad \|\ov{\vm}(0,\cdot) - \vm_0\|_{L^\infty(\T^n)} < \tfrac{1}{4}\ep, 
	\end{split} \label{eq:est-molli-rm} \\[6mm]
	\begin{split} 
		& \bullet \quad \left( \frac{\vm_\ts\otimes \vm_\ts}{\rho_\ts} + p(\rho_\ts) \id \right) \ast \phi_\sigma - \left( \frac{\ov{\vm} \otimes \ov{\vm}}{\rho_1} + p(\rho_1) \id \right) \leq \frac{\ep^2}{144 nc } \id \\ 
		& \qquad\qquad\qquad\qquad \text{ for all }(t,\vx)\in \big[0,\half T_\ts\big]\times \T^n,
	\end{split} \label{eq:est-molli-matrix} \\[6mm]
	\begin{split} 
		& \bullet \quad p(\rho_\ts) \ast \phi_\sigma - p(\rho_1) \leq \frac{(\gamma-1) \delta}{16 c^2} \qquad\qquad \text{ for all }(t,\vx)\in \big[0,\half T_\ts\big]\times \T^n, 
	\end{split} \label{eq:est-molli-pressure} \\[6mm] 
	\begin{split} 
		& \bullet \quad \left| \frac{\vm_\ts}{\rho_\ts} - \frac{\ov{\vm}}{\rho_1}\right| \leq \frac{\delta c^2}{4n} \qquad\qquad \text{ for all }(t,\vx)\in \big[0,\half T_\ts\big]\times \T^n,
	\end{split} \label{eq:est-molli-velocity} \\[6mm] 
	\begin{split} 
		& \bullet \quad \left|  \frac{\frac{n}{2} \big(q_\ts - p(\rho_\ts)\big) + P(\rho_\ts) + p(\rho_\ts)}{\rho_\ts} \vm_\ts - \frac{\frac{n}{2} \big(q_\ts \ast \phi_{\sigma} - p(\rho_1)\big) + P(\rho_1) + p(\rho_1)}{\rho_1} \ov{\vm} \right| \leq \frac{\delta}{8} \\ 
		& \qquad\qquad\qquad\qquad \text{ for all }(t,\vx)\in \big[0,\half T_\ts\big]\times \T^n,
	\end{split} \label{eq:est-molli-energyflux} \\[6mm]
	\begin{split}
		& \bullet \quad \Big\| p(\rho_\ts) \ast \phi_\sigma - p(\rho_1)\Big\|_{C^1\big(\big[0,\half T_\ts\big]\times\T^n\big)} \leq \frac{(\gamma - 1)\delta}{16 c},
	\end{split} \label{eq:est-molli-pressurederivative} \\[6mm]
	\begin{split}
		& \bullet \quad \left\| \Big( \tfrac{n}{2} \big(q_\ts - p(\rho_\ts)\big) + P(\rho_\ts) \Big) - \Big( \tfrac{n}{2} \big(q_\ts \ast \phi_{\sigma} - p(\rho_1)\big) + P(\rho_1) \Big) \right\|_{C^1\big(\big[0,\half T_\ts\big]\times\T^n\big)} \leq \frac{\delta}{8 c}.
	\end{split} \label{eq:est-molli-energy} 
\end{align}

Note that
\begin{equation} \label{eq:subsol1-reg-m}
	(\rho_1,\ov{\vm}) \in C^\infty\big(\big[0,\half T_\ts\big]\times \T^n;\R^+\times \R^n\big).
\end{equation}
Moreover, as $\sC_\rho$ and $\sC_\vm$ are convex, there holds 
\begin{equation} \label{eq:subsol1-m-in-C}
	(\rho_1,\ov{\vm})(t,\vx)\in \sC_\rho\times\sC_\vm\qquad \text{ for all }(t,\vx)\in \big[0,\half T_\ts\big]\times \T^n.
\end{equation}
In addition to that, we observe that the PDEs 
\begin{align} 
	\partial_t \rho_1 + \Div \ov{\vm} &= 0, \label{eq:subsol1-mass-nolowerbound} \\
	\partial_t \ov{\vm} + \Div \bigg[\underbrace{\bigg(\frac{\vm_\ts \otimes \vm_\ts}{\rho_\ts} + p(\rho_\ts) \id\bigg)}_{=\mU_\ts + q_\ts\id}\ast \phi_{\sigma}\bigg] &= 0, \label{eq:subsol1-mom-nolambda} 
\end{align}
hold pointwise in $\big[0,\half T_\ts\big]\times \T^n$.

\subsubsection{Step 3: A lower bound for the density and definition of $\ov{\rho}$} \label{subsubsec:moddata-s3} 

Now we may fix $0<r<\tfrac{1}{4}\ep$ sufficiently small such that the following estimates hold: 
\begin{align} 
	\begin{split}
		& \bullet \quad p(\rho+r) - p(\rho) \leq \frac{(\gamma-1) \delta}{16 c^2} \qquad\qquad \text{ for all } \rho\in \left[0,\left(\frac{1}{\sigma}\right)^{n}\int_{\T^n} \rho_0 \dx\right]\cup \sC_\rho, 
	\end{split} \label{eq:est-shift-pressure} \\[6mm]
	\begin{split}
		& \bullet \quad \frac{1}{\rho} - \frac{1}{\rho + r} \leq \frac{\delta c}{4n} \qquad\qquad \text{ for all } \rho\in \sC_\rho, 
	\end{split} \label{eq:est-shift-density} \\[6mm]
	\begin{split}
		& \bullet \quad \left| \frac{\frac{n}{2} \big(q - p(\rho+r)\big) + P(\rho+r) + p(\rho+r)}{\rho+r} - \frac{\frac{n}{2} \big(q- p(\rho)\big) + P(\rho) + p(\rho)}{\rho} \right| \leq \frac{\delta}{8c} \\ 
		& \qquad\qquad\qquad\qquad \text{ for all }(\rho,q)\in \sC_\rho\times \sC_q, 
	\end{split} \label{eq:est-shift-energyflux} \\[6mm] 
	\begin{split}
		& \bullet \quad \Big| p'(\rho+r)-p'(\rho) \Big|\leq \frac{(\gamma-1) \delta}{16 c  \left( \|\rho_1\|_{C^1\big(\big[0,\half T_\ts\big]\times\T^n\big)} + 1 \right)} \qquad\qquad \text{ for all }\rho\in \sC_\rho. 
	\end{split} \label{eq:est-shift-pressurederivative} 
\end{align}
Note that this is possible as the mappings 
$$ 
	\rho\mapsto p(\rho), \quad \rho\mapsto \frac{1}{\rho}, \quad (\rho,q)\mapsto \frac{\frac{n}{2} \big(q- p(\rho)\big) + P(\rho) + p(\rho)}{\rho}, \quad \text{ and } \quad \rho\mapsto p'(\rho) 
$$ 
are uniformly continuous on compact sets. 

We then set 
$$
	\ov{\rho} := \rho_1 + r = \rho_\ts \ast \phi_{\sigma} + r.
$$
Using \eqref{eq:subsol1-reg-m}, we observe that 
\begin{equation} \label{eq:subsol1-reg-rho}
	\ov{\rho}\in C^\infty\big(\big[0,\half T_\ts\big]\times \T^n;\R^+\big).
\end{equation}
Furthermore, due to \eqref{eq:bounds-c-strongsol} and the fact that $r<\ep$, we have
\begin{equation} \label{eq:subsol1-rho-in-C}
	\ov{\rho}(t,\vx)\in \sC_\rho\qquad \text{ for all }(t,\vx)\in \big[0,\half T_\ts\big]\times \T^n,
\end{equation}
and -- in view of \eqref{eq:subsol1-mass-nolowerbound} -- the PDE
\begin{equation} \label{eq:subsol1-mass} 
	\partial_t \ov{\rho} + \Div \ov{\vm} = 0, 
\end{equation}
holds pointwise in $\big[0,\half T_\ts\big]\times \T^n$.

\subsubsection{Step 4: Definition and properties of $\ov{\mU}$ and $\ov{q}$} \label{subsubsec:moddata-s4} 

Let us now define
\begin{align*}
	\ov{\mU} &:= \mU_\ts \ast \phi_{\sigma}  ,&
	\ov{q} &:= q_\ts \ast \phi_{\sigma} + \ov{\lambda}(t),
\end{align*}
where $\mU_\ts$ and $q_\ts$ are defined in \eqref{eq:strong-U-q}, and $\ov{\lambda}(t)$ is given by 
\begin{align*} 
	\ov{\lambda}(t) &:= - \int_{\T^n} \Big( q_\ts \ast \phi_{\sigma} - p(\ov{\rho}) + \tfrac{2}{n} P(\ov{\rho}) \Big) \dx + \int_{\T^n}  \frac{2}{n} \left( \frac{|\vm_0|^2}{2\rho_0} + P(\rho_0) \right) \dx \\
	&\qquad + \frac{\delta}{8 n c^2} + \frac{1}{c^2} \exp( - ct ). 
\end{align*}

Note that
\begin{equation} \label{eq:subsol1-reg-Uq}
	(\ov{\mU},\ov{q})\in C^\infty\big(\big[0,\half T_\ts\big]\times \T^n; \symz{n} \times \R\big).
\end{equation} 
Due to the convexity of $\sC_\mU$, there holds
\begin{equation} \label{eq:subsol1-U-in-C}
	\ov{\mU}(t,\vx)\in \sC_\mU \qquad \text{ for all }(t,\vx)\in \big[0,\half T_\ts\big]\times \T^n.
\end{equation}
Furthermore, according to \eqref{eq:subsol1-mom-nolambda}, the PDE 
\begin{equation} \label{eq:subsol1-mom} 
	\partial_t \ov{\vm} + \Div (\ov{\mU} + \ov{q} \id) = 0 
\end{equation}
holds pointwise in $\big[0,\half T_\ts\big]\times \T^n$. 

Let us next find bounds for $\ov{\lambda}(t)$.

\begin{lemma} \label{lemma:estimate-lambda-ov} 
	For any $t\in \big[0,\half T_\ts\big]$ the following estimates hold
	\begin{align*} 
		&0 < \frac{(\gamma-1) \delta}{16c^2} + \frac{1}{c^2} \exp( - ct ) \leq \ov{\lambda}(t) \leq \frac{\delta}{4 n c^2} + \frac{1}{c^2} \exp( - ct ) \leq \frac{\ep^2}{144 n c}, \\ 
		& \left| \partial_t \left(\ov{\lambda}(t) -\frac{1}{c^2} \exp(-ct) \right) \right| \leq \frac{\delta}{4nc} . 
	\end{align*}
\end{lemma}

\begin{proof} 
	It is a simple observation that 
	\begin{equation} \label{eq:total-mass-mollification}
		\int_{\T^n}  f(\vx) \dx = \int_{\T^n} (f \ast_\vx \varphi_\sigma)(\vx) \dx 
	\end{equation}
	for a function $f$ satisfying $f\geq 0$. Moreover, is well-known that strong solutions conserve the total energy, i.e.
	$$
		 \int_{\T^n}  \left( \frac{|\vm_\ts|^2}{2\rho_\ts} + P(\rho_\ts) \right) \dx 
	$$
	is constant for all $t\in[-T_\ts,T_\ts]$. Combining these two facts, we find that 
	\begin{align*}
		\int_{\T^n}  \left( \frac{|\vm_\ts|^2}{2\rho_\ts} + P(\rho_\ts) \right) \ast \phi_\sigma \dx &= \left( \int_{\T^n}  \left( \frac{|\vm_\ts|^2}{2\rho_\ts} + P(\rho_\ts) \right) \ast_\vx \varphi_\sigma \dx \right) \ast_t \psi_\sigma \\
		&=  \int_{\T^n}  \left( \frac{|\vm_0|^2}{2\rho_0} + P(\rho_0) \right) \dx 
	\end{align*}
	for all $t\in\big[0,\half T_\ts\big]$. Consequently there holds\footnote{Recall \eqref{eq:strong-U-q}.}
	$$
		- \int_{\T^n} \Big( q_\ts - p(\rho_\ts) + \tfrac{2}{n} P(\rho_\ts) \Big) \ast \phi_{\sigma} \dx + \int_{\T^n}  \frac{2}{n} \left( \frac{|\vm_0|^2}{2\rho_0} + P(\rho_0) \right) \dx \ = \ 0
	$$
	for $t\in\big[0,\half T_\ts\big]$. Hence we may simplify $\ov{\lambda}$ to 
	\begin{align} 
		\ov{\lambda}(t) &= - \int_{\T^n} \Big(- p(\ov{\rho}) + \tfrac{2}{n} P(\ov{\rho}) \Big) \dx + \int_{\T^n} \Big( - p(\rho_\ts) + \tfrac{2}{n} P(\rho_\ts) \Big) \ast \phi_{\sigma} \dx \label{eq:lambda-better-defn} \\
		&\qquad + \frac{\delta}{8 n c^2} + \frac{1}{c^2} \exp( - ct ). \notag
	\end{align}
	
	Keeping in mind that $\gamma\leq 1+ \frac{2}{n}$ and hence $\tfrac{2}{n} \, \tfrac{1}{\gamma-1} - 1\geq 0$, we obtain
	\begin{align*} 
		&- \int_{\T^n} \Big(- p(\ov{\rho}) + \tfrac{2}{n} P(\ov{\rho}) \Big) \dx + \int_{\T^n} \Big( - p(\rho_\ts) + \tfrac{2}{n} P(\rho_\ts) \Big) \ast \phi_{\sigma} \dx \\
		&= \Big( \tfrac{2}{n} \, \tfrac{1}{\gamma-1} - 1 \Big) \left( \int_{\T^n} \Big( p(\rho_\ts) \ast \phi_{\sigma} - p(\rho_1) \Big) \dx + \int_{\T^n} \Big( p(\rho_1) - p(\ov{\rho}) \Big) \dx \right) \\
		&\left\{ \begin{array}{l} \leq \Big( \tfrac{2}{n} \, \tfrac{1}{\gamma-1} - 1 \Big) \frac{(\gamma-1) \delta}{16 c^2}, \\ \geq - \Big( \tfrac{2}{n} \, \tfrac{1}{\gamma-1} - 1 \Big) \frac{(\gamma-1) \delta}{16 c^2} .\end{array} \right. 
	\end{align*}
	Here we made use of 
	$$
		0 \leq \int_{\T^n} \Big( p(\rho_\ts) \ast \phi_{\sigma} - p(\rho_1) \Big) \dx \leq \frac{(\gamma-1) \delta}{16 c^2}, 
	$$
	which hold according to Jensen's inequality (Lemma~\ref{lemma:jensen-real-valued}) and \eqref{eq:est-molli-pressure}, as well as 
	$$
		- \frac{(\gamma-1) \delta}{16 c^2} \leq \int_{\T^n} \Big( p(\rho_1) - p(\ov{\rho}) \Big) \dx \leq 0
	$$ 
	which follow from \eqref{eq:est-shift-pressure} and the fact that $\rho\mapsto p(\rho)$ is non-decreasing.
	
	Plugging this into \eqref{eq:lambda-better-defn} we end up with
	$$
		\ov{\lambda}(t) -  \frac{1}{c^2} \exp( - ct ) \leq \frac{\delta}{8 n c^2} + \left( \frac{2}{n} \, \frac{1}{\gamma-1} - 1 \right) \frac{(\gamma-1) \delta}{16 c^2} \leq \frac{\delta}{8 n c^2} + \frac{\delta}{8 n c^2} = \frac{\delta}{4 n c^2},
	$$	
	and 
	$$
		\ov{\lambda}(t) -  \frac{1}{c^2} \exp( - ct ) \geq \frac{\delta}{8 n c^2} - \left( \frac{2}{n} \, \frac{1}{\gamma-1} - 1 \right) \frac{(\gamma-1) \delta}{16 c^2} = \frac{\delta}{8 n c^2} - \frac{\delta}{8 n c^2} + \frac{(\gamma-1) \delta}{16 c^2} = \frac{(\gamma-1) \delta}{16 c^2}. 
	$$
	
	Moreover, we obtain from \eqref{eq:bound-delta} and \eqref{eq:bounds-c-ep}
	$$
		\frac{\delta}{4 n c^2} + \frac{1}{c^2} \exp( - ct ) \leq \frac{\ep^2}{288 n c^6} + \frac{1}{c^2} \leq \frac{\ep^2}{288 n c} + \frac{\ep^2}{288 n c} = \frac{\ep^2}{144 n c}. 
	$$
	
	Finally, we infer from \eqref{eq:lambda-better-defn}, \eqref{eq:est-molli-pressurederivative} and \eqref{eq:est-shift-pressurederivative}
	\begin{align*}
		&\left| \partial_t \left(\ov{\lambda}(t) -\frac{1}{c^2} \exp(-ct) \right) \right| \\
		&= \Big( \tfrac{2}{n} \, \tfrac{1}{\gamma-1} - 1 \Big) \left| \int_{\T^n} \partial_t \Big( p(\rho_\ts) \ast \phi_{\sigma} - p(\rho_1) \Big) \dx + \int_{\T^n} \partial_t \Big( p(\rho_1) - p(\ov{\rho}) \Big) \dx \right| \\
		& \leq \frac{2}{n} \, \frac{1}{\gamma-1} \Big\| p(\rho_\ts) \ast \phi_{\sigma} - p(\rho_1) \Big\|_{C^1\big(\big[0,\half T_\ts\big]\times\T^n\big)} \\
		&\qquad + \frac{2}{n} \, \frac{1}{\gamma-1} \Big\| p'(\rho_1) - p'(\rho_1 + r) \Big\|_{C^0\big(\big[0,\half T_\ts\big]\times\T^n\big)} \|\rho_1\|_{C^1\big(\big[0,\half T_\ts\big]\times\T^n\big)} \\
		&\leq \frac{2}{n} \, \frac{1}{\gamma-1} \frac{(\gamma - 1)\delta}{16 c} + \frac{2}{n} \, \frac{1}{\gamma-1} \frac{(\gamma-1) \delta}{16 c  \left( \|\rho_1\|_{C^1\big(\big[0,\half T_\ts\big]\times\T^n\big)} + 1 \right)} \|\rho_1\|_{C^1\big(\big[0,\half T_\ts\big]\times\T^n\big)} \\
		&\leq \frac{\delta}{8 n c} + \frac{\delta}{8 n c} \ = \ \frac{\delta}{4 n c}.
	\end{align*}
\end{proof}

With the help of Lemma~\ref{lemma:estimate-lambda-ov} we may now deduce that
\begin{equation} \label{eq:subsol1-q-in-C}
	\ov{q}(t,\vx)\in \sC_q \qquad \text{ for all }(t,\vx)\in \big[0,\half T_\ts\big]\times \T^n.
\end{equation}
Indeed we find using Lemma~\ref{lemma:estimate-lambda-ov} and \eqref{eq:bounds-c-ep}
$$
	\ov{q}(t,\vx) = (q_\ts \ast \phi_\sigma)(t,\vx) + \ov{\lambda}(t) \leq \frac{c^3}{n} + p(c) + \frac{\ep^2}{144 n c} \leq \frac{2c^3}{n} + p(c).
$$

Next, we define for $(t,\vx)\in \big[0,\half T_\ts\big]\times \T^n$ the following symmetric matrices
\begin{align*} 
	\ov{\mM}_1(t,\vx) &:= \frac{\ov{\vm}\otimes\ov{\vm}}{\rho_1} - \frac{\ov{\vm}\otimes\ov{\vm}}{\ov{\rho}} , \\
	\ov{\mM}_2(t,\vx) &:= \big( p(\ov{\rho}) - p(\rho_1) \big) \id, \\
	\ov{\mM}_3(t,\vx) &:= \left(\frac{\vm_\ts \otimes \vm_\ts}{\rho_\ts} + p(\rho_\ts)\id \right)\ast \phi_{\sigma} - \left( \frac{\ov{\vm}\otimes\ov{\vm}}{\rho_1} + p(\rho_1)\id \right). 
\end{align*}
For those matrices, we have the following bounds.

\begin{lemma} \label{lemma:estimate-matrices-ov}
	For any $(t,\vx)\in \big[0,\half T_\ts\big]\times \T^n$ the following estimates hold:
	\begin{align*} 
		0 &\leq \ov{\mM}_1 \leq \frac{\ep^2}{144 n c} \id , &
		0 &\leq \ov{\mM}_2 \leq \frac{(\gamma-1) \delta}{16c^2} \id, &
		0 &\leq \ov{\mM}_3 \leq \frac{\ep^2}{144 nc} \id . 
	\end{align*}
\end{lemma}

\begin{proof} 
	We begin with the estimate concerning $\ov{\mM}_1$. From \eqref{eq:est-shift-density} and \eqref{eq:bound-delta} we obtain 
	$$
		\ov{\mM}_1 = \ov{\vm} \otimes \ov{\vm} \left( \frac{1}{\rho_1} - \frac{1}{\rho_1+r} \right) \leq c^2 \frac{\delta c}{4 n} \id \leq \frac{\ep^2}{288 n c} \id \leq \frac{\ep^2}{144 n c} \id. 
	$$
	As $\ov{\vm}\otimes \ov{\vm}$ is positive semi-definite\footnote{It is simple to show that $|\vm|^2$ and $0$ are the only eigenvalues of the matrix $\ov{\vm}\otimes \ov{\vm}$.} and $\frac{1}{\rho_1} - \frac{1}{\rho_1 + r} \geq 0$, we infer $\ov{\mM}_1 \geq 0$.
	
	Next, we look at $\ov{\mM}_2$. As $\rho\mapsto p(\rho)$ is a non-decreasing function, we have $p(\rho_1+r) - p(\rho_1)\geq 0$ and hence $\ov{\mM}_2\geq 0$. The upper bound $\ov{\mM}_2 \leq  \frac{(\gamma-1) \delta}{16c^2} \id$ follows immediately from \eqref{eq:est-shift-pressure}. 
	
	Finally, we prove the estimate regarding $\ov{\mM}_3$. Jensen's inequality (see Lemma~\ref{lemma:jensen-matrix-valued}) and Lemma~\ref{lemma:map-convex} yield $\ov{\mM}_3 \geq 0$. The remaining bound $\ov{\mM}_3 \leq \frac{\ep^2}{144 nc} \id$ comes from \eqref{eq:est-molli-matrix}.
\end{proof}

Using Lemmas~\ref{lemma:estimate-lambda-ov} and \ref{lemma:estimate-matrices-ov} we may now estimate for arbitrary $(t,\vx)\in \big[0,\half T_\ts\big]\times \T^n$ 
\begin{align} 
	&\frac{\ov{\vm}\otimes\ov{\vm}}{\ov{\rho}} + p(\ov{\rho})\id - \ov{\mU} - \ov{q}\id \notag \\
	&= \frac{\ov{\vm}\otimes\ov{\vm}}{\ov{\rho}} + p(\ov{\rho})\id - \left(\frac{\vm_\ts \otimes \vm_\ts}{\rho_\ts} + p(\rho_\ts)\id \right)\ast \phi_{\sigma} - \ov{\lambda}(t)\id \notag \\
	&= - \ov{\mM}_1 + \ov{\mM}_2 - \ov{\mM}_3 - \ov{\lambda}(t) \id \notag \\
	&\geq - \frac{\ep^2}{144 n c} \id - \frac{\ep^2}{144 n c} \id - \frac{\ep^2}{144 n c} \id \ =\  - \frac{\ep^2}{48 n c} \id , \label{eq:matrix-lower-bound}
\end{align} 
as well as
\begin{align} 
	&\frac{\ov{\vm}\otimes\ov{\vm}}{\ov{\rho}} + p(\ov{\rho})\id - \ov{\mU} - \ov{q}\id \notag \\
	&= -\ov{\mM}_1 + \ov{\mM}_2 - \ov{\mM}_3 - \ov{\lambda}(t) \id \notag \\
	&\leq \frac{(\gamma-1) \delta}{16c^2} \id - \frac{(\gamma-1) \delta}{16c^2} \id - \frac{1}{c^2} \exp( - ct ) \id\ \leq \ - \frac{1}{c^2} \exp( - cT_\ts ) \id. \label{eq:matrix-upper-bound}
\end{align}

\subsubsection{Step 5: Definition of $\ov{\vF}$} \label{subsubsec:moddata-s5}

Now, we turn our attention towards the definition of $\ov{\vF}$. A straightforward computation yields
\begin{align*} 
	&\int_{\T^n} \partial_t \left( \frac{n}{2} \left(\ov{q} - \frac{1}{c^2} \exp(-ct) - p(\ov{\rho})\right) + P(\ov{\rho}) \right) \dx \\
	&= \partial_t \int_{\T^n} \left( \frac{n}{2} \left(q_\ts \ast \phi_{\sigma} + \ov{\lambda}(t) - \frac{1}{c^2} \exp(-ct) - p(\ov{\rho})\right) + P(\ov{\rho}) \right) \dx \\
	&= \partial_t \bigg[ \int_{\T^n} \Big( \tfrac{n}{2} \big(q_\ts \ast \phi_{\sigma} - p(\ov{\rho})\big) + P(\ov{\rho}) \Big) \dx + \frac{n}{2} \left(\ov{\lambda}(t) - \frac{1}{c^2} \exp(-ct)\right) \bigg] \\
	&= \partial_t \bigg[\int_{\T^n} \Big( \tfrac{n}{2} \big(q_\ts \ast \phi_{\sigma} - p(\ov{\rho})\big) + P(\ov{\rho}) \Big) \dx - \frac{n}{2} \int_{\T^n} \Big( q_\ts \ast \phi_{\sigma} - p(\ov{\rho}) + \tfrac{2}{n} P(\ov{\rho}) \Big) \dx \\
	&\qquad + \int_{\T^n}  \left( \frac{|\vm_0|^2}{2\rho_0} + P(\rho_0) \right) \dx + \frac{n}{2} \, \frac{\delta}{8 n c^2} \bigg] \ =\ 0,
\end{align*} 
and, consequently, by Lemma~\ref{lemma:poisson-mf} there exists $v$ such that 
\begin{equation} \label{eq:subsol1-poisson-energy}
	\Lap v = \partial_t \left( \frac{n}{2} \left(\ov{q} - \frac{1}{c^2} \exp(- ct) - p(\ov{\rho})\right) + P(\ov{\rho}) \right).
\end{equation}
As the right-hand side of \eqref{eq:subsol1-poisson-energy} is smooth (see \eqref{eq:subsol1-reg-rho}, \eqref{eq:subsol1-reg-Uq}), we have $v\in C^2\big(\big[0,\half T_\ts\big]\times \T^n\big)$. Finally, we set 
$$
	\ov{\vF} := -\Grad v + \Leray \vF_\ts + \frac{n}{2c^2} \exp(-ct) \frac{\vm_\ts}{\rho_\ts} .
$$
Here $\Leray$ is the projection onto divergence-free vector fields, see Lemma~\ref{lemma:helmholtz} for details, and 
\begin{equation} \label{eq:identity-Fs}
	\vF_\ts := \left(\frac{|\vm_\ts|^2}{2\rho_\ts} + P(\rho_\ts) + p(\rho_\ts)\right) \frac{\vm_\ts}{\rho_\ts} = \frac{\frac{n}{2} \big(q_\ts - p(\rho_\ts)\big) + P(\rho_\ts) + p(\rho_\ts)}{\rho_\ts} \vm_\ts 
\end{equation}
is the energy flux of the strong solution $(\rho_\ts,\vm_\ts)$. Note that $\vF_\ts\in C^1\big(\big[0,\half T_\ts\big];C^3(\T^n;\R^n)\big)$ and hence $\vF_\ts\in C^1\big(\big[0,\half T_\ts\big];H^3(\T^n;\R^n)\big)$. According to Lemma~\ref{lemma:helmholtz}, this implies $\Leray\vF_\ts\in C^1\big(\big[0,\half T_\ts\big];H^3(\T^n;\R^n)\big)$, and Sobolev embedding yields $\Leray\vF_\ts\in C^1\big(\big[0,\half T_\ts\big]\times\T^n;\R^n\big)$. Keeping in mind that $(\rho_\ts,\vm_\ts)\in C^1\big(\big[0,\half T_\ts\big];C^3(\T^n;\R^+\times \R^n)\big)$, we therefore find
\begin{equation} \label{eq:subsol1-reg-F}
	\ov{\vF} \in C^1\big(\big[0,\half T_\ts\big]\times \T^n; \R^n\big).
\end{equation}
Moreover, due to \eqref{eq:subsol1-poisson-energy} and \eqref{eq:bounds-c-strongsol} the inequality 
\begin{equation} \label{eq:subsol1-energy} 
	\partial_t \Big( \tfrac{n}{2} \big(\ov{q} - p(\ov{\rho}) \big) + P(\ov{\rho})\Big) + \Div \ov{\vF} = \frac{n}{2c^2} \exp(-ct) \left[- c + \Div\left(\frac{\vm_\ts}{\rho_\ts}\right) \right] \leq 0 
\end{equation}
is satisfied pointwise in $\big[0,\half T_\ts\big]\times \T^n$.

Next, we check the condition of Prop.~\ref{prop:U}~\ref{item:U-2}.

\begin{lemma} \label{lemma:estimate-F} 
	For any $(t,\vx)\in \big[0,\half T_\ts\big]\times \T^n$ there holds
	$$
		\left| \ov{\vF} -  \frac{\frac{n}{2} \big(\ov{q} - p(\ov{\rho})\big) + P(\ov{\rho}) + p(\ov{\rho})}{\ov{\rho}} \ov{\vm} \right| < \delta.
	$$
\end{lemma}

\begin{proof}
	Using \eqref{eq:identity-Fs}, we split the left-hand side of the desired estimate as 
	\begin{align*} 
		&\left| \ov{\vF} - \frac{\frac{n}{2} \big(\ov{q} - p(\ov{\rho})\big) + P(\ov{\rho}) + p(\ov{\rho})}{\ov{\rho}} \ov{\vm} \right| \\
		&\leq \left|\ov{\vF} - \frac{n}{2c^2} \exp(-ct) \frac{\vm_\ts}{\rho_\ts} - \vF_\ts\right| \\
		&\qquad + \left|  \frac{\frac{n}{2} \big(q_\ts - p(\rho_\ts)\big) + P(\rho_\ts) + p(\rho_\ts)}{\rho_\ts} \vm_\ts - \frac{\frac{n}{2} \big(q_\ts \ast \phi_{\sigma} - p(\rho_1)\big) + P(\rho_1) + p(\rho_1)}{\rho_1} \ov{\vm} \right| \\
		&\qquad + \left| \frac{\frac{n}{2} \big(q_\ts \ast \phi_{\sigma} - p(\rho_1)\big) + P(\rho_1) + p(\rho_1)}{\rho_1} \ov{\vm} - \frac{\frac{n}{2} \big(q_\ts \ast \phi_{\sigma} - p(\ov{\rho})\big) + P(\ov{\rho}) + p(\ov{\rho})}{\ov{\rho}} \ov{\vm} \right| \\
		&\qquad + \left| \frac{n}{2c^2} \exp(-ct) \frac{\vm_\ts}{\rho_\ts} - \frac{n}{2} \ov{\lambda}(t) \frac{\ov{\vm}}{\ov{\rho}} \right| \ =: \ A_1 + A_2 + A_3 + A_4 
	\end{align*}
	and estimate the four terms $A_i$ ($i=1,2,3,4$) on the right-hand side separately. 
	
	We begin with the most delicate term $A_1$. Let $v_\ts$ denote the solution of the Poisson equation 
	\begin{equation} \label{eq:strong-poisson-energy}
		\Lap v_\ts = \partial_t \Big( \tfrac{n}{2} \big(q_\ts - p(\rho_\ts)\big) + P(\rho_\ts) \Big) = \partial_t \left( \frac{|\vm_\ts|^2}{2\rho_\ts} + P(\rho_\ts)\right).
	\end{equation}
	Note that the right-hand side $\partial_t \left( \frac{|\vm_\ts|^2}{2\rho_\ts} + P(\rho_\ts)\right)\in C^0\big(\big[0,\half T_\ts\big];H^3(\T^n;\R^n)\big)$, as well as 
	$$
		\partial_t \int_{\T^n}  \left( \frac{|\vm_\ts|^2}{2\rho_\ts} + P(\rho_\ts)\right) \dx = 0,
	$$
	due to the fact that strong solutions conserve the total energy. Consequently, Lemma~\ref{lemma:poisson-mf} indeed yields existence of a solution $v_\ts\in C^0\big(\big[0,\half T_\ts\big];H^5(\T^n;\R^n)\big)$ of \eqref{eq:strong-poisson-energy}.
	
	Now let $\vF_\ts = -\Grad w_\ts + \Leray \vF_\ts$ be the Helmholtz decomposition of $\vF_\ts$, see Lemma~\ref{lemma:helmholtz}. As strong solutions satisfy the energy inequality \eqref{eq:euler-energy} as an equation, we find
	$$
		\Lap w_\ts = - \Div \vF_\ts = \partial_t \left( \frac{|\vm_\ts|^2}{2\rho_\ts} + P(\rho_\ts)\right).
	$$ 
	In particular, $v_\ts= w_\ts$ by uniqueness of solutions to the Poisson equation.
	
	Using Cor.~\ref{cor:poisson-mf-sobolev}  and keeping in mind that we have chosen $c$ sufficiently large, so that it can be used as the implicit constant in \eqref{eq:poisson-sobolev}, we may estimate
	\begin{align*}
		&\left\|\ov{\vF} - \frac{n}{2c^2} \exp(-ct) \frac{\vm_\ts}{\rho_\ts} - \vF_\ts\right\|_{C^0\big(\big[0,\half T_\ts\big]\times \T^n\big)} \\ 
		&= \Big\|-\Grad v + \Leray \vF_\ts + \Grad v_\ts - \Leray \vF_\ts\|_{C^0\big(\big[0,\half T_\ts\big]\times \T^n\big)} \\
		&= \Big\| \Grad( v_\ts - v )\Big\|_{C^0\big(\big[0,\half T_\ts\big]\times \T^n\big)} \\
		&\leq c \left\|\partial_t \Big( \tfrac{n}{2} \big(q_\ts - p(\rho_\ts)\big) + P(\rho_\ts) \Big) - \partial_t \left( \frac{n}{2} \left(\ov{q} - \frac{1}{c^2} \exp(- ct) - p(\ov{\rho})\right) + P(\ov{\rho}) \right) \right\|_{C^0\big(\big[0,\half T_\ts\big];C^1(\T^n)\big)} \\
		&\leq c \left\| \Big( \tfrac{n}{2} \big(q_\ts - p(\rho_\ts)\big) + P(\rho_\ts) \Big) - \Big( \tfrac{n}{2} \big(q_\ts \ast \phi_{\sigma} - p(\rho_1)\big) + P(\rho_1) \Big) \right\|_{C^1\big(\big[0,\half T_\ts\big]\times\T^n\big)} \\
		&\qquad + c \left\| \Big( - \tfrac{n}{2} p(\rho_1) + P(\rho_1) \Big) - \Big( - \tfrac{n}{2} p(\rho_1+r) + P(\rho_1+r) \Big) \right\|_{C^1\big(\big[0,\half T_\ts\big]\times\T^n\big)} \\
		&\qquad + \frac{nc}{2} \left\| \partial_t \left( \ov{\lambda}(t) - \frac{1}{c^2} \exp(-ct) \right) \right\|_{C^0\big(\big[0,\half T_\ts\big]\big)} \ =:\ B_1 + B_2+ B_3.
	\end{align*}
	
	From \eqref{eq:est-molli-energy}, we immediately see that $B_1\leq \frac{\delta}{8}$. Using the fact that $\gamma-1 \leq \frac{2}{n}$, as well as \eqref{eq:est-shift-pressure} and \eqref{eq:est-shift-pressurederivative}, we find
	\begin{align*} 
		B_2 &= c \left| \tfrac{n}{2}-\tfrac{1}{\gamma-1}\right| \Big\| p(\rho_1+r) -  p(\rho_1)\Big\|_{C^1\big(\big[0,\half T_\ts\big]\times\T^n\big)}  \\
		&\leq c \left( \tfrac{1}{\gamma-1} - \tfrac{n}{2}\right) \bigg(\Big\| p(\rho_1+r) -  p(\rho_1)\Big\|_{C^0\big(\big[0,\half T_\ts\big]\times\T^n\big)} + \\
		&\qquad\qquad\qquad\qquad + \Big\| p'(\rho_1+r) -  p'(\rho_1)\Big\|_{C^0\big(\big[0,\half T_\ts\big]\times\T^n\big)} \|\rho_1\|_{C^1\big(\big[0,\half T_\ts\big]\times\T^n\big)} \bigg) \\
		&\leq \frac{\delta}{16} + \frac{\delta \|\rho_1\|_{C^1\big(\big[0,\half T_\ts\big]\times\T^n\big)}}{16 \left(\|\rho_1\|_{C^1\big(\big[0,\half T_\ts\big]\times\T^n\big)} + 1\right)} \ <\ \frac{\delta}{8}.
	\end{align*} 
	From Lemma~\ref{lemma:estimate-lambda-ov} we obtain $B_3\leq \frac{\delta}{8}$. Thus, we have $A_1 \leq B_1+B_2+B_3 < \frac{3\delta}{8}$. 
	
	We immediately infer $A_2\leq\frac{\delta}{8}$ and $A_3\leq\frac{\delta}{8}$ from \eqref{eq:est-molli-energyflux} and \eqref{eq:est-shift-energyflux}, respectively. 
	
	Finally, we estimate $A_4$. There holds
	\begin{align*}
		A_4 &= \left| \frac{n}{2c^2} \exp(-ct) \frac{\vm_\ts}{\rho_\ts} - \frac{n}{2} \ov{\lambda}(t) \frac{\ov{\vm}}{\ov{\rho}} \right| \\
		&\leq \frac{n}{2c^2} \exp(-ct) \left| \frac{\vm_\ts}{\rho_\ts} - \frac{\ov{\vm}}{\rho_1} \right| +  \frac{n}{2c^2} \exp(-ct) \left| \frac{\ov{\vm}}{\rho_1} - \frac{\ov{\vm}}{\ov{\rho}} \right| + \frac{n}{2} \left| \frac{1}{c^2} \exp(-ct) - \ov{\lambda}(t) \right| \left| \frac{\ov{\vm}}{\ov{\rho}} \right| \\
		& \leq \frac{n}{2c^2} \, \frac{\delta c^2}{4n} +  \frac{n}{2c^2} \, \frac{\delta c}{4n} \,c +  \frac{n}{2} \,\frac{\delta}{4nc^2} \, c^2 \ = \ \frac{3\delta}{8}, 
	\end{align*}
	according to \eqref{eq:est-molli-velocity}, \eqref{eq:est-shift-density} and Lemma~\ref{lemma:estimate-lambda-ov}.
	
	Summing all terms, we find $A_1+A_2+A_3+A_4 < \delta$ as desired. 
\end{proof}

\subsubsection{Step 6: Adjusting the initial datum using Prop.~\ref{prop:ci-data}} \label{subsubsec:moddata-s6} 

In the preceding steps, we have constructed a tuple of functions 
\begin{equation} \label{eq:subsol1-reg}
	(\ov{\rho},\ov{\vm},\ov{\mU},\ov{q},\ov{\vF})\in C^1\big(\big[0,\half T_\ts\big]\times \T^n;\R^+ \times \phase\big),
\end{equation}
see \eqref{eq:subsol1-reg-rho}, \eqref{eq:subsol1-reg-m}, \eqref{eq:subsol1-reg-Uq} and \eqref{eq:subsol1-reg-F}. This tuple satisfies the linear system \eqref{eq:eulerlin-mass}-\eqref{eq:eulerlin-energy} for all $(t,\vx)\in \big(0,\half T_\ts\big)\times \T^n$ according to \eqref{eq:subsol1-mass}, \eqref{eq:subsol1-mom} and \eqref{eq:subsol1-energy}, respectively. 

Moreover, we have already seen that 
$$
	(\ov{\rho},\ov{\vm}, \ov{\mU}, \ov{q})(t,\vx)\in \sC_\rho \times \sC_\vm \times \sC_\mU \times \sC_q \qquad \text{ for all }(t,\vx)\in \big[0,\half T_\ts\big]\times \T^n,
$$
cf.~\eqref{eq:subsol1-rho-in-C}, \eqref{eq:subsol1-m-in-C}, \eqref{eq:subsol1-U-in-C} and \eqref{eq:subsol1-q-in-C}. From \eqref{eq:matrix-upper-bound} we obtain
\begin{equation*} 
	(\ov{\rho},\ov{\vm}, \ov{\mU}, \ov{q})(t,\vx)\in \sC \qquad \text{ for all }(t,\vx)\in \big[0,\half T_\ts\big]\times \T^n,
\end{equation*}
where $\sC$ was defined in \eqref{eq:defn-C}. Lemma~\ref{lemma:estimate-F} yields that 
$$
	(\ov{\rho},\ov{\vm}, \ov{\mU}, \ov{q},\ov{\vF})(t,\vx)\in \sB_\delta \qquad \text{ for all }(t,\vx)\in \big[0,\half T_\ts\big]\times \T^n,
$$
with $\sB_\delta$ defined in Prop.~\ref{prop:U}~\ref{item:U-2}. Consequently, we deduce from Prop.~\ref{prop:U}~\ref{item:U-2} that 
$$
	(\ov{\vm}, \ov{\mU}, \ov{q},\ov{\vF})(t,\vx) \in \sU_{\ov{\rho}(t,\vx),\ov{Q}(t,\vx)} \qquad \text{ for all }(t,\vx)\in \big[0,\half T_\ts\big]\times \T^n,
$$ 
where\footnote{Recall from Sect.~\ref{subsubsec:moddata-s1} that $\delta$ corresponds to the case where $\frac{\ep^2}{48 n c}$ plays the role of $\ep$ in Prop.~\ref{prop:U}~\ref{item:U-2}.} 
\begin{equation} \label{eq:subsol1-Q}
	\ov{Q}(t,\vx) = \ov{q}(t,\vx) + \frac{\ep^2}{48 n c} .
\end{equation}
It is obvious from \eqref{eq:subsol1-reg} that 
\begin{equation} \label{eq:subsol1-reg-Q}
	\ov{Q} \in C^1\big(\big[0,\half T_\ts\big]\times \T^n;\R^+ \big),
\end{equation}

Next, in order to apply Prop.~\ref{prop:ci-data}, we fix $T_0\in \big(0,\tfrac{1}{4}T_\ts\big)$ such that 
\begin{equation} \label{eq:est-timeshift-rm}
	\|\ov{\rho}(T_0,\cdot) - \ov{\rho}(0,\cdot)\|_{L^\infty(\T^n)} < \tfrac{1}{4}\ep \qquad \text{ and } \qquad \|\ov{\vm}(T_0,\cdot) - \ov{\vm}(0,\cdot)\|_{L^\infty(\T^n)} < \tfrac{1}{4}\ep,
\end{equation}
which is possible by continuity. Moreover, we choose $\beta>0$ such that 
\begin{equation} \label{eq:beta}
	0<T_0-\beta<T_0<T_0+\beta<\tfrac{1}{4}T_\ts.
\end{equation}
Finally, we may apply Prop.~\ref{prop:ci-data}, where $\frac{\ep^2}{48}$ plays the role of $\ep$, to obtain 
\begin{equation} \label{eq:subsol1-mod-reg}
	(\widetilde{\vm},\widetilde{\mU},\widetilde{q},\widetilde{\vF}) \in \Cweak\big(\big[0,\half T_\ts\big];L^2(\T^n;\phase)\big) \cap C^1\big(\big(\big(0,\half T_\ts\big)\setminus\{T_0\}\big)\times \T^n;\phase\big)
\end{equation} 
with properties~\ref{item:idiK-pde}-\ref{item:idiK-m-close} stated in Prop.~\ref{prop:ci-data}. In particular, the tuple $(\ov{\rho},\widetilde{\vm},\widetilde{\mU},\widetilde{q}, \widetilde{\vF})$ solves the equations and inequalities \eqref{eq:eulerlin-mass}-\eqref{eq:eulerlin-energy}, and there hold
\begin{align}
	(\widetilde{\vm},\widetilde{\mU},\widetilde{q}, \widetilde{\vF})(t,\vx) &\in \sU_{\ov{\rho}(t,\vx), \ov{Q}(t,\vx)} \qquad \text{ for all }t\in \big(0,\half T_\ts\big)\setminus\{T_0\}\text{ and all }\vx \in \T^n, \label{eq:mod-subsol-U} \\
	(\widetilde{\vm},\widetilde{\mU},\widetilde{q}, \widetilde{\vF})(T_0,\vx) &\in \sK_{\ov{\rho}(T_0,\vx), \ov{Q}(T_0,\vx)} \qquad \text{ for a.e. }\vx \in \T^n, \label{eq:mod-subsol-K} \\
	(\widetilde{\vm},\widetilde{\mU},\widetilde{q}, \widetilde{\vF})(t,\vx) &= (\ov{\vm},\ov{\mU},\ov{q}, \ov{\vF})(t,\vx) \qquad \text{ for all }t\in \big(\tfrac{1}{4} T_\ts,\half T_\ts\big)\text{ and all }\vx \in \T^n. \label{eq:mod-subsol-outside} 
\end{align}
Note that \eqref{eq:mod-subsol-outside} follows from Prop.~\ref{prop:ci-data}~\ref{item:idiK-outside} and the choice of $\beta$, see \eqref{eq:beta}.

From property~\ref{item:idiK-m-close} we infer
\begin{equation} \label{eq:101}
	\Big\|\widetilde{\vm}(T_0,\cdot) - \ov{\vm}(T_0,\cdot)\Big\|_{L^2(\T^n)}^2 \leq \big\|\widetilde{\vm}(T_0,\cdot)\big\|_{L^2(\T^n)}^2 - \big\|\ov{\vm}(T_0,\cdot)\big\|_{L^2(\T^n)}^2 + \frac{\ep^2}{48},
\end{equation}
and \eqref{eq:mod-subsol-K} yields\footnote{Recall the definition of $\sK_{\rho,Q}$ in \eqref{eq:K}.} with the help of \eqref{eq:subsol1-Q}
\begin{align}
	\big|\widetilde{\vm}(T_0,\vx)\big|^2 &= n \big( \widetilde{q}(T_0,\vx) - p(\ov{\rho}(T_0,\vx)) \big) \ov{\rho}(T_0,\vx) \notag \\
	&\leq n \big( \ov{Q}(T_0,\vx) - p(\ov{\rho}(T_0,\vx)) \big) \ov{\rho}(T_0,\vx) \notag \\
	&= n \left( \ov{q}(T_0,\vx) + \frac{\ep^2}{48 n c} - p(\ov{\rho}(T_0,\vx)) \right) \ov{\rho}(T_0,\vx) \qquad\text{ for a.e. }\vx\in \T^n . \label{eq:102}
\end{align}
Taking the trace in \eqref{eq:matrix-lower-bound}, we find 
\begin{equation} \label{eq:103}
	- \big|\ov{\vm}(T_0,\vx)\big|^2 \leq \left( \frac{\ep^2}{48 c} + n \Big( p(\ov{\rho}(T_0,\vx)) - \ov{q}(T_0,\vx) \Big) \right) \ov{\rho}(T_0,\vx) \qquad\text{ for all }\vx\in \T^n .
\end{equation}
Combining \eqref{eq:102} with \eqref{eq:103}, we obtain 
\begin{equation} \label{eq:104}
	\big|\widetilde{\vm}(T_0,\vx)\big|^2 - \big|\ov{\vm}(T_0,\vx)\big|^2 \leq \left( \frac{\ep^2}{48 c} + \frac{\ep^2}{48 c} \right) \ov{\rho}(T_0,\vx)\ \leq \ \frac{\ep^2}{24}  \qquad\text{ for a.e. }\vx\in \T^n . 
\end{equation}
Plugging \eqref{eq:104} into \eqref{eq:101}, we infer 
\begin{equation} \label{eq:est-data-m}
	\Big\|\widetilde{\vm}(T_0,\cdot) - \ov{\vm}(T_0,\cdot)\Big\|_{L^2(\T^n)}^2 \leq \frac{\ep^2}{24} + \frac{\ep^2}{48} \ = \ \frac{\ep^2}{16}. 
\end{equation}

Let us set 
\begin{equation} \label{eq:defn-mod-data}
	\rho_{0,\ep}:= \ov{\rho}(T_0,\cdot), \quad \text{ and } \quad \vm_{0,\ep} := \widetilde{\vm}(T_0,\cdot).
\end{equation}
From \eqref{eq:est-timeshift-rm}, the fact that $r< \tfrac{1}{4} \ep$, as well as \eqref{eq:est-molli-rm}, we see that 
\begin{align*}
	&\|\rho_{0,\ep} - \rho_0\|_{L^\infty(\T^n)} \\
	&\leq \| \ov{\rho}(T_0,\cdot) - \ov{\rho}(0,\cdot) \|_{L^\infty(\T^n)} + \| \ov{\rho}(0,\cdot) - \rho_1(0,\cdot) \|_{L^\infty(\T^n)} + \| \rho_1(0,\cdot) - \rho_0 \|_{L^\infty(\T^n)} \\
	& < \tfrac{1}{4} \ep + \tfrac{1}{4} \ep + \tfrac{1}{4} \ep \ <\ \ep.
\end{align*}
Similarly, we obtain from \eqref{eq:est-data-m}, \eqref{eq:est-timeshift-rm} and \eqref{eq:est-molli-rm}
\begin{align*} 
	&\|\vm_{0,\ep} - \vm_0\|_{L^2(\T^n)} \\
	&\leq \| \widetilde{\vm}(T_0,\cdot) - \ov{\vm}(T_0,\cdot) \|_{L^2(\T^n)} + \| \ov{\vm}(T_0,\cdot) - \ov{\vm}(0,\cdot) \|_{L^\infty(\T^n)} + \| \ov{\vm}(0,\cdot) - \vm_0 \|_{L^\infty(\T^n)} \\
	& < \tfrac{1}{4} \ep + \tfrac{1}{4} \ep + \tfrac{1}{4} \ep \ <\ \ep.
\end{align*}
as well as 
\begin{align*} 
	\|\vm_{0,\ep}\|_{L^\infty(\T^n)} &= \|\widetilde{\vm}(T_0,\cdot) \|_{L^\infty(\T^n)} \leq \|\ov{\vm}(T_0,\cdot) \|_{L^\infty(\T^n)} + \frac{\ep}{\sqrt{24}} \\
	&\leq \|\ov{\vm}(T_0,\cdot) - \ov{\vm}(0,\cdot) \|_{L^\infty(\T^n)} + \|\ov{\vm}(0,\cdot)- \vm_0 \|_{L^\infty(\T^n)} + \|\vm_0\|_{L^\infty(\T^n)} + \frac{\ep}{\sqrt{24}} \\
	& < \|\vm_0\|_{L^\infty(\T^n)} + \frac{\ep}{4} + \frac{\ep}{4} + \frac{\ep}{\sqrt{24}} \ <\ \|\vm_0\|_{L^\infty(\T^n)} + \ep ,
\end{align*}
where we used \eqref{eq:104} instead of \eqref{eq:est-data-m}.

\subsection{The modified initial datum is wild} \label{subsec:proof-wildness}

In order to prove Thm.~\ref{thm:density-2}, it remains to show that $(\rho_{0,\ep},\vm_{0,\ep})$ constructed in Sect.~\ref{subsec:proof-modification} above is a wild initial datum in the sense of Defn.~\ref{defn:wild-data}. To this end, we construct a subsolution
$$
	(\un{\rho}, \un{\vm},\un{\mU},\un{q},\un{\vF}) \in C^\infty([0,T+T_\ts]\times \T^n;\R^+ \times \phase)
$$
for a given final time $T>0$. We keep in mind that the values for $T_\ts , c, \delta, \sigma, r, T_0$ have been already chosen in Sect.~\ref{subsec:proof-modification} above. We will show in Sect.~\ref{subsubsec:infsol-s6} that the subsolution $(\ov{\rho}, \widetilde{\vm},\widetilde{\mU},\widetilde{q},\widetilde{\vF})$, which we found in Sect.~\ref{subsubsec:moddata-s6} above and which lives on the short time interval $\big[T_0,\half T_\ts\big]$, and the subsolution $(\un{\rho}, \un{\vm},\un{\mU},\un{q},\un{\vF})$, which exists on the long run, i.e.~on $[0,T+T_\ts]$, can be glued together to obtain a subsolution $(\rho^\ast,\vm^\ast,\mU^\ast,q^\ast,\vF^\ast)$ on the time interval $[T_0,T+T_0]$ which satisfies the desired initial condition. 

For the construction of $(\rho^\ast,\vm^\ast,\mU^\ast,q^\ast,\vF^\ast)$, we again proceed in six steps, cf Sect.~\ref{subsec:proof-outline}.

\subsubsection{Step 1: Dissipative weak solution} \label{subsubsec:infsol-s1}

According to Prop.~\ref{prop:ex-dws}, there exists a dissipative weak solutions 
$$
	\rho_\td \in \Cweak\big([0,T + 2 T_\ts];L^\gamma(\T^n;\R^+_0)\big), \qquad \vm_\td\in \Cweak\big([0,T + 2 T_\ts];L^{\frac{2\gamma}{\gamma+1}}(\T^n;\R^n)\big)
$$	
of \eqref{eq:euler-mass}-\eqref{eq:initial} with initial datum $(\rho_0,\vm_0)$. Here $T_\ts>0$ is the existence time of the strong solution for the same initial data $(\rho_0,\vm_0)$, see Sect.~\ref{subsubsec:moddata-s1} above. Combining this with Prop.~\ref{prop:ex-strong} and Prop.~\ref{prop:wsu}, we even have 
$$
	\rho_\td \in \Cweak\big([-T_\ts,T+ 2T_\ts];L^\gamma(\T^n;\R^+_0)\big), \qquad \vm_\td\in \Cweak\big([-T_\ts,T+ 2T_\ts];L^{\frac{2\gamma}{\gamma+1}}(\T^n;\R^n)\big)
$$	
and 
\begin{equation} \label{eq:dws=strong-small-times}
	(\rho_\td,\vm_\td)(t,\cdot) = (\rho_\ts,\vm_\ts)(t,\cdot)\qquad \text{ for all } t\in [-T_\ts,T_\ts]. 
\end{equation}
In particular, 
\begin{equation} \label{eq:defect=zero-small-times}
	\mathfrak{R}(t,\cdot)= 0 \qquad \text{ for a.e. } t\in (-T_\ts,T_\ts). 
\end{equation}

Analogous to \eqref{eq:strong-U-q}, we will use the abbreviation 
\begin{align}
	\mU_\td &:= 1_{\rho_\td>0} \frac{\vm_\td \otimes \vm_\td}{\rho_\td} - 1_{\rho_\td>0} \frac{|\vm_\td|^2}{n\rho_\td} \id & q_\td &:= 1_{\rho_\td>0} \frac{|\vm_\td|^2}{n \rho_\td} + p(\rho_\td) . \label{eq:dws-U-q}
\end{align}

\subsubsection{Step 2: Mollification of the dissipative weak solution and definition of $\un{\vm}$} \label{subsubsec:infsol-s2}

Similarly to Sect.~\ref{subsubsec:moddata-s2}, we mollify $(\rho_\td,\vm_\td)$ with the mollification parameter $0<\sigma < \half T_\ts$ found in Sect.~\ref{subsubsec:moddata-s2} above. This way we obtain 
\begin{align*} 
	\rho_2 &:= \rho_\td \ast \phi_{\sigma} , &
	\un{\vm} &:= \vm_\td \ast \phi_{\sigma} ,
\end{align*}
where 
\begin{equation} \label{eq:subsol2-reg-m}
	(\rho_2,\un{\vm}) \in C^\infty\big([0,T + T_\ts]\times \T^n;\R_0^+\times \R^n\big).
\end{equation}
Let us emphasise that -- in contrast to $\rho_1$ in Sect.~\ref{subsubsec:moddata-s2} above -- $\rho_2(t,\vx)$ does not need to be strictly larger than zero. This is the reason why we add $r>0$ to the density in Sect.~\ref{subsubsec:infsol-s3} below, see also Difficulty 2 in Sect.~\ref{subsec:proof-outline}. 

The pair $(\rho_2,\un{\vm})$ solves the following PDEs.

\begin{lemma} \label{lemma:pdes-mollified-dws}
	The PDEs 
	\begin{align} 
		\partial_t \rho_2 + \Div \un{\vm} &= 0, \label{eq:subsol2-mass-nolowerbound} \\
		\partial_t \un{\vm} + \Div \bigg[\underbrace{\bigg(1_{\rho_\td>0}\frac{\vm_\td \otimes \vm_\td}{\rho_\td} + p(\rho_\td) \id\bigg)}_{=\mU_\td + q_\td\id} \ast \phi_{\sigma} + R\bigg] &= 0 \label{eq:subsol2-mom-nolambda} 
	\end{align}
	hold pointwise in $[0,T+ T_\ts]\times \T^n$, where 
	\begin{equation} \label{eq:defn-R}
		R(t,\vx):= \int_{-T_\ts}^{T+2 T_\ts} \int_{\T^n} \phi_\sigma (t-s, \vx-\vy) \dR(s) \ds .
	\end{equation}
	There hold 
	\begin{itemize} 
		\item $R\in C^\infty\big([0,T+T_\ts]\times \T^n; \sym{n}\big)$, 
		\item $R(t,\vx)$ is positive semi-definite for all $[0,T+T_\ts]\times \T^n$, and 
		\item $R(t,\vx)=0$ for all $(t,\vx)\in \big[0,\half T_\ts\big]\times \T^n$.  
	\end{itemize}
\end{lemma}

\begin{proof}
	For arbitrary $(s,\vy)\in [0,T+ T_\ts]\times \T^n$, we choose the test function in \eqref{eq:dws-mass} to be $\phi_\sigma (s-t, \vy - \vx)$. Note that, since the dissipative weak solution $(\rho_\td,\vm_\td)$ exists on the time interval $[-T_\ts,T+2T_\ts]$, we may integrate in \eqref{eq:dws-mass} from $-T_\ts$ to $T+2T_\ts$ instead of from $0$ to $\tau$. This yields that \eqref{eq:subsol2-mass-nolowerbound} holds in $(s,\vy)$. Proceeding analogously with \eqref{eq:dws-mom}, we find \eqref{eq:subsol2-mom-nolambda}. 
	
	We deduce from \eqref{eq:space-fracR} that 
	$$
		(t,\vx) \mapsto \int_{\T^n} \varphi_\sigma (\vx-\vy) \dR(t)\quad \in \quad L^\infty\big((-T_\ts,T + 2T_\ts);C^\infty (\T^n; \sym{n})\big) .
	$$
	Hence we can mollify this mapping in time to obtain that
	\begin{align*}
		(t,\vx) \mapsto R(t,\vx) &= \int_{-T_\ts}^{T+2 T_\ts} \int_{\T^n} \phi_\sigma (t-s, \vx-\vy) \dR(s) \ds \\
		&= \int_{-T_\ts}^{T+2 T_\ts} \psi_\sigma(t-s) \int_{\T^n} \varphi_\sigma (\vx-\vy) \dR(s) \ds 
	\end{align*}
	lies in $C^\infty\big([0,T+T_\ts]\times \T^n; \sym{n}\big)$ as desired. 
	
	The fact that $R(t,\vx)$ is positive semi-definite for all $[0,T+T_\ts]\times \T^n$ follows immediately from the definition of the space $\mathcal{M}^+ (\T^n; \sym{n})$, see Sect.~\ref{subsec:bb-dws}. 
	
	Finally we note that \eqref{eq:defect=zero-small-times} implies $R(t,\vx)=0$ for all $(t,\vx)\in \big[0,\half T_\ts\big]\times \T^n$.  
\end{proof} 

Next, we prove an upper bound for $\rho_2$. To this end, we deduce from Young's inequality 
\begin{align}
	\| \rho_2 \|_{L^\infty((0,T+ T_\ts)\times \T^n)} &= \| \rho_\td \ast \phi_\sigma \|_{L^\infty((0,T + T_\ts)\times \T^n)} = \| (\rho_\td \ast_\vx \varphi_\sigma) \ast_t \psi_\sigma \|_{L^\infty((0,T + T_\ts )\times \T^n)} \notag \\
	&\leq \| \rho_\td \ast_\vx \varphi_\sigma \|_{L^\infty((0,T + T_\ts )\times \T^n)} \| \psi_\sigma \|_{L^1((0,T + T_\ts))} = \| \rho_\td \ast_\vx \varphi_\sigma \|_{L^\infty((0,T + T_\ts )\times \T^n)} , \label{eq:201}
\end{align}
and 
\begin{equation} \label{eq:202}
	\left\| \rho_\td(t,\cdot) \ast_\vx \varphi_\sigma \right\|_{L^\infty(\T^n)} \leq \left\| \rho_\td(t,\cdot) \right\|_{L^1(\T^n)} \left\| \varphi_\sigma \right\|_{L^\infty(\T^n)} \leq  \left(\frac{1}{\sigma}\right)^n \int_{\T^n} \rho_0 \dx\qquad \text{ for all } t\in [0,T+T_\ts].
\end{equation}
Here we have used that the total mass of a dissipative weak solution is conserved, which follows from \eqref{eq:dws-mass}, i.e.
$$
	\int_{\T^n} \rho_\td (t,\vx) \dx = \int_{\T^n} \rho_0 \dx \qquad \text{ for all } t\in [0,T+T_\ts].
$$
Plugging \eqref{eq:202} into \eqref{eq:201}, we find 
\begin{equation} \label{eq:bound-rho2}
	\rho_2(t,\vx) \leq \left(\frac{1}{\sigma}\right)^n \int_{\T^n} \rho_0 \dx, \qquad \text{ for all } (t,\vx)\in [0,T+T_\ts] \times \T^n,
\end{equation}
an estimate which will be used below.

\subsubsection{Step 3: A lower bound for the density and definition of $\un{\rho}$} \label{subsubsec:infsol-s3}

Analogously to Sect.~\ref{subsubsec:moddata-s3}, we set 
$$
	\un{\rho} := \rho_2 + r = \rho_\td \ast \phi_{\sigma} + r,
$$
using the parameter $0<r<\tfrac{1}{4}\ep$ fixed in Sect.~\ref{subsubsec:moddata-s3}. Using \eqref{eq:subsol2-reg-m} we find 
\begin{equation} \label{eq:subsol2-reg-rho}
	\un{\rho}\in C^\infty\big([0,T + T_\ts]\times \T^n;\R^+\big).
\end{equation}
Obviously, 
$$
	\un{\rho}(t,\vx)\geq r >0 \qquad\text{ for all }(t,\vx)\in [0,T + T_\ts]\times \T^n, 
$$
and, due to \eqref{eq:subsol2-mass-nolowerbound}, the PDE
\begin{equation} \label{eq:subsol2-mass} 
	\partial_t \un{\rho} + \Div \un{\vm} = 0, 
\end{equation}
holds pointwise in $[0,T + T_\ts]\times \T^n$.

\subsubsection{Step 4: Definition and properties of $\un{\mU}$ and $\un{q}$} \label{subsubsec:infsol-s4}

We set
\begin{align*}
	\un{\mU} &:= \mU_\td \ast \phi_{\sigma} + R - \frac{\tr(R)}{n} \id ,&
	\un{q} &:= q_\td \ast \phi_{\sigma} + \frac{\tr(R)}{n} + \un{\lambda}(t),
\end{align*} 
with $\mU_\td$ and $q_\td$ defined in \eqref{eq:dws-U-q}, $R$ fixed in \eqref{eq:defn-R}, and $\un{\lambda}(t)$ given by 
\begin{align*} 
	\un{\lambda}(t) &:= - \int_{\T^n} \left( q_\td \ast \phi_{\sigma} + \frac{\tr(R)}{n} - p(\un{\rho}) + \frac{2}{n} P(\un{\rho}) \right) \dx + \int_{\T^n}  \frac{2}{n} \left( \frac{|\vm_0|^2}{2\rho_0} + P(\rho_0) \right) \dx \\
	&\qquad + \frac{\delta}{8 n c^2} + \frac{1}{c^2} \exp( - ct ). 
\end{align*}
Here $c$ and $\delta$ are the parameters from Sect.~\ref{subsubsec:moddata-s1}.

Using Lemma~\ref{lemma:pdes-mollified-dws}, we find
\begin{equation} \label{eq:subsol2-reg-Uq}
	(\un{\mU},\un{q})\in C^\infty\big([0, T+ T_\ts]\times \T^n; \symz{n} \times \R\big),
\end{equation} 
and, according to \eqref{eq:subsol2-mom-nolambda}, the PDE 
\begin{equation} \label{eq:subsol2-mom} 
	\partial_t \un{\vm} + \Div (\un{\mU} + \un{q} \id) = 0 
\end{equation}
holds pointwise in $[0,T + T_\ts]\times \T^n$. 

In addition to that, we obtain from \eqref{eq:dws=strong-small-times} and Lemma~\ref{lemma:pdes-mollified-dws} that 
\begin{equation} \label{eq:un=ov-rho-m-U-q}
	(\un{\rho},\un{\vm},\un{\mU},\un{q})(t,\vx) = (\ov{\rho},\ov{\vm},\ov{\mU},\ov{q})(t,\vx) \qquad \text{ for all }(t,\vx)\in \big[0,\half T_\ts\big] \times \T^n,
\end{equation}
as well as
$$
	\un{\lambda}(t) = \ov{\lambda}(t)  \qquad \text{ for all }t\in \big[0,\half T_\ts\big] .
$$

We observe that the lower bound shown in Lemma~\ref{lemma:estimate-lambda-ov}, still holds for $\un{\lambda}$ on the large time interval $[0,T + T_\ts]$:

\begin{lemma} \label{lemma:estimate-lambda-un} 
	There holds
	$$
		\un{\lambda}(t) \geq \frac{(\gamma-1) \delta}{16c^2} + \frac{1}{c^2} \exp( - ct ) > 0 \qquad \text{ for all } t\in [0,T+T_\ts].
	$$
\end{lemma} 

\begin{proof} 
	We proceed analogously to the proof of Lemma~\ref{lemma:estimate-lambda-ov}. First we observe that, analogously to \eqref{eq:total-mass-mollification}, there holds
	$$
		\int_{\T^n} \int_{\T^n} \varphi_\sigma (\vx - \vy) \,{\rm d}(\tr \mathfrak{R}) \dx = \int_{\T^n} \int_{\T^n} \varphi_\sigma (\vx - \vy) \dx \,{\rm d}(\tr \mathfrak{R}) = \int_{\T^n} \,{\rm d}(\tr \mathfrak{R}).
	$$
	Using this fact together with \eqref{eq:total-mass-mollification} we deduce from the energy inequality \eqref{eq:dws-energy} 
	\begin{align*}
		&\int_{\T^n} \left( 1_{\rho_\td>0} \frac{|\vm_\td|^2}{2\rho_\td} + P(\rho_\td) \right) \ast_\vx \varphi_\sigma \dx + \frac{1}{2} \int_{\T^n} \int_{\T^n} \varphi_\sigma (\vx - \vy) \,{\rm d}(\tr \mathfrak{R}) \dx \\
		&= \int_{\T^n} \left( 1_{\rho_\td>0} \frac{|\vm_\td|^2}{2\rho_\td} + P(\rho_\td) \right) \dx + \frac{1}{2} \int_{\T^n} \,{\rm d}(\tr \mathfrak{R}) \\
		&\leq \int_{\T^n} \left( \frac{|\vm_0|^2}{2\rho_0} + P(\rho_0) \right) \dx
	\end{align*}
	for all $t\in [0, T+ 2T_\ts]$. According to \eqref{eq:dws=strong-small-times}, \eqref{eq:defect=zero-small-times} and the fact that strong solutions conserve the total energy, the above estimate even holds for $t\in [-T_\ts, 0]$. Hence, by mollifying in time, we end up with 
	\begin{align*}
		&\int_{\T^n} \left( 1_{\rho_\td>0} \frac{|\vm_\td|^2}{2\rho_\td} + P(\rho_\td) \right) \ast \phi_\sigma \dx + \frac{1}{2} \int_{\T^n} \tr(R) \dx \\
		&= \left(\int_{\T^n} \left( 1_{\rho_\td>0} \frac{|\vm_\td|^2}{2\rho_\td} + P(\rho_\td) \right) \ast_\vx \varphi_\sigma \dx \right) \ast_t \psi_\sigma + \frac{1}{2} \left( \int_{\T^n} \int_{\T^n} \varphi_\sigma (\vx-\vy) \,{\rm d}(\tr \mathfrak{R}) \dx \right) \ast_t \psi_\sigma \\
		&\leq \int_{\T^n} \left( \frac{|\vm_0|^2}{2\rho_0} + P(\rho_0) \right) \dx
	\end{align*}
	for all $t\in [0, T+ T_\ts]$. Recalling \eqref{eq:dws-U-q}, this leads to  
	\begin{equation} \label{eq:301} 
		- \int_{\T^n} \left( \left( q_\td - p(\rho_\td) + \frac{2}{n} P(\rho_\td) \right) \ast \phi_{\sigma} + \frac{\tr(R)}{n} \right) \dx + \int_{\T^n}  \frac{2}{n} \left( \frac{|\vm_0|^2}{2\rho_0} + P(\rho_0) \right) \dx \ \geq \ 0
	\end{equation}
	for $t\in[0,T + T_\ts]$. 
	
	Since $\tfrac{2}{n} \, \tfrac{1}{\gamma-1} - 1\geq 0$, and, according to Jensen's inequality (Lemma~\ref{lemma:jensen-real-valued}),
	$$
		\int_{\T^n} \Big( p(\rho_\td) \ast \phi_{\sigma} - p(\rho_2) \Big) \dx \geq 0, 
	$$
	as well as 
	$$
		\int_{\T^n} \Big( p(\rho_2) - p(\un{\rho}) \Big) \dx \geq - \frac{(\gamma-1) \delta}{16 c^2} 
	$$ 
	which follows from  \eqref{eq:est-shift-pressure} and \eqref{eq:bound-rho2}, we may infer
	\begin{align} 
		&- \int_{\T^n} \Big(- p(\un{\rho}) + \tfrac{2}{n} P(\un{\rho}) \Big) \dx + \int_{\T^n} \Big( - p(\rho_\td) + \tfrac{2}{n} P(\rho_\td) \Big) \ast \phi_{\sigma} \dx \notag \\
		&= \Big( \tfrac{2}{n} \, \tfrac{1}{\gamma-1} - 1 \Big) \left( \int_{\T^n} \Big( p(\rho_\td) \ast \phi_{\sigma} - p(\rho_2) \Big) \dx + \int_{\T^n} \Big( p(\rho_2) - p(\un{\rho}) \Big) \dx \right) \notag \\
		&\geq - \Big( \tfrac{2}{n} \, \tfrac{1}{\gamma-1} - 1 \Big) \frac{(\gamma-1) \delta}{16 c^2} \label{eq:302}
	\end{align}
	for all $t\in [0,T+T_\ts]$. 
	
	Plugging \eqref{eq:301} and \eqref{eq:302} into the definition of $\un{\lambda}$, we find
	$$
		\un{\lambda}(t) \geq - \left( \frac{2}{n} \, \frac{1}{\gamma-1} - 1 \right) \frac{(\gamma-1) \delta}{16 c^2} + \frac{\delta}{8 n c^2} + \frac{1}{c^2} \exp( - ct ) = \frac{(\gamma-1) \delta}{16 c^2} + \frac{1}{c^2} \exp( - ct ) 
	$$
	as desired.
\end{proof}

\begin{rem} \label{rem:gamma-2}
	In the proof of Lemma~\ref{lemma:estimate-lambda-un} it was essential that $\gamma \leq 1 + \frac{2}{n}$, and that the estimate \eqref{eq:est-shift-pressure} still holds for $\rho_2$. The latter follows from \eqref{eq:bound-rho2}. 
\end{rem}

Next, we set 
\begin{align*} 
	\un{\mM}_1(t,\vx) &:= \frac{\un{\vm}\otimes\un{\vm}}{\rho_2} - \frac{\un{\vm}\otimes\un{\vm}}{\un{\rho}} , \\
	\un{\mM}_2(t,\vx) &:= \big( p(\un{\rho}) - p(\rho_2) \big) \id, \\
	\un{\mM}_3(t,\vx) &:= \left(1_{\rho_\td>0}\frac{\vm_\td \otimes \vm_\td}{\rho_\td} + p(\rho_\td) \id \right)\ast \phi_{\sigma} - \left( \frac{\un{\vm}\otimes\un{\vm}}{\rho_2} + p(\rho_2)\id \right) 
\end{align*} 
for $(t,\vx)\in [0,T + T_\ts]\times \T^n$, similarly to $\ov{\mM}_i$ ($i=1,2,3$) in Sect.~\ref{subsubsec:moddata-s4} above. Some of the bounds obtained in Lemma~\ref{lemma:estimate-matrices-ov} are still valid for $\un{\mM}_i$ ($i=1,2,3$) on the large time interval $[0,T + T_\ts]$, more precisely:

\begin{lemma} \label{lemma:estimate-matrices-un}
	For any $(t,\vx)\in [0,T + T_\ts]\times \T^n$ the following estimates hold:
	\begin{align*} 
		\un{\mM}_1 &\geq 0, &
		\un{\mM}_2 &\leq \frac{(\gamma-1) \delta}{16c^2} \id, &
		\un{\mM}_3 & \geq 0. 
	\end{align*}
\end{lemma}

\begin{proof}
	We follow the proof of Lemma~\ref{lemma:estimate-matrices-ov}. Again, $\frac{1}{\rho_2} - \frac{1}{\rho_2 + r}\geq 0$ implies $\un{\mM}_1 \geq 0$. As the estimate \eqref{eq:est-shift-pressure} still holds for $\rho_2$ due to \eqref{eq:bound-rho2}, we deduce $\un{\mM}_2 \leq \frac{(\gamma-1) \delta}{16c^2} \id$. Finally, $\un{\mM}_3 \geq 0$ is a consequence of Jensen's inequality (Lemma~\ref{lemma:jensen-matrix-valued}). 
\end{proof} 

Using Lemmas~\ref{lemma:pdes-mollified-dws}, \ref{lemma:estimate-lambda-un} and \ref{lemma:estimate-matrices-un}, we compute for any $(t,\vx)\in [0,T+T_\ts]\times \T^n$ 
\begin{align} 
	&\frac{\un{\vm}\otimes\un{\vm}}{\un{\rho}} + p(\un{\rho})\id - \un{\mU} - \un{q}\id \notag \\
	&= -\un{\mM}_1 + \un{\mM}_2 - \un{\mM}_3 - R - \un{\lambda}(t) \id \notag \\
	&\leq \frac{(\gamma-1) \delta}{16c^2} \id - \frac{(\gamma-1) \delta}{16c^2} \id - \frac{1}{c^2} \exp( - ct ) \id\ \leq \ - \frac{1}{c^2} \exp( - c (T+T_\ts)) \id, \label{eq:subsol2-matrix-neg-def}
\end{align}
similarly to \eqref{eq:matrix-upper-bound}.

\subsubsection{Step 5: Definition of $\un{\vF}$} \label{subsubsec:infsol-s5}

Goal of this subsection is to define $\un{\vF}$. It will turn out that this is much simpler than the definition of $\ov{\vF}$ carried out in Sect.~\ref{subsubsec:moddata-s5} above. We compute 
\begin{align*} 
	&\int_{\T^n} \partial_t \Big( \tfrac{n}{2} \big(\un{q} - p(\un{\rho})\big) + P(\un{\rho}) \Big) \dx \\
	&= \partial_t \int_{\T^n} \left( \frac{n}{2} \left(q_\td \ast \phi_{\sigma} + \frac{\tr (R)}{n} + \un{\lambda}(t) - p(\un{\rho})\right) + P(\un{\rho}) \right) \dx \\
	&= \partial_t \bigg[ \int_{\T^n} \left( \frac{n}{2} \left(q_\td \ast \phi_{\sigma} + \frac{\tr (R)}{n} - p(\un{\rho})\right) + P(\un{\rho}) \right) \dx + \frac{n}{2} \un{\lambda}(t) \bigg] \\
	&= \partial_t \bigg[ \int_{\T^n} \left( \frac{n}{2} \left(q_\td \ast \phi_{\sigma} + \frac{\tr (R)}{n} - p(\un{\rho})\right) + P(\un{\rho}) \right) \dx \\
	&\qquad - \frac{n}{2} \int_{\T^n} \left( q_\td \ast \phi_{\sigma} + \frac{\tr(R)}{n} - p(\un{\rho}) + \frac{2}{n} P(\un{\rho}) \right) \dx \bigg] - \frac{n}{2c} \exp(-ct) \ \leq\ 0.
\end{align*} 
Hence, Cor.~\ref{cor:poisson-not-mf} yields\footnote{Note that $w$ is indeed smooth because $\un{\rho},\un{q}$ are, see \eqref{eq:subsol2-reg-rho} and \eqref{eq:subsol2-reg-Uq}, respectively.} $w\in C^\infty([0,T + T_\ts]\times \T^n)$ which solves 
\begin{equation} \label{eq:subsol2-poisson-energy}
	\Lap w \geq \partial_t \Big( \tfrac{n}{2} \big(\un{q} - p(\un{\rho})\big) + P(\un{\rho}) \Big).
\end{equation}

We now simply define
$$
	\un{\vF} := -\Grad w .
$$
Obviously, there holds 
\begin{equation} \label{eq:subsol2-reg-F}
	\un{\vF} \in C^\infty\big([0,T + T_\ts]\times \T^n; \R^n\big),
\end{equation}
and \eqref{eq:subsol2-poisson-energy} implies that the inequality 
\begin{equation} \label{eq:subsol2-energy} 
	\partial_t \Big( \tfrac{n}{2} \big(\un{q} - p(\un{\rho}) \big) + P(\un{\rho})\Big) + \Div \un{\vF} \leq 0 
\end{equation}
holds pointwise in $[0,T + T_\ts]\times \T^n$.

\subsubsection{Step 6: Infinitely many solutions due to Prop.~\ref{prop:ci-nonuniqueness}}  \label{subsubsec:infsol-s6} 

In the preceding steps, we have constructed a tuple of functions 
\begin{equation} \label{eq:subsol2-reg}
	(\un{\rho},\un{\vm},\un{\mU},\un{q},\un{\vF})\in C^\infty\big([0,T + T_\ts]\times \T^n;\R^+ \times \phase\big),
\end{equation}
see \eqref{eq:subsol2-reg-rho}, \eqref{eq:subsol2-reg-m}, \eqref{eq:subsol2-reg-Uq} and \eqref{eq:subsol2-reg-F}. According to \eqref{eq:subsol2-mass}, \eqref{eq:subsol2-mom} and \eqref{eq:subsol2-energy}, this tuple fulfills the linear system \eqref{eq:eulerlin-mass}-\eqref{eq:eulerlin-energy} for all $(t,\vx)\in (0,T + T_\ts)\times \T^n$.

By continuity, we know that 
$$
	(\un{\rho},\un{\vm},\un{\mU},\un{q},\un{\vF})(t,\vx)\in \un{\sC}\qquad \text{ for all }(t,\vx)\in [0,T + T_\ts]\times \T^n,
$$ 
with a suitable compact set $\un{\sC}\subset \R^+\times \phase$. Due to \eqref{eq:subsol2-matrix-neg-def} the assertion \eqref{eq:matrix-neg-def} holds on $\un{\sC}$. Hence, according to Prop.~\ref{prop:U}~\ref{item:U-1}, there exists a constant $\un{Q}>0$ such that 
\begin{equation} \label{eq:subsol2-subscond}
	(\un{\vm}, \un{\mU}, \un{q},\un{\vF})(t,\vx) \in \sU_{\un{\rho}(t,\vx),\un{Q}} \qquad \text{ for all }(t,\vx)\in [0,T + T_\ts]\times \T^n.
\end{equation}
Without loss of generality, we may assume that 
\begin{equation} \label{eq:order-Qun-vs-Qov}
	\ov{Q}(t,\vx) \leq \un{Q}\qquad \text{ for all }(t,\vx)\in \big[ 0, \half T_\ts \big] \times \T^n, 
\end{equation}	
by taking $\un{Q}$ sufficiently large, recall \eqref{eq:subsol1-reg-Q}. 

Finally, we define a tuple of functions 
\begin{equation} \label{eq:subsolstar-rho-m-U-q}
	(\rho^\ast,\vm^\ast,\mU^\ast,q^\ast)(t,\vx) := \left\{ \begin{array}{ll} (\ov{\rho},\widetilde{\vm},\widetilde{\mU},\widetilde{q})(t,\vx) & \text{ for }(t,\vx)\in \big[T_0, \half T_\ts\big) \times \T^n, \\ (\un{\rho},\un{\vm},\un{\mU},\un{q})(t,\vx) & \text{ for } (t,\vx) \in \big(\tfrac{1}{4} T_\ts, T + T_0\big] \times \T^n, \end{array} \right. 
\end{equation}
with the time $T_0$ from Sect.~\ref{subsubsec:moddata-s6}. Note that $T_0< \tfrac{1}{4} T_\ts <T_\ts$, and 
\begin{equation} \label{eq:un=til-rho-m-U-q}
	(\ov{\rho},\widetilde{\vm},\widetilde{\mU},\widetilde{q})(t,\vx) = (\ov{\rho},\ov{\vm},\ov{\mU},\ov{q})(t,\vx) = (\un{\rho},\un{\vm},\un{\mU},\un{q})(t,\vx)\qquad \text{ for all } (t,\vx)\in \big( \tfrac{1}{4} T_\ts ,  \half T_\ts\big) \times \T^n 
\end{equation}
due to \eqref{eq:mod-subsol-outside} and \eqref{eq:un=ov-rho-m-U-q}. Therefore, we may deduce from \eqref{eq:subsol1-reg-rho}, \eqref{eq:subsol1-mod-reg} and \eqref{eq:subsol2-reg} that 
$$
	\rho^\ast \in C^1\big( [T_0, T + T_0] \times \T^n; \R^+\big),  
$$
as well as 
\begin{align*}
	(\vm^\ast, \mU^\ast, q^\ast) &\in \Cweak\big([T_0,T + T_0];L^2(\T^n;\R^n \times \symz{n} \times \R )\big) \\
	&\qquad \cap C^1\big((T_0,T + T_0)\times \T^n;\R^n \times \symz{n} \times \R \big).
\end{align*} 

In order to define $\vF^\ast$ and $Q^\ast$, we choose two non-increasing cut-off functions in time $\chi_1, \chi_2 \in C^\infty\big([0,\infty);[0,1]\big)$ with 
\begin{align*}
	\chi_1(t) &= \left\{ \begin{array}{cc} 1 & \text{ if } t\leq \tfrac{3}{8} T_\ts, \\ 0 & \text{ if } t\geq \half T_\ts, \end{array} \right. &
	\chi_2(t) &= \left\{ \begin{array}{cc} 1 & \text{ if } t\leq \tfrac{1}{4} T_\ts, \\ 0 & \text{ if } t\geq \tfrac{3}{8} T_\ts. \end{array} \right. 
\end{align*}
We may then set 
\begin{align*}
	\vF^\ast &:= \chi_1 \widetilde{\vF} + (1-\chi_1) \un{\vF}, &
	Q^\ast &:= \chi_2 \ov{Q} + (1-\chi_2) \un{Q}.
\end{align*}
We infer from \eqref{eq:subsol1-mod-reg} and \eqref{eq:subsol2-reg} that 
$$
	\vF^\ast \in \Cweak\big([T_0,T + T_0];L^2(\T^n;\R^n)\big) \cap C^1\big((T_0,T + T_0)\times \T^n;\R^n\big),
$$
and from \eqref{eq:subsol1-reg-Q} that
$$
	Q^\ast \in C^1\big( [T_0, T + T_0] \times \T^n; \R^+\big). 
$$

Due to \eqref{eq:subsolstar-rho-m-U-q}, the tuple $(\rho^\ast,\vm^\ast, \mU^\ast, q^\ast)$ satisfies the PDEs \eqref{eq:eulerlin-mass} and \eqref{eq:eulerlin-mom} pointwise for all $(t,\vx)\in (T_0,T+T_0) \times \T^n$, since both tuples $(\ov{\rho},\widetilde{\vm},\widetilde{\mU},\widetilde{q})$ and $(\un{\rho},\un{\vm},\un{\mU},\un{q})$ do. Let us next show that $(\rho^\ast,\vm^\ast, \mU^\ast, q^\ast,\vF^\ast)$ also satisfies \eqref{eq:eulerlin-energy} pointwise for all $(t,\vx)\in (T_0,T+T_0) \times \T^n$. For $t\in \big(T_0, \tfrac{3}{8} T_\ts \big] \cup \big[\half T_\ts, T_0+T\big)$ this follows immediately from the fact that \eqref{eq:eulerlin-energy} holds for $(\ov{\rho},\widetilde{\vm},\widetilde{\mU},\widetilde{q},\widetilde{\vF})$ and $(\un{\rho},\un{\vm},\un{\mU},\un{q},\un{\vF})$, respectively. For $t\in \big( \tfrac{3}{8} T_\ts , \half T_\ts \big)$ we compute with the help of \eqref{eq:un=til-rho-m-U-q} 
\begin{align*}
	&\partial_t \Big(\tfrac{n}{2} \big(q^\ast - p(\rho^\ast)\big) +  P(\rho^\ast)\Big) + \Div \vF^\ast \\
	&= \chi_1 \partial_t \Big(\tfrac{n}{2} \big(q^\ast - p(\rho^\ast)\big) +  P(\rho^\ast)\Big) + (1-\chi_1) \partial_t \Big(\tfrac{n}{2} \big(q^\ast - p(\rho^\ast)\big) +  P(\rho^\ast)\Big) + \Div \big( \chi_1 \widetilde{\vF} + (1-\chi_1) \un{\vF} \big) \\
	&=  \chi_1 \Big[ \partial_t \Big(\tfrac{n}{2} \big(\widetilde{q} - p(\ov{\rho})\big) +  P(\ov{\rho})\Big) + \Div \widetilde{\vF} \Big] + (1-\chi_1) \Big[ \partial_t \Big(\tfrac{n}{2} \big(\un{q} - p(\un{\rho})\big) +  P(\un{\rho})\Big) + \Div \un{\vF} \Big] \ \leq \ 0.
\end{align*}

Let us next prove that 
\begin{equation} \label{eq:401}
	(\vm^\ast, \mU^\ast, q^\ast, \vF^\ast) (t,\vx) \in \sU_{\rho^\ast(t,\vx),Q^\ast(t,\vx)}\qquad \text{ for all }(t,\vx)\in (T_0,T+T_0)\times \T^n. 
\end{equation} 
Again for $t\in \big( T_0, \tfrac{1}{4} T_\ts\big] \cup \big[\half T_\ts , T + T_0\big)$, \eqref{eq:401} simply follows from \eqref{eq:mod-subsol-U} and \eqref{eq:subsol2-subscond}, respectively. If $t\in \big(\frac{1}{4} T_\ts , \tfrac{3}{8} T_\ts \big]$, then $(\rho^\ast,\vm^\ast, \mU^\ast, q^\ast,\vF^\ast)(t,\vx)=(\ov{\rho},\widetilde{\vm},\widetilde{\mU},\widetilde{q},\widetilde{\vF})(t,\vx)$ and $Q^\ast(t,\vx)\geq \ov{Q}(t,\vx)$ due to \eqref{eq:order-Qun-vs-Qov}. Hence, again \eqref{eq:mod-subsol-U} yields \eqref{eq:401}. Finally, for $t\in \big(\tfrac{3}{8} T_\ts, \half T_\ts\big)$ we know that both
\begin{equation*}
	(\widetilde{\vm}, \widetilde{\mU}, \widetilde{q}, \widetilde{\vF}) (t,\vx) \in \sU_{\rho^\ast(t,\vx),Q^\ast(t,\vx)}\quad \text{ and }\quad (\un{\vm}, \un{\mU}, \un{q} , \un{\vF}) (t,\vx) \in \sU_{\rho^\ast(t,\vx),Q^\ast(t,\vx)}
\end{equation*}
due to \eqref{eq:subsolstar-rho-m-U-q}, \eqref{eq:subsol2-subscond}, \eqref{eq:mod-subsol-U} and \eqref{eq:order-Qun-vs-Qov}. Using \eqref{eq:un=til-rho-m-U-q} and the convexity of the set $\sU_{\rho^\ast(t,\vx),Q^\ast(t,\vx)}$, we conclude with \eqref{eq:401}.

Now, as we have verified all the assumptions of Prop.~\ref{prop:ci-nonuniqueness}, we may apply this proposition to obtain infinitely many 
$$
	\vm\in \Cweak\big([T_0,T+T_0];L^2(\T^n;\R^n)\big) \cap L^\infty\big((T_0,T+T_0)\times \T^n; \R^n\big)
$$ 
such that $(\rho=\rho^\ast, \vm)$ solve the equations \eqref{eq:ci-sol-mass}, \eqref{eq:ci-sol-mom} and the inequality \eqref{eq:ci-sol-energy}. We also know that
\begin{align*}
	\rho^\ast(T_0,\cdot ) &= \rho_{0,\ep}, &
	\vm^\ast(T_0,\cdot ) &= \vm_{0,\ep},
\end{align*}
see \eqref{eq:subsolstar-rho-m-U-q} and \eqref{eq:defn-mod-data}. Moreover, there holds 
$$
	\tfrac{n}{2} \big( q^\ast - p(\rho^\ast)\big) + P(\rho^\ast) = \frac{|\vm_{0,\ep}|^2}{2 \rho_{0,\ep}} + P(\rho_{0,\ep})
$$
according to \eqref{eq:subsolstar-rho-m-U-q} and \eqref{eq:mod-subsol-K}. Plugging these facts into \eqref{eq:ci-sol-mass}, \eqref{eq:ci-sol-mom}, \eqref{eq:ci-sol-energy}, we end up with \eqref{eq:euler-weak-mass}, \eqref{eq:euler-weak-mom}, \eqref{eq:euler-weak-energy}, respectively. Hence, all of the infinitely many pairs $(\rho,\vm)$ are indeed admissible weak solutions on the time interval $(T_0, T+ T_0)$, or -- by shifting in time -- on $(0,T)$. Consequently, the initial datum $(\rho_{0,\ep},\vm_{0,\ep})$ is wild in the sense of Defn.~\ref{defn:wild-data} as desired.

\section{Comparison with \cite{ChiFei24_1} and further remarks} \label{sec:remarks} 

We showed that the initial data for which the isentropic Euler equations \eqref{eq:euler-mass}, \eqref{eq:euler-mom} admit infinitely many admissible weak solutions form a dense set in any $L^r$, $r\in [1,\infty)$, see Thm.~\ref{thm:density}. The crucial difference to the work by \name{Chiodaroli}-\name{Feireisl}~\cite{ChiFei24_1}, who proved a similar result, is that our result is global in time, cf. Rem.~\ref{rem:local-vs-global-in-time}. In order to explain this in more detail, let us define ``$T$-wild'' initial data as follows.

\begin{defn} 
	Let $T>0$. We call an initial datum $(\rho_{0},\vm_{0})\in L^\infty(\T^n; \R^+\times \R)$ \emph{$T$-wild} if there exist infinitely many admissible weak solutions 
	$$
		(\rho,\vm) \in L^\infty\big((0,T) \times \T^n; \R^+ \times \R^n\big)
	$$
	on the time interval $(0,T)$. Moreover, we denote the set of all $T$-wild initial data by
	$$
		\sS_T := \left\{ (\rho_{0},\vm_{0}) \in L^\infty\big(\T^n; \R^+ \times \R^n\big)\, \Big|\, (\rho_{0},\vm_{0})\text{ is }T\text{-wild} \right\}.
	$$
\end{defn}

Now, we may phrase Thm.~\ref{thm:density} as:
\begin{center}
	The set $\bigcap_{T>0} \sS_T$ is a dense subset of $L^r\big(\T^n; \R_0^+ \times \R^n\big)$\\ with respect to the $L^r$-norm (for any $r\in [1,\infty)$). 
\end{center}
On the contrary, \name{Chiodaroli}-\name{Feireisl}~\cite{ChiFei24_1} showed that the set $\bigcup_{T>0} \sS_T$ is dense, which is a weaker result because of the obvious fact 
$$
	\bigcap_{T>0} \sS_T \subset \bigcup_{T>0} \sS_T. 
$$
Note furthermore that Thm.~\ref{thm:density} also implies that for any fixed $T>0$, the set $\sS_T$ is dense in the above sense. 

Finally, we remark that using similar ideas one may show an analogous statement of Thm.~\ref{thm:density} for the incompressible Euler equations, i.e.~a version of \name{Sz{\'e}kelyhidi}-\name{Wie\-de\-mann}~\cite[Cor.~3]{SzeWie12} which considers solutions satisfying the \emph{local} energy inequality. The details of such a result are left for future work.

\section*{Acknowledgements} 
D.W.B.~acknowledges support from the Cambridge Trust and the Cantab Capital Institute for Mathematics of Information. 
S.M.~acknowledges financial support from the Alexander von Humboldt foundation and also from the Deutsche Forschungsgemeinschaft (DFG, German Research Foundation) within SPP 2410, project number 525935467.

\appendix

\section{Existence of a suitable differential operator in 3-D} 

Goal of this section is to prove a 3-D version of \cite[Lemma~3.3]{Markfelder24}, which is necessary to prove Props.~\ref{prop:ci-nonuniqueness} and \ref{prop:ci-data}. To this end, let us first recall the definition of the wave cone from \cite{Markfelder24} and \cite{BouMarTit26pre}.

\begin{defn}[See {\cite[Sect.~3.1]{Markfelder24}} and {\cite[Defn.~2.7]{BouMarTit26pre}}] \label{defn:wave-cone}
	The wave cone $\Lambda\subset \phase$ for the linear system \eqref{eq:eulerlin-mass}-\eqref{eq:eulerlin-energy} reads
	$$
		\Lambda:= \left\{ (\vm,\mU,q,\vF)\in \phase \, \Big|\, \exists \veta=(\eta_t,\veta_\vx)\in \R^{1+n} \text{ with } \veta_\vx\neq 0 \text{ and } \left( \begin{array}{cc} 0 & \vm^\trans \\ \vm & \mU + q\id \\ \tfrac{n}{2} q & \vF^\trans \end{array} \right) \cdot \veta = 0\right\}.
	$$
\end{defn}

Now we are ready to state the 3-D version of \cite[Lemma~3.3]{Markfelder24}.

\begin{lemma} \label{lemma:suitable-operator-3D}
	Let $n=3$ and $(\vm,\mU,q,\vF)\in \Lambda$. Then there exists a third order homogeneous differential operator 
	$$
		\opL_{(\vm,\mU,q,\vF)} : C^\infty (\R^4) \to C^\infty(\R^4;\phase)
	$$
	with the following two properties: 
	\begin{enumerate}
		\item \label{item:suitable-operator-3D-a} For any function $g\in C^\infty(\R^4)$, $\opL_{(\vm,\mU,q,\vF)}[g]$ solves the homogeneous linear system 
		\begin{align} 
			\Div \vm &= 0 ,  \label{eq:eulerlin-hom-mass} \\
			\partial_t \vm + \Div (\mU + q\id) &= 0 , \label{eq:eulerlin-hom-mom} \\
			\tfrac{3}{2} \partial_t q + \Div \vF &= 0. \label{eq:eulerlin-hom-energy} 
		\end{align}
		
		\item \label{item:suitable-operator-3D-b} If we set $g(t,\vx):= h( (t,\vx)\cdot \veta)$ with an arbitrary function $h\in C^\infty(\R)$ and where $\veta$ corresponds to $(\vm,\mU,q,\vF)\in \Lambda$, we obtain
		\begin{equation} \label{eq:lin-op-g-special-form}
			\opL_{(\vm,\mU,q,\vF)}[g](t,\vx) = (\vm,\mU,q,\vF) h'''((t,\vx)\cdot \veta).
		\end{equation}
	\end{enumerate}
\end{lemma}

The proof is similar to the 2-D case presented in \cite[Sect.~3.2]{Markfelder24}.

\begin{proof}
According to Defn.~\ref{defn:wave-cone}, there exists $\veta=(\eta_t,\veta_\vx)\in \R^4$ with $\veta_\vx\neq 0$ and 
\begin{equation} \label{eq:op-lambda}
	\left(\begin{array}{cc} 0 & \vm^\trans \\ \vm & \mU + q\id \\ \tfrac{3}{2} q & \vF^\trans\end{array}\right)\cdot\veta = 0 .
\end{equation}

We first consider the case that $\veta$ is of the form $\veta=(a,b,0,0)$, $b\neq 0$. Define 
\begin{align*} 
	\opL_{m_1}[g] &:= - \frac{1}{b^3} \Big( m_2 \parthree{1}{1}{2}g + m_3 \parthree{1}{1}{3}g \Big) ,\\
	\opL_{m_2}[g] &:= \frac{m_2}{b^3} \parthree{1}{1}{1}g ,\\
	\opL_{m_3}[g] &:= \frac{m_3}{b^3} \parthree{1}{1}{1}g ,
\end{align*}
\begin{align*} 
	\opL_q[g] &:= \frac{1}{3 b^3} \Big( 2 m_2 \parthree{t}{1}{2}g + 2 m_3 \parthree{t}{1}{3}g + \big(U_{22} + q\big) \big( \parthree{1}{1}{1}g + \parthree{1}{2}{2}g \big) \\
	&\qquad\qquad + \big(U_{33} + q\big) \big( \parthree{1}{1}{1}g + \parthree{1}{3}{3}g \big) + 2 U_{23} \parthree{1}{2}{3}g \Big) ,
\end{align*}
\begin{align*} 
	\opL_{U_{11}}[g] &:= \frac{1}{b^3} \Big( 2 m_2 \parthree{t}{1}{2}g + 2 m_3 \parthree{t}{1}{3}g + \big( U_{22} + q \big) \parthree{1}{2}{2}g + \big( U_{33} + q \big) \parthree{1}{3}{3}g + 2 U_{23} \parthree{1}{2}{3}g \Big) - \opL_q[g] ,\\  
	\opL_{U_{12}}[g] &:= - \frac{1}{b^3} \Big( m_2 \parthree{t}{1}{1}g + \big( U_{22} + q \big) \parthree{1}{1}{2}g + U_{23} \parthree{1}{1}{3}g \Big) ,\\
	\opL_{U_{13}}[g] &:= - \frac{1}{b^3} \Big( m_3 \parthree{t}{1}{1}g + \big( U_{33} + q \big) \parthree{1}{1}{3}g + U_{23} \parthree{1}{1}{2}g \Big) ,\\ 
	\opL_{U_{22}}[g] &:= \frac{U_{22} + q}{b^3} \parthree{1}{1}{1}g - \opL_q[g] ,\\ 
	\opL_{U_{23}}[g] &:= \frac{U_{23}}{b^3} \parthree{1}{1}{1}g ,\\ 
	\opL_{U_{33}}[g] &:= \frac{U_{33} + q}{b^3} \parthree{1}{1}{1}g - \opL_q[g] ,
\end{align*}
\begin{align*} 
	\opL_{F_1}[g] &:= - \frac{1}{2b^3} \Big( 2 m_2 \parthree{t}{t}{2}g + 2 m_3 \parthree{t}{t}{3}g + \big(U_{22} + q\big) \big( \parthree{t}{1}{1}g + \parthree{t}{2}{2}g \big) \\
	&\qquad\qquad + \big(U_{33} + q\big) \big( \parthree{t}{1}{1}g + \parthree{t}{3}{3}g \big) + 2 U_{23} \parthree{t}{2}{3}g \Big) - \frac{F_2}{b^3} \parthree{1}{1}{2}g - \frac{F_3}{b^3} \parthree{1}{1}{3}g , \\
	\opL_{F_2}[g] &:= \frac{F_2}{b^3} \parthree{1}{1}{1}g , \\
	\opL_{F_3}[g] &:= \frac{F_3}{b^3} \parthree{1}{1}{1}g .
\end{align*}
A simple computation reveals that 
$$
	\opL_{U_{11}}[g] + \opL_{U_{22}}[g] + \opL_{U_{33}}[g] = 0, 
$$
i.e.~$\opL_{\mU}[g]$ is traceless as desired.

It is then straightforward to check that $\opL_{(\vm,\mU,q,\vF)}[g]$ as defined above solves the linear system \eqref{eq:eulerlin-hom-mass}-\eqref{eq:eulerlin-hom-energy} for any $g\in C^\infty(\R^4)$, i.e.~item~\ref{item:suitable-operator-3D-a} of Lemma~\ref{lemma:suitable-operator-3D} is satisfied. 

Now let $g(t,\vx):= h\big((t,\vx)\cdot \veta\big)= h(at + bx)$. Then we obtain
\begin{align*} 
	\opL_{m_1}[g] &= 0 ,\\
	\opL_{m_2}[g] &= m_2 h'''(at + bx) ,\\
	\opL_{m_3}[g] &= m_3 h'''(at + bx) , \\
	\opL_q[g] &= \frac{1}{3} \big(U_{22} + q + U_{33} + q\big) h'''(at + bx) , \\
	\opL_{U_{11}}[g] &= - q h'''(at + bx) ,\\  
	\opL_{U_{12}}[g] &= - \frac{a}{b} m_2 h'''(at + bx) ,\\
	\opL_{U_{13}}[g] &=  - \frac{a}{b} m_3 h'''(at + bx) ,\\ 
	\opL_{U_{22}}[g] &= U_{22} h'''(at + bx) ,\\ 
	\opL_{U_{23}}[g] &= U_{23} h'''(at + bx) ,\\ 
	\opL_{U_{33}}[g] &= U_{33} h'''(at + bx) , \\
	\opL_{F_1}[g] &:= - \frac{a}{2b} \big(U_{22} + q + U_{33} + q\big) h'''(at + bx) , \\
	\opL_{F_2}[g] &:= F_2 h'''(at + bx) , \\
	\opL_{F_3}[g] &:= F_3 h'''(at + bx) .
\end{align*}

Keeping \eqref{eq:op-lambda} in mind, i.e.
\begin{align*}
	b\,m_1 &=0, \\
	a\,m_1 + b\,(U_{11}+q) &=0, \\
	a\,m_2 + b\, U_{12} &=0, \\
	a\,m_3 + b\, U_{13} &=0, \\
	\frac{3a}{2}q + b\, F_1 &=0,
\end{align*}
where $b\neq 0$, we find $m_1= U_{11} + q = 0$ and 
$$
	U_{12} = - \frac{a}{b} m_2, \qquad U_{13} = - \frac{a}{b} m_3, \qquad F_1 = - \frac{3a}{2b}q.
$$
This allows to verify \eqref{eq:lin-op-g-special-form}, i.e.~item~\ref{item:suitable-operator-3D-b} of Lemma~\ref{lemma:suitable-operator-3D} is satisfied. 

Now we consider the general case $\veta=(\eta_t,\veta_\vx)\in \R^4$ with $\veta_\vx\neq 0$. We proceed in a similar fashion as in \cite[Sect.~4.4.3]{Markfelder}. Let $\mA\in \R^{3\times 3}$ be an orthogonal matrix whose first column is equal to $\frac{\veta_\vx}{|\veta_\vx|}$. Define $(\ov{\vm},\ov{\mU}, \ov{q}, \ov{\vF})\in \phase$ by 
\begin{align} \label{eq:app-A-001}
	\ov{\vm} &:= \mA^\trans \vm, &
	\ov{\mU} &:= \mA^\trans \mU \mA, &
	\ov{q} &:= q, &
	\ov{\vF} &:= \mA^\trans \vF,
\end{align}
and 
$$
	\ov{\veta}= (\ov{\eta}_t, \ov{\veta}_\vx) :=(\eta_t, \mA^\trans \veta_\vx).
$$
It is then simple to check that 
\begin{equation*} 
	\left(\begin{array}{cc} 0 & \ov{\vm}^\trans \\ \ov{\vm} & \ov{\mU} + \ov{q} \id \\ \tfrac{3}{2} \ov{q} & \ov{\vF}^\trans\end{array}\right)\cdot\ov{\veta} = 0 . 
\end{equation*}
Moreover, $\ov{\veta}\in \R^4$ is of the form $(a,b,0,0)$ with $b\neq 0$ studied above. Hence we obtain a differential operator $\opL_{(\ov{\vm},\ov{\mU}, \ov{q}, \ov{\vF})}$ satisfying items~\ref{item:suitable-operator-3D-a} and \ref{item:suitable-operator-3D-b} of Lemma~\ref{lemma:suitable-operator-3D} for $(\ov{\vm},\ov{\mU}, \ov{q}, \ov{\vF})\in \phase$ instead of $(\vm,\mU,q,\vF)\in \phase$. 

Next, we set 
\begin{align} 
	\opL_{\vm}[g](t,\vx) &:= \mA \opL_{\ov{\vm}}[g](t,\mA^\trans \vx), &
	\opL_{\mU}[g](t,\vx) &:= \mA \opL_{\ov{\mU}}[g](t,\mA^\trans \vx) \mA^\trans, \label{eq:app-A-002} \\
	\opL_{q}[g](t,\vx) &:= \opL_{\ov{q}}[g](t,\mA^\trans \vx), &
	\opL_{\vF}[g](t,\vx) &:= \mA \opL_{\ov{\vF}}[g](t,\mA^\trans \vx). \label{eq:app-A-003}
\end{align}
It straightforward to check that $\opL_{(\vm,\mU,q,\vF)}$ solves \eqref{eq:eulerlin-hom-mass}-\eqref{eq:eulerlin-hom-energy} for any $g\in C^\infty(\R^4)$, i.e.~item~\ref{item:suitable-operator-3D-a} of Lemma~\ref{lemma:suitable-operator-3D} holds. For $g(t,\vx):= h( (t,\vx)\cdot \veta) = h( (t,\mA^\trans \vx)\cdot \ov{\veta})$ we obtain 
\begin{equation} \label{eq:app-A-004}
	\opL_{(\ov{\vm},\ov{\mU}, \ov{q}, \ov{\vF})}[g](t,\mA^\trans \vx) = (\ov{\vm},\ov{\mU}, \ov{q}, \ov{\vF}) h'''( (t,\mA^\trans \vx)\cdot \ov{\veta}) = (\ov{\vm},\ov{\mU}, \ov{q}, \ov{\vF}) h'''((t,\vx)\cdot \veta).
\end{equation}
Combining \eqref{eq:app-A-001}, \eqref{eq:app-A-002}, \eqref{eq:app-A-003} and \eqref{eq:app-A-004}, we infer item~\ref{item:suitable-operator-3D-b} of Lemma~\ref{lemma:suitable-operator-3D} as desired. 
\end{proof}

\section{A three-dimensional version of \cite[Prop.~3.9]{Markfelder24}} \label{sec:app-U} 

The goal of this section is to prove a three-dimensional analogue of \cite[Prop.~3.9]{Markfelder24}, i.e.~$\sW_{\rho,Q}\subset \sU_{\rho,Q}$ with a set $\sW_{\rho,Q}$ to be defined below, see \eqref{eq:WrhoQ-2d} above for the two-dimensional version. Before we go on, let us state and prove the following lemma, which we will need later on. 

\begin{lemma} \label{lemma:definiteness-remains-for-sum}
	Let $\mM\in \sym{3}$ be a symmetric, negative definite matrix, and $\vsigma=(\sigma_1,\sigma_2,\sigma_3)\in \R^3\setminus\{0\}$. Then 
	$$
		\mM - \frac{1}{\vsigma^\trans\cdot \mM^{-1} \cdot \vsigma} \vsigma \otimes \vsigma
	$$
	is negative semi-definite. 
\end{lemma}

\begin{proof}
In order to keep the presentation short, we use the abbreviation 
$$
	\mA := \mM - \frac{1}{\vsigma^\trans\cdot \mM^{-1} \cdot \vsigma} \vsigma \otimes \vsigma.
$$

Let us first assume that $\mM=\diag(\lambda_1,\lambda_2,\lambda_3)$. As $\mM$ is negative definite, we know that $\lambda_i<0$ for all $i=1,2,3$. Moreover, $\mM^{-1}= \diag(\lambda_1^{-1},\lambda_2^{-1},\lambda_3^{-1})$. Hence, in this case, $\mA$ reads
$$
	\mA = \diag(\lambda_1,\lambda_2,\lambda_3) - \frac{1}{\lambda_1^{-1} (\sigma_1)^2 + \lambda_2^{-1} (\sigma_2)^2 + \lambda_3^{-1} (\sigma_3)^2} \vsigma\otimes \vsigma. 
$$ 

It is well-known that the characteristic polynomial of a $3\times 3$-matrix $\mA$ is given by
$$
	\chi_\mA(s) = s^3 - (\tr \mA) s^2 + \half \big((\tr \mA)^2 - \tr \mA^2 \big)s - \det \mA,
$$
see e.g.~\cite[Fact~4.9.2]{Bernstein}). Using \cite[Fact 2.16.3]{Bernstein}, we find $\det \mA = 0$. Consequently, one eigenvalue of $\mA$ is zero, and the other two eigenvalues are given by the zeros of the quadratic polynomial
\begin{equation} \label{eq:app-quad-poly}
	s^2 - (\tr \mA) s + \half \big((\tr \mA)^2 - \tr \mA^2 \big). 
\end{equation}
It is not difficult to see that if 
\begin{equation} \label{eq:app-trace}
	\tr \mA < 0, \qquad \text{ and }\qquad (\tr \mA)^2 - \tr \mA^2 > 0,
\end{equation}
then the zeros of the polynomial \eqref{eq:app-quad-poly} are negative. Hence, it suffices to show \eqref{eq:app-trace}.

We compute
\begin{align*}
	\tr \mA &= \lambda_1+\lambda_2+\lambda_3 - \frac{(\sigma_1)^2 + (\sigma_2)^2 + (\sigma_3)^2}{\lambda_1^{-1} (\sigma_1)^2 + \lambda_2^{-1} (\sigma_2)^2 + \lambda_3^{-1} (\sigma_3)^2} \\
	&= \frac{1}{\lambda_1^{-1} (\sigma_1)^2 + \lambda_2^{-1} (\sigma_2)^2 + \lambda_3^{-1} (\sigma_3)^2} \\
	&\qquad \Big( \big(\lambda_1^{-1} (\sigma_1)^2 + \lambda_2^{-1} (\sigma_2)^2 + \lambda_3^{-1} (\sigma_3)^2\big) (\lambda_1+\lambda_2+\lambda_3) - (\sigma_1)^2 - (\sigma_2)^2 - (\sigma_3)^2 \Big) \\
	&= \frac{1}{\lambda_1^{-1} (\sigma_1)^2 + \lambda_2^{-1} (\sigma_2)^2 + \lambda_3^{-1} (\sigma_3)^2} \\
	&\qquad \left( \frac{\lambda_1}{\lambda_2} (\sigma_2)^2 + \frac{\lambda_1}{\lambda_3} (\sigma_3)^2+ \frac{\lambda_2}{\lambda_1} (\sigma_1)^2+ \frac{\lambda_2}{\lambda_3} (\sigma_3)^2+\frac{\lambda_3}{\lambda_1} (\sigma_1)^2+\frac{\lambda_3}{\lambda_2} (\sigma_2)^2\right)\ <\ 0, 
\end{align*}
since $\lambda_i<0$ for all $i=1,2,3$. One can find similarly that
\begin{align*}
	(\tr \mA)^2 - \tr \mA^2 &= \frac{2}{\lambda_1^{-1} (\sigma_1)^2 + \lambda_2^{-1} (\sigma_2)^2 + \lambda_3^{-1} (\sigma_3)^2} \\
	&\qquad \left( \frac{\lambda_1\lambda_3}{\lambda_2} (\sigma_2)^2 + \frac{\lambda_1\lambda_2}{\lambda_3} (\sigma_3)^2+ \frac{\lambda_2\lambda_3}{\lambda_1} (\sigma_1)^2\right)\ >\ 0, 
\end{align*}
the details of which we leave to the reader. So \eqref{eq:app-trace} is shown.

It remains to consider the general case where $\mM$ is not necessarily diagonal. As $\mM$ is symmetric and negative definite, there exists an orthogonal matrix $\mS$ such that $\mS^\trans \cdot \mM \cdot \mS = \diag(\lambda_1,\lambda_2,\lambda_3)$ and $\mS^\trans \cdot \mM^{-1} \cdot \mS = \diag(\lambda_1^{-1},\lambda_2^{-1},\lambda_3^{-1})$. What we already proved above implies 
\begin{align*}
	\mS^\trans \cdot \mA \cdot \mS &= \mS^\trans \cdot \mM \cdot \mS - \frac{1}{\vsigma^\trans\cdot \mM^{-1} \cdot \vsigma} (\mS^\trans\cdot\vsigma) \otimes(\mS^\trans\cdot \vsigma) \\
	&= \mS^\trans \cdot \mM \cdot \mS - \frac{1}{(\mS^\trans\cdot \vsigma)^\trans \cdot (\mS^\trans \cdot \mM^{-1} \cdot \mS) \cdot (\mS^\trans \cdot \vsigma)} (\mS^\trans\cdot\vsigma) \otimes(\mS^\trans\cdot \vsigma)\ \leq \ 0,
\end{align*}
which leads to $\mA\leq 0$ as desired, see e.g.~\cite[Prop.~8.1.2~(xiv)]{Bernstein}.
\end{proof}

Let us next construct $\sW_{\rho,Q}$, which will be done in a similar fashion as in 2-D. Similarly to the two-dimensional case, we begin by defining eight vectors $\vsigma^j\in \R^3$, $j=1,2,...,8$:
\begin{align*}
	\vsigma^1 &:= \left(\begin{array}{r} \hphantom{-}1 \\ 1 \\ 1 \end{array} \right), & \vsigma^2 &:= \left(\begin{array}{r} -1 \\ -1 \\ 1 \end{array} \right), & \vsigma^3 &:= \left(\begin{array}{r} 1 \\ -1 \\ 1 \end{array} \right), & \vsigma^4 &:= \left(\begin{array}{r} -1 \\ 1 \\ 1 \end{array} \right), \\
	\vsigma^5 &:= \left(\begin{array}{r} 1 \\ 1 \\ -1 \end{array} \right), & \vsigma^6 &:= \left(\begin{array}{r} -1 \\ -1 \\ -1 \end{array} \right), & \vsigma^7 &:= \left(\begin{array}{r} 1 \\ -1 \\ -1 \end{array} \right), & \vsigma^8 &:= \left(\begin{array}{r} -1 \\ 1 \\ -1 \end{array} \right). 
\end{align*}
For given $\rho,Q>0$, we set 
\begin{align*} 
	\sV_{\rho,Q} := \bigg\{ (\vm,\mU,q) \in \R^3\times \symz{3} \times \R \,\Big|\, & \bullet \ \frac{\vm\otimes\vm}{\rho} + p(\rho)\id - \mU -q \id< 0, \\
	& \bullet \ q< Q \quad \bigg\} . 
\end{align*}
Moreover, we define auxiliary functions $A^j_{\rho,Q}, r^j_{\rho,Q}:\sV_{\rho,Q} \to \R$ and $\vf_{\rho,Q}^j:\sV_{\rho,Q} \to \R^3$ for $j=1,2,...,8$ by 
\begin{align*}
	A_{\rho,Q}^j(\vm,\mU,q) &:= \frac{-1}{(\vsigma^j)^\trans \cdot \left(\frac{\vm\otimes \vm}{\rho} - \mU + (p(\rho) - q)\id \right)^{-1}\cdot \vsigma^j} , \\
	r_{\rho,Q}^j(\vm,\mU,q) &:= \frac{1}{2 (Q-q)} \left( - \vm \cdot \vsigma^j + \sqrt{(\vm \cdot \vsigma^j)^2 + 4\rho A_{\rho,Q}^j(\vm,\mU,q) + 4\rho (Q-q)} \right) , \\
	\vf_{\rho,Q}^j(\vm,\mU,q) &:= \frac{A_{\rho,Q}^j(\vm,\mU,q)}{r_{\rho,Q}^j(\vm,\mU,q)} \vsigma^j .
\end{align*}
Note that 
\begin{equation} \label{eq:app-A>0}
	A_{\rho,Q}^j(\vm,\mU,q) > 0 \qquad \text{ for all } (\vm,\mU,q)\in \sV_{\rho,Q} \text{ and all }j=1,2,...,8,
\end{equation}
because every negative definite matrix is invertible and its inverse is again negative definite. It is then simple to deduce that also 
\begin{equation} \label{eq:app-r>0}
	r_{\rho,Q}^j(\vm,\mU,q) > 0 \qquad \text{ for all } (\vm,\mU,q)\in \sV_{\rho,Q} \text{ and all }j=1,2,...,8,
\end{equation}
and hence $\vf_{\rho,Q}^j$ is well-defined for all $(\vm,\mU,q)\in \sV_{\rho,Q}$ and all $j=1,2,...,8$. 

\begin{rem} 
	Using Cramer's rule $\mM^{-1}= \frac{\adj(\mM)}{\det (\mM)}$, where $\adj(\mM)$ denotes the adjugate matrix of $\mM$, one can easily verify that the above expression for $A^j_{\rho,Q}$ coincides in two space dimensions with the expression for $A^j_{\rho,Q}$ used in \cite{Markfelder24}, see also Sect.~\ref{subsec:ci-U} above. 
\end{rem}

Now, we are ready to define the sets $\sW_{\rho,Q}$ as
\begin{align*} 
	\sW_{\rho,Q} := \Bigg\{ (\vm,\mU,q,\vF) \in \phase\,\Big|\, & \bullet \ \frac{\vm\otimes\vm}{\rho} + p(\rho)\id - \mU -q \id< 0,
	\\
	& \bullet \ \exists\,\kappa_j\in \R^+, j=1,2,...,8, \text{ with } \ \sum_{j=1}^8\kappa_j =1\ \text{ such that } \\
	&\qquad \vF - \frac{\frac{3}{2} \big(q - p(\rho)\big) + P(\rho) + p(\rho)}{\rho} \vm = \sum_{j=1}^8\kappa_j \vf^j_{\rho,Q}(\vm,\mU,q), \\
	& \bullet \ q< Q \quad \Bigg\} . 
\end{align*}
Similarly to \cite[Prop.~3.9]{Markfelder24} the following holds.

\begin{prop} \label{prop:WsubsetU-3D}
	It holds that $\sW_{\rho,Q}\subset \sU_{\rho,Q}$. 
\end{prop}

In order to prove \ref{prop:WsubsetU-3D}, we follow \cite[Proof of Prop.~3.9]{Markfelder24} and proceed in three steps.

\subsection{Step 1: Rigid energy flux} 

\begin{lemma} \label{lemma:WsubsetQ-step1}
	Let $(\vm,\mU,q,\vF)\in \phase$ such that 
	\begin{align*}
		\frac{\vm\otimes\vm}{\rho} + p(\rho)\id - \mU -q \id &\leq 0, \\
		q &\leq Q, \\
		\vF &= \frac{\frac{3}{2} \big(q - p(\rho)\big) + P(\rho) + p(\rho)}{\rho} \vm.
	\end{align*}
	Then $(\vm,\mU,q,\vF)\in (\sK_{\rho,Q})^\co$. 
\end{lemma}

\begin{proof}
From \cite[Lemma~4.3.6]{Markfelder} we infer that there exist $N\in \N$ and $\big(\tau_i,(\vm_i,\mU_i)\big)\in \R^+\times \R^3 \times \symz{3}$ for $i=1,...,N$ such that 
\begin{align*}
	\sum_{i=1}^N \tau_i &= 1, \\
	\frac{\vm_i \otimes \vm_i}{\rho} + p(\rho) \id &= \mU_i + q \id, \qquad \text{ for all }i=1,...,N, \text{ and } \\
	(\vm,\mU) &= \sum_{i=1}^N \tau_i (\vm_i,\mU_i).
\end{align*}

Setting 
$$
	\vF_i := \frac{\frac{3}{2} \big(q - p(\rho)\big) + P(\rho) + p(\rho)}{\rho} \vm_i, \qquad \text{ for }i=1,...,N,
$$
we obtain $(\vm_i,\mU_i,q,\vF_i)\in \sK_{\rho,Q}$, as well as 
$$
	(\vm,\mU,q,\vF) = \sum_{i=1}^N \tau_i (\vm_i,\mU_i,q,\vF_i).
$$
Hence $(\vm,\mU,q,\vF)\in (\sK_{\rho,Q})^\co$. 
\end{proof}

\subsection{Step 2: Eight particular energy fluxes} 

\begin{lemma} \label{lemma:WsubsetQ-step2}
	Let $(\vm,\mU,q,\vF)\in \phase$ such that $\exists j\in \{1,2,...,8\}$ with
	\begin{align*}
		\frac{\vm\otimes\vm}{\rho} + p(\rho)\id - \mU -q \id &< 0, \\
		q &< Q, \\
		\vF - \frac{\frac{3}{2} \big(q - p(\rho)\big) + P(\rho) + p(\rho)}{\rho} \vm &= \vf^j_{\rho,Q}(\vm,\mU,q).
	\end{align*}
	Then $(\vm,\mU,q,\vF)\in (\sK_{\rho,Q})^\co$. 
\end{lemma}

\begin{proof}
We set
\begin{align*}
	\mu_\pm := \frac{1}{2r_{\rho,Q}^j} \left( \frac{2}{3} \left( \frac{\rho}{r^j_{\rho,Q}} - \vm\cdot \vsigma^j\right) \pm \sqrt{4 \rho A^j_{\rho,Q} + \frac{4}{9} \left( \frac{\rho}{r^j_{\rho,Q}} - \vm\cdot \vsigma^j\right)^2}\right),
\end{align*}
as well as 
\begin{align*}
	\widehat{\vm} &:= r^j_{\rho,Q} \vsigma^j, \\
	\widehat{\mU} &:= \frac{r^j_{\rho,Q}}{\rho} \left( \vm \otimes \vsigma^j + \vsigma^j\otimes \vm + \frac{2}{3} \left( \frac{\rho}{r^j_{\rho,Q}} - \vm\cdot \vsigma^j\right) \vsigma^j \otimes \vsigma^j \right) - \frac{2}{3} \id, \\ 
	\widehat{q} &:= \frac{2}{3}, \\
	\widehat{\vF} &:= \frac{\vm}{\rho} + \left[ \frac{\frac{3}{2} \big(q - p(\rho)\big) + P(\rho) + p(\rho)}{\rho}  r^j_{\rho,Q} + \frac{2}{3\rho} \left( \frac{\rho}{r^j_{\rho,Q}} - \vm\cdot \vsigma^j\right) \right] \vsigma^j .
\end{align*}

Using \eqref{eq:app-A>0} and \eqref{eq:app-r>0}, we find
\begin{equation*} 
	\mu_- < 0 \qquad \text{ and } \qquad \mu_+ > 0.
\end{equation*}
So, if we set
\begin{align*}
	\tau_1 &:= \frac{\mu_+}{\mu_+ - \mu_-}, & \tau_2 &:= -\frac{\mu_-}{\mu_+ - \mu_-},
\end{align*}
then $\tau_1,\tau_2\in (0,1)$ with $\tau_1+\tau_2 = 1$. Moreover, 
$$
	(\vm,\mU,q,\vF) = \tau_1 (\vm_1,\mU_1,q_1,\vF_1) + \tau_2 (\vm_2,\mU_2,q_2,\vF_2)
$$
for 
\begin{align*}
	(\vm_1,\mU_1,q_1,\vF_1) &:= (\vm,\mU,q,\vF) + \mu_- (\widehat{\vm},\widehat{\mU},\widehat{q},\widehat{\vF}), \\
	(\vm_2,\mU_2,q_2,\vF_2) &:= (\vm,\mU,q,\vF) + \mu_+ (\widehat{\vm},\widehat{\mU},\widehat{q},\widehat{\vF}). 
\end{align*}
Thus, it suffices to prove that $(\vm_1,\mU_1,q_1,\vF_1), (\vm_2,\mU_2,q_2,\vF_2)\in (\sK_{\rho,Q})^\co$. 

The definition of $\mu_\pm$ yields 
\begin{equation} \label{eq:app-mu-solves-quadratic-eq}
	(\mu_\pm)^2 r^j_{\rho,Q} - \mu_\pm \frac{2}{3} \left( \frac{\rho}{r^j_{\rho,Q}} - \vm\cdot \vsigma^j\right) = \frac{\rho A^j_{\rho,Q}}{r^j_{\rho,Q}}.
\end{equation}
Next, we compute 
\begin{align*} 
	&\frac{\vm_{1,2}\otimes\vm_{1,2}}{\rho} + p(\rho)\id - \mU_{1,2} -q_{1,2} \id \\
	&= \frac{\vm\otimes\vm}{\rho} + p(\rho)\id - \mU -q \id + \frac{r^j_{\rho,Q}}{\rho} \left[ (\mu_\mp)^2 r^j_{\rho,Q} - \mu_\mp \frac{2}{3} \left( \frac{\rho}{r^j_{\rho,Q}} - \vm\cdot \vsigma^j\right) \right] \vsigma^j \otimes \vsigma^j \\
	&= \frac{\vm\otimes\vm}{\rho} + p(\rho)\id - \mU -q \id + A^j_{\rho,Q} \vsigma^j \otimes \vsigma^j,
\end{align*}
where we have used \eqref{eq:app-mu-solves-quadratic-eq} in the last step. Keeping the definition of $A^j_{\rho,Q}$ in mind, we obtain by Lemma~\ref{lemma:definiteness-remains-for-sum} that 
$$
	\frac{\vm_{1,2}\otimes\vm_{1,2}}{\rho} + p(\rho)\id - \mU_{1,2} -q_{1,2} \id \leq 0. 
$$

Let us next check if $q_{1,2}\leq Q$. As $\mu_-<0$, we immediately see that $q_1 = q +\mu_- \widehat{q} < q < Q$. Using the definition of $r^j_{\rho,Q}$, we find
\begin{align*}
	&\frac{9}{4} \Big( r^j_{\rho,Q} \Big)^2 (Q-q)^2 + r^j_{\rho,Q} (Q-q) \vm\cdot \vsigma^j - \rho A^j_{\rho,Q} - \rho (Q-q) \\
	&\geq \Big( r^j_{\rho,Q} \Big)^2 (Q-q)^2 + r^j_{\rho,Q} (Q-q) \vm\cdot \vsigma^j - \rho A^j_{\rho,Q} - \rho (Q-q) = 0.
\end{align*}
The latter is equivalent to 
$$
	4 \rho A^j_{\rho,Q} + \frac{4}{9} \left( \frac{\rho}{r^j_{\rho,Q}} - \vm\cdot \vsigma^j\right)^2 \leq \left( 3 r^j_{\rho,Q} (Q-q) - \frac{2}{3}\left( \frac{\rho}{r^j_{\rho,Q}} - \vm\cdot \vsigma^j\right) \right)^2, 
$$
and, consequently, there holds $\mu_+ \leq \frac{3}{2} (Q-q)$. Here we have used that 
\begin{align*}
	&3 r^j_{\rho,Q} (Q-q) - \frac{2}{3}\left( \frac{\rho}{r^j_{\rho,Q}} - \vm\cdot \vsigma^j\right) \\
	&= \frac{1}{r^j_{\rho,Q} (Q-q)} \left( 3 \Big( r^j_{\rho,Q} \Big)^2 (Q-q)^2 + \frac{2}{3} r^j_{\rho,Q} (Q-q) \vm\cdot \vsigma^j - \frac{2}{3} \rho (Q-q) \right) \\
	& \geq \frac{2}{3r^j_{\rho,Q} (Q-q)} \left( \Big( r^j_{\rho,Q} \Big)^2 (Q-q)^2 + r^j_{\rho,Q} (Q-q) \vm\cdot \vsigma^j - \rho A^j_{\rho,Q} - \rho (Q-q) \right)= 0. 
\end{align*}
So, we finally obtain $q_2 = q + \mu_+ \widehat{q} \leq q + \frac{3}{2} (Q-q) \frac{2}{3} = Q$ as desired. 

For the energy flux, we compute 
\begin{align*}
	&\vF_{1,2} -  \frac{\frac{3}{2} \big(q_{1,2} - p(\rho)\big) + P(\rho) + p(\rho)}{\rho} \vm_{1,2} \\
	&= \vF - \frac{\frac{3}{2} \big(q - p(\rho)\big) + P(\rho) + p(\rho)}{\rho} \vm - \frac{1}{\rho} \left[ (\mu_\mp)^2 r^j_{\rho,Q} - \mu_\mp \frac{2}{3} \left( \frac{\rho}{r^j_{\rho,Q}} - \vm\cdot \vsigma^j\right) \right] \vsigma^j \\
	&= \vf^j_{\rho,Q} - \vf^j_{\rho,Q} = 0, 
\end{align*}
where we made use of \eqref{eq:app-mu-solves-quadratic-eq}.

Summarising, the points $(\vm_1,\mU_1,q_1,\vF_1)$ and $(\vm_2,\mU_2,q_2,\vF_2)$ satisfy the assumptions of Lemma~\ref{lemma:WsubsetQ-step1}, and hence $(\vm_1,\mU_1,q_1,\vF_1), (\vm_2,\mU_2,q_2,\vF_2)\in (\sK_{\rho,Q})^\co$ as desired. 
\end{proof}

\subsection{Step 3: Final proof of Prop.~\ref{prop:WsubsetU-3D}} 

\begin{proof}[Proof of Prop.~\ref{prop:WsubsetU-3D}]
Let $(\vm,\mU,q,\vF)\in \sW_{\rho,Q}$. Set 
$$
	\vF_j := \frac{\frac{3}{2} \big(q - p(\rho)\big) + P(\rho) + p(\rho)}{\rho} \vm + \vf^j_{\rho,Q}(\vm,\mU,q)\qquad \text{ for }j=1,2,...,8.
$$
Then each $(\vm,\mU,q,\vF_j)$, $j=1,2...,8$, satisfies the assumptions of Lemma~\ref{lemma:WsubsetQ-step2}, and hence $(\vm,\mU,q,\vF_j)\in (\sK_{\rho,Q})^\co$ for all $j=1,2,...,8$. Because of 
$$
	(\vm,\mU,q,\vF) = \sum_{j=1}^8 \kappa_j (\vm,\mU,q,\vF_j),
$$ 
this implies $(\vm,\mU,q,\vF)\in (\sK_{\rho,Q})^\co$. Thus, $\sW_{\rho,Q}\subset (\sK_{\rho,Q})^\co$. Since $\sW_{\rho,Q}$ is open (by continuity of the map $(\vm,\mU,q)\mapsto \vf^j_{\rho,Q}(\vm,\mU, q)$), we even have $\sW_{\rho,Q}\subset \interior{\big((\sK_{\rho,Q})^\co\big)}=\sU_{\rho,Q}$ as desired. 
\end{proof}

\section{Jensen's inequality}

In this section we recall Jensen's inequality both for real-valued and for matrix-valued convex functions. Jensen's inequality deals with the commutator of the mollification and a convex function. 

In this section we consider $n\in \{2,3\}$, $m\in \N$, and functions $u= u(t,\vx)\in \R^m$. Let $\phi$ be a space-time mollifier. For a parameter $\sigma>0$, the mollification $u\ast \phi_\sigma$ of a function $u$ is given by the convolution with $\phi_\sigma(t,\vx):= \frac{1}{\sigma^{n+1}} \phi\big(\frac{t}{\sigma}, \frac{\vx}{\sigma}\big)$. 

\begin{lemma}[Jensen's inequality for real-valued functions] \label{lemma:jensen-real-valued}
	Let $S\subset \R^m$ a convex set, and $f:S\to \R$ a convex map. Then for all functions $u$ which take values in $S$, and all parameters $\sigma>0$, the estimate
	$$
		f(u \ast \phi_\sigma) \leq f(u) \ast \phi_\sigma
	$$
	holds pointwise in $(t,\vx)$. 
\end{lemma}

\begin{lemma}[Jensen's inequality for matrix-valued functions] \label{lemma:jensen-matrix-valued}
	Let $S\subset \R^m$ a convex set, and $\mM:S\to \sym{n}$ a convex map\footnote{Here, convexity is understood in the sense of definiteness, i.e.~there holds
	$$
		\mM(\tau \vu_1 + (1-\tau) \vu_2) - \tau \mM(\vu_1) - (1-\tau) \mM(\vu_2) \leq 0 \qquad \text{ for all }\tau\in [0,1] \text{ and all }\vu_1,\vu_2\in S,
	$$
	where the inequality above means that the matrix on the left-hand side is negative semi-definite.}. Then for all functions $u$ which take values in $S$, and all parameters $\sigma>0$, the estimate\footnote{Similarly to the convexity of matrix-valued functions, this inequality has to be understood in the sense of definiteness, i.e.~$\mM(u) \ast \phi_\sigma - \mM(u \ast \phi_\sigma)$ is negative semi-definite.}
	$$
		\mM(u \ast \phi_\sigma) \leq \mM(u) \ast \phi_\sigma
	$$
	holds pointwise in $(t,\vx)$.  
\end{lemma}

Both Lemma~\ref{lemma:jensen-real-valued} and \ref{lemma:jensen-matrix-valued} are well-known so we do not present the proof here. 

Finally, we show that a certain matrix-valued function, which appears several times throughout this paper, is convex. 

\begin{lemma} \label{lemma:map-convex}
	The map $\R^+\times \R^n \to \sym{n}$, $(\rho,\vm)\mapsto \frac{\vm\otimes \vm}{\rho} + p(\rho) \id$ with $p$ given in \eqref{eq:isentropic-pressure} is convex. 
\end{lemma}

\begin{proof}
	Let $\tau\in [0,1]$ and $(\rho_1,\vm_1),(\rho_2,\vm_2)\in \R^+\times \R^n$. We have to show that the matrix 
	\begin{align*}
		\mA & :=\frac{(\tau\vm_1+(1-\tau)\vm_2)\otimes (\tau\vm_1+(1-\tau)\vm_2)}{\tau\rho_1 + (1-\tau)\rho_2} + p\big(\tau\rho_1 + (1-\tau)\rho_2\big) \id \\
		&\quad - \tau \left(\frac{\vm_1\otimes \vm_1}{\rho_1} + p(\rho_1) \id \right) - (1-\tau) \left(\frac{\vm_2\otimes \vm_2}{\rho_2} + p(\rho_2) \id \right) 
	\end{align*} 
	is negative semi-definite. A long but straightforward computation yields
	\begin{align*}
		\mA &= - \frac{\tau(1-\tau)}{\rho_1 \rho_2 (\tau\rho_1 + (1-\tau)\rho_2)} \big( \rho_2 \vm_1 - \rho_1 \vm_2 \big) \otimes \big( \rho_2 \vm_1 - \rho_1 \vm_2 \big) \\
		&\quad + \Big( p\big(\tau\rho_1 + (1-\tau)\rho_2\big) - \tau p( \rho_1) - (1-\tau) p(\rho_2) \Big) \id.
	\end{align*} 
	As for any $\vm^\ast\in \R^n$, the eigenvalues of the matrix $\vm^\ast \otimes \vm^\ast$ are $0$ and $|\vm^\ast|^2$, and $\rho\mapsto p(\rho)$ is convex, we infer $\mA\leq 0$ as desired. 
\end{proof}

\section{Poisson equation and Helmholtz decomposition} 

Let $n\in \{2,3\}$, as everywhere in this paper. Let us summarise some facts regarding the Poisson equation.

\begin{lemma} \label{lemma:poisson-mf} 
	Let $s\in \R$, and $f\in H^s(\T^n)$ with $\int_{\T^n} f \dx = 0$. Then there exists a unique $v\in H^{s+2}(\T^n)$ with $\int_{\T^n} v \dx = 0$ and $\Lap v = f$. Moreover, the estimate $\|v\|_{H^{s+2}(\T^n)} \lesssim \|f\|_{H^s(\T^n)}$ holds. 
\end{lemma}

As Lemma~\ref{lemma:poisson-mf} is well-known, we skip its proof for the sake of brevity. Combining Lemma~\ref{lemma:poisson-mf} with Sobolev embedding, we obtain the following statement. 

\begin{cor} \label{cor:poisson-mf-sobolev} 
	Let $f\in H^1(\T^n)$ with $\int_{\T^n} f \dx = 0$. Then there exists a unique $v\in H^{3}(\T^n)$ with $\int_{\T^n} v \dx = 0$ and $\Lap v = f$. There even holds $v\in C^1(\T^n)$, and the estimate 
	\begin{equation} \label{eq:poisson-sobolev}
		\|\Grad v\|_{C^0(\T^n)} \lesssim \|f\|_{C^1(\T^n)}
	\end{equation}
	is satisfied. 
\end{cor} 

\begin{proof}
	The first part follows immediately from Lemma~\ref{lemma:poisson-mf} with $s=1$. Keeping in mind that we consider $n\in \{2,3\}$, Sobolev embedding yields $v\in C^1(\T^n)$ and 
	$$
		\|\Grad v\|_{C^0(\T^n)} \lesssim \|v\|_{C^1(\T^n)} \lesssim \|v\|_{H^3(\T^n)} \lesssim \|f\|_{H^1(\T^n)} \lesssim  \|f\|_{C^1(\T^n)}.
	$$ 
\end{proof}

\begin{cor} \label{cor:poisson-not-mf} 
	Let $s\in \R$, and $f\in H^s(\T^n)$ with $\int_{\T^n} f \dx \leq 0$. Then there exists a $v\in H^{s+2}(\T^n)$ with $\int_{\T^n} v \dx = 0$ and $\Lap v \geq f$. Moreover, the estimate $\|v\|_{H^{s+2}(\T^n)} \lesssim \|f\|_{H^s(\T^n)}$ holds. 
\end{cor}

\begin{proof}
	Set $\widetilde{f}:= f - \int_{\T^n} f \dx$ and apply Lemma~\ref{lemma:poisson-mf} to $\widetilde{f}$ to obtain $v\in H^{s+2}(\T^n)$ with $\int_{\T^n} v \dx = 0$ and 
	$$
		\Lap v = \widetilde{f} = f - \int_{\T^n} f \dx \geq f.
	$$
	Moreover, there holds
	$$
		\|v\|_{H^{s+2}(\T^n)}^2 \lesssim \|\widetilde{f}\|_{H^s(\T^n)}^2 = \sum_{\vk\in \Z^n\setminus\{0\}} (1+|\vk|^2)^s |\widehat{f}_\vk|^2 \leq \sum_{\vk\in \Z^n} (1+|\vk|^2)^s |\widehat{f}_\vk|^2 = \|f\|_{H^s(\T^n)}^2.
	$$
\end{proof}

Let us also recall the Helmholtz decomposition in $H^s$.

\begin{lemma} \label{lemma:helmholtz} 
	Let $s\in \R$, and $\vv\in H^s(\T^n;\R^n)$. There exists a unique decomposition $\vv = \vw + \Grad \psi$ where $\vw\in H^s(\T^n;\R^n)$ with $\Div\vw=0$, and $\psi \in H^{s+1}(\T^n)$ with $\int_{\T^n} \psi \dx = 0$. Moreover, the following estimates hold
	\begin{align*}
		\|\vw\|_{H^s(\T^n)} &\lesssim \|\vv\|_{H^s(\T^n)}, & \|\psi\|_{H^{s+1}(\T^n)} &\lesssim \|\vv\|_{H^s(\T^n)}.
	\end{align*}
	We write $\Leray \vv:=\vw$.
\end{lemma}

Lemma~\ref{lemma:helmholtz} is again well-known, so we skip its proof.

\printbibliography[heading=bibintoc] 

@book{BenSer,
  author  = "S.~Benzoni-Gavage and D.~Serre",
  title   = "Multi-dimensional Hyberbolic Partial Differential Equations, First-order Systems and Applications",
  publisher = "Oxford University Press",
  year    = "2007",
  address = "Oxford",
  series  = "Oxford Mathematical Monographs"
}

@book{Bernstein,
	author  = "D.~S.~Bernstein",
	title   = "Matrix Mathematics",
	publisher = "Princeton University Press",
	year    = "2009",
	address = "Princeton and Oxford",
	edition = "2"
}

@article{BreFeiHof20_1,
  author  = "D.~Breit and E.~Feireisl and M.~Hofmanov{\'a}",
  title   = "Solution semiflow to the isentropic {E}uler system",
  journal = "Arch. Ration. Mech. Anal.",
  year    = "2020",
  volume  = "235",
  number  = "1",
  pages   = "167-194"
}

@article{BreChiKre18,
	author  = "J.~B{\v r}ezina and E.~Chiodaroli and O.~Kreml",
	title   = "Contact discontinuities in multi-dimensional isentropic {E}uler equations",
	journal = "Electron. J. Differential Equations",
	year    = "2018",
	volume  = "2018",
	number  = "94",
	pages   = "1-11"
}

@article{CheVasYu21,
  author  = "R.~M.~Chen and A.~Vasseur and C.~Yu",
  title   = "Global ill-posedness for a dense set of initial data to the isentropic system of gas dynamics",
  journal = "Adv. Math.",
  year    = "2021",
  volume  = "393",
  note    = "Paper No. 108057" 
}

@article{Chiodaroli14,
  author  = "E.~Chiodaroli",
  title   = "A counterexample to well-posedness of entropy solutions to the compressible {E}uler system",
  journal = "J. Hyperbolic Differ. Equ.",
  year    = "2014",
  volume  = "11",
  number  = "3",
  pages   = "493-519"
}

@article{ChiDelKre15,
  author  = "E.~Chiodaroli and C.~{De~Lellis} and O.~Kreml",
  title   = "Global ill-posedness of the isentropic system of gas dynamics",
  journal = "Comm. Pure Appl. Math.",
  year    = "2015",
  volume  = "68",
  number  = "7",
  pages   = "1157-1190"
}

@article{ChiFei24_1,
  author  = "E.~Chiodaroli and E.~Feireisl",
  title   = "On the density of ``wild'' initial data for the barotropic {E}uler system",
  journal = "Ann. Mat. Pura Appl. (4)",
  year    = "2024",
  volume  = "203",
  number  = "4",
  pages   = "1809-1817"
}

@article{ChiFei24_2,
	author  = "E.~Chiodaroli and E.~Feireisl",
	title   = "{G}limm's method and density of wild data for the {E}uler system of gas dynamics",
	journal = "Nonlinearity",
	year    = "2024",
	volume  = "37",
	number  = "3",
	note    = "Paper No. 035005"
}

@article{ChiKre14,
  author  = "E.~Chiodaroli and O.~Kreml",
  title   = "On the energy dissipation rate of solutions to the compressible isentropic {E}uler system",
  journal = "Arch. Ration. Mech. Anal.",
  year    = "2014",
  volume  = "214",
  number  = "3",
  pages   = "1019-1049"
}

@article{ChiKre18,
  author  = "E.~Chiodaroli and O.~Kreml",
  title   = "Non-uniqueness of admissible weak solutions to the {R}iemann problem for isentropic {E}uler equations",
  journal = "Nonlinearity",
  year    = "2018",
  volume  = "31",
  number  = "4",
  pages   = "1441-1460"
}

@article{CKMS21,
  author = "E.~Chiodaroli and O.~Kreml and V.~M{\'a}cha and S.~Schwarzacher",
  title  = "Non-uniqueness of admissible weak solutions to the compressible {E}uler equations with smooth initial data",
  journal = "Trans. Amer. Math. Soc.",
  year    = "2021",
  volume  = "374",
  number  = "4",
  pages   = "2269-2295"
}

@article{DelSze09,
  author  = "C.~{De~Lellis} and L.~{Sz{\'e}kelyhidi~Jr.}",
  title   = "The {E}uler equations as a differential inclusion",
  journal = "Ann. of Math. (2)",
  year    = "2009",
  volume  = "170",
  number  = "3",
  pages   = "1417-1436"
}

@article{DelSze10,
  author  = "C.~{De~Lellis} and L.~{Sz{\'e}kelyhidi~Jr.}",
  title   = "On admissibility criteria for weak solutions of the {E}uler equations",
  journal = "Arch. Ration. Mech. Anal.",
  year    = "2010",
  volume  = "195",
  number  = "1",
  pages   = "225-260"
}

@article{Feireisl14,
  author  = "E.~Feireisl",
  title   = "Maximal dissipation and well-posedness for the compressible {E}uler system",
  journal = "J. Math. Fluid Mech.",
  year    = "2014",
  volume  = "16",
  pages   = "447-461"
}

@book{FLMS,
  author  = "E.~Feireisl and M.~Luk{\'a}{\v c}ov{\'a}-Medvid'ov{\'a} and H.~Mizerov{\'a} and B.~She",
  title   = "Numerical analysis of compressible fluid flows",
  publisher = "Springer",
  year    = "2021",
  address = "Cham",
  series  = "Modeling, Simulation and Applications",
  number  = "20"
}

@article{GebKol22,
  author = "B.~Gebhard and J.~Kolumb{\'a}n",
  title  = "On bounded two-dimensional globally dissipative {E}uler flows",
  journal = "SIAM J. Math. Anal.",
  year    = "2022",
  volume  = "54",
  number  = "3",
  pages   = "3457-3479"
}

@article{Kato75, 
  author  = "T.~Kato",
  title   = "The {C}auchy problem for quasi-linear symmetric hyperbolic systems",
  journal = "Arch. Ration. Mech. Anal.",
  year    = "1975",
  volume  = "58",
  number  = "3",
  pages   = "181-205"
}

@article{KliMar18_1,
  author  = "C.~Klingenberg and S.~Markfelder",
  title   = "The {R}iemann problem for the multidimensional isentropic system of gas dynamics is ill-posed if it contains a shock",
  journal = "Arch. Ration. Mech. Anal.",
  year    = "2018",
  volume  = "227",
  number  = "3",
  pages   = "967-994"
}

@book{Majda,
  author  = "A.~Majda",
  title   = "Compressible Fluid Flow and Systems of Conservation Laws in Several Space Variables",
  publisher = "Springer",
  year    = "1984",
  address = "New York",
  series  = "Applied Mathematical Sciences",
  number  = "53"
}

@article{Markfelder24,
  author  = "S.~Markfelder",
  title   = "A new convex integration approach for the compressible {E}uler equations and failure of the local maximal dissipation criterion",
  journal = "Nonlinearity",
  year    = "2024",
  volume  = "37",
  number  = "11",
  pages   = "1-60"
}

@book{Markfelder,
  author  = "S.~Markfelder",
  title   = "Convex Integration Applied to the Multi-Dimensional Compressible Euler Equations",
  publisher = "Springer",
  year    = "2021",
  address = "Cham, Switzerland",
  series  = "Lecture Notes in Mathematics",
  number  = "2294"
}

@article{SzeWie12,
  author  = "L.~{Sz{\'e}kelyhidi~Jr.} and E.~Wiedemann",
  title   = "Young measures generated by ideal incompressible fluid flows",
  journal = "Arch. Ration. Mech. Anal.",
  year    = "2012",
  volume  = "206",
  number  = "1",
  pages   = "333-366"
}

@misc{BouMarTit26pre,
	author  = "D.~W.~Boutros and S.~Markfelder and E.~S.~Titi",
	title   = "A generalized framework for $L^r$ convex integration and its application to geophysical models",
	archivePrefix = "arXiv",
	eprint  = "2606.12192",
	year    = "2026"
}

\end{document}